\documentclass[reqno]{amsart} 
\usepackage{amsmath,amssymb,latexsym,mathrsfs}
\usepackage{indentfirst}
\usepackage{cite}
\usepackage{color}
\usepackage{marginnote}

\usepackage[top=2cm,bottom=2cm,left=2cm,right=2cm]{geometry}
\allowdisplaybreaks[1] 
\numberwithin{equation}{section}
\newtheorem{theorem}{Theorem}[section]

\newtheorem{lemma}[theorem]{Lemma}
\newtheorem{rem}{Remark}[section]
\newtheorem{proposition}[theorem]{Proposition}

\allowdisplaybreaks

\allowdisplaybreaks

\let\e=\varepsilon

\begin{document}



\title[Global Hilbert expansion for the ionic Vlasov-Poisson-Boltzmann system]
{Global Hilbert expansion for the \\ionic Vlasov-Poisson-Boltzmann system}

\author[F.-C. Li]{fucai Li}
\address{School of Mathematics, Nanjing University, Nanjing
 210093, P. R. China}
\email{fli@nju.edu.cn}

\author[Y.-C. Wang]{Yichun Wang}
\address{School of Mathematics,
Southeast University, Nanjing, 211189, P. R. China}
 \email{yichunwang@seu.edu.cn}

\date{}

\begin{abstract}
We justify the global-in-time validity of Hilbert expansion for the ionic Vlasov-Poisson-Boltzmann system in $\mathbb{R}^3$, a fundamental model describing ion dynamics in dilute collisional plasmas. As the Knudsen number approaches zero, we rigorously derive the compressible Euler-Poisson system governing global smooth irrotational ion flows.
The truncated Hilbert expansion exhibits a multi-layered mathematical structure: the expansion coefficients satisfy linear hyperbolic systems, while the remainder equation couples with a nonlinear Poisson equation for the electrostatic potential.  This requires refined elliptic estimates addressing the exponential nonlinearities and some new enclosed $L^2\cap W^{1,\infty}$ estimates for the potential-dependent terms.
\end{abstract}



 \maketitle


\section{Introduction}
\subsection{The model and our results}
The kinetic evolution of ion particles in dilute collisional plasmas is governed by the ionic Vlasov-Poisson-Boltzmann system:
\begin{equation}\label{vpb}
\left\{
\begin{aligned}
&\partial_t F+v\cdot \nabla_x F-\nabla_x \phi \cdot \nabla_v F=\frac{1}{\e}Q(F,F), \\
&\Delta \phi = e^\phi-\rho, \quad \rho=\int_{\mathbb{R}^3}F {\rm d}v,
\end{aligned} \right.
\end{equation}
where $F=F(t,x,v)$ is the ion distribution function with the time $t\geq 0$, position $x=(x^1,x^2,x^3)\in \mathbb{R}^3$, and velocity  $v=(v^1,v^2,v^3)\in  \mathbb{R}^3$.
The self-consistent electrostatic potential $\phi(t,x)$, which couples to $F(t,x,v)$ through the nonlinear Poisson equation in \eqref{vpb}$_2$, incorporates the electron density $e^\phi$ via the well-known Maxwell-Boltzmann law.
The parameter $\e>0$, called Knudsen number, is proportional to the mean free path.
The Boltzmann collision operator $Q(\cdot,\cdot)$ is defined as
\begin{align*}
Q(F_1, F_2)=\;& \iint_{\mathbb{R}^3\times \mathbb{S}^2} |(u-v)\cdot \omega|[F_1(u')F_2(v')-F_1(u)F_2(v)]\mathrm{d}\omega \mathrm{d}u\\
=\;& Q_{\mathrm{gain}}(F_1, F_2)-Q_{\mathrm{loss}}(F_1, F_2),
\end{align*}
where the pre-collision velocities $(u, v)$ and post-collision velocities $(u', v')$ satisfy the momentum and energy conservations: 
$$u+v=u'+v',\quad |u|^2+|v|^2 = |u'|^2+|v'|^2,$$ which implies that
$$u'=u-[(u-v)\cdot \omega]\omega,\quad v'=v+[(u-v)\cdot \omega]\omega,\quad \omega \in \mathbb{S}^2.$$

Formally, Bardos et al. \cite{Bardos2018} derived a system similar to \eqref{vpb} with a time-dependent electron temperature from 
the bipolar Vlasov-Poisson-Boltzmann system under the massless electrons limit (vanishing electron-to-ion mass ratio). This article focuses on the system \eqref{vpb} with normalized electron temperature.
Li-Yang-Zhong \cite{Zhong2016} constructed the existence and uniqueness of the strong solution to \eqref{vpb} in $\mathbb{R}^3$. In \cite{Fucai-iVPB-Torus}, we showed the global existence and exponential decay of classical solutions to \eqref{vpb} in a periodic box.

On the fluid scale, ion dynamics are governed by the compressible ionic Euler-Poisson system:
\begin{equation}\label{EPS}
\left\{ \begin{aligned}
&\partial_t \rho_0+\nabla\cdot (\rho_0 u_0) =0, \\
&\partial_t u_0+(u_0\cdot \nabla)u_0+\frac{1}{\rho_0}\nabla p_0+\nabla \phi_0=0,\\
&\Delta \phi_0= e^{\phi_0}-\rho_0,
\end{aligned}
\right.
\end{equation}
where $\rho_0$ and $u_0$ are the density and velocity of the fluid at the time $t\geq 0$, position $x=(x_1,x_2,x_3)\in \mathbb{R}^3$.  Guo-Pausader \cite{Guo-ion-2011} proved the global existence of smooth solutions $[\rho_0, u_0]$ to \eqref{EPS} near a uniform state $[1, 0]$ with small, irrotational perturbations in $\mathbb{R}^3$. 
Thus, a natural problem arises: can we show the convergence from \eqref{vpb} to \eqref{EPS} as the Knudsen number $\e\rightarrow 0$, and thus contract their connections? 
Our purpose of this paper is to give an affirmative answer to this problem. We shall prove that the strong solutions of the ionic Vlasov-Poisson-Boltzmann system \eqref{vpb} converge to the classical solutions of the ionic Euler-Poisson system \eqref{EPS} globally in time, as $\e\rightarrow 0$. It provides a first rigorous justification for the global validity of the Hilbert expansion in the ionic Vlasov-Poisson-Boltzmann system \eqref{vpb} within the entire space $\mathbb{R}^3$.

We take the truncated Hilbert expansions
\begin{align}
F=&F_0+\e F_1+\cdots +\e^{2k-1}F_{2k-1}+\e^k R,\label{expansion1}\\
\phi=&\phi_0+\e \phi_1+\cdots +\e^{2k-1}\phi_{2k-1}+\e^k \phi_R,\label{expansion2}
\end{align}
where $[F_0,\phi_0]$ is the leading order, $[F_i,\phi_i]$ $(1\leq i\leq 2k-1)$ are the coefficients and $[R,\phi_R]$ is the remainder. 

Plugging  \eqref{expansion1} and \eqref{expansion2} into \eqref{vpb} yields
\begin{equation}\label{plugging}
\left\{ \begin{aligned}
&\partial_t \bigg(\sum^{2k-1}_{n=0}\e^n F_n +\e^k R\bigg)
+v\cdot \nabla_x \bigg(\sum^{2k-1}_{n=0}\e^n F_n +\e^k R\bigg)\\
&\quad -\nabla_x \bigg(\sum^{2k-1}_{n=0}\e^n \phi_n +\e^k \phi_R\bigg)\cdot \nabla_v \bigg(\sum^{2k-1}_{n=0}\e^n F_n +\e^k R\bigg)\\
&\qquad =\frac{1}{\e}Q\bigg(\sum^{2k-1}_{n=0}\e^n F_n +\e^k R, \sum^{2k-1}_{n=0}\e^n F_n +\e^k R\bigg),\\
&\Delta \bigg(\sum^{2k-1}_{n=0}\e^n \phi_n +\e^k \phi_R\bigg)=\exp\bigg(\sum^{2k-1}_{n=0}\e^n \phi_n +\e^k \phi_R\bigg)-\int_{\mathbb{R}^3} \bigg(\sum^{2k-1}_{n=0}\e^n \phi_n +\e^k \phi_R\bigg) \mathrm{d}v.
\end{aligned}
\right.
\end{equation}
We compare the coefficients of different powers of the parameter $\e$ in \eqref{plugging} to get
\begin{equation}\label{Fi}
\begin{split}
\frac{1}{\e}:\,\,& Q(F_0, F_0)=0,\\
\e^0:\,\,& \partial_t F_0+v \cdot \nabla_x F_0- \nabla_x \phi_0 \cdot \nabla_v F_0= Q(F_1,F_0)+Q(F_0,F_1),\\
& \Delta \phi_0=e^{\phi_0}-\int_{\mathbb{R}^3}F_0 \mathrm{d}v,\\
\e^1:\,\,& \partial_t F_1+v \cdot \nabla_x F_1- \nabla_x \phi_0 \cdot \nabla_v F_1= Q(F_2,F_0)+Q(F_0,F_2)+Q(F_1,F_1),\\
& \Delta \phi_1=e^{\phi_0}\phi_1-\int_{\mathbb{R}^3}F_1 \mathrm{d}v,\\
& \cdots \\
\e^n:\,\,& \partial_t F_n+v \cdot \nabla_x F_n- \nabla_x \phi_0 \cdot \nabla_v F_n-\nabla_x \phi_n \cdot \nabla_v F_0\\
&\,\, \quad \,\,\,\,=\sum_{ { i, j\geq 0,i+j=n+1 }}Q(F_i,F_j)+\sum_{ { i, j\geq 1 , i+j=n }}\nabla_x \phi_i \cdot \nabla_v F_j,\\
& \Delta \phi_n=e^{\phi_0}A_n-\int_{\mathbb{R}^3}F_n \mathrm{d}v.
\end{split}
\end{equation}
Here, 
\begin{align*}
A_n=e^{-\phi_0}\frac{1}{n!}\frac{\mathrm{d}^n}{\mathrm{d}\e^n}\Big|_{\e=0}
\left(\exp\big\{\phi_0+\e \phi_1+\cdots +\e^{2k-1}\phi_{2k-1}\big\}\right),\quad n=0,\ldots,2k-1.
\end{align*}
\begin{rem}\label{taylor}
Formally, note that the following difference is of order $\e^{2k}$$:$
\begin{align*}
&\exp\big\{\phi_0+\e \phi_1+\cdots +\e^{2k-1}\phi_{2k-1}+\e^{2k}\phi_{2k}+\cdots\big\}
-\exp\big\{\phi_0+\e \phi_1+\cdots +\e^{2k-1}\phi_{2k-1}\big\}\\
=&\exp\big\{\phi_0+\e \phi_1+\cdots +\e^{2k-1}\phi_{2k-1}\big\} \left(\exp\big\{\e^{2k}\phi_{2k}+\cdots\big\}-1\right)\\
\sim &\,\e^{2k} \phi_{2k} \exp\big\{\phi_0+\e \phi_1+\cdots +\e^{2k-1}\phi_{2k-1}\big\}.
\end{align*}
Hence, taking the $(2k-1)$th-order Taylor polynomial $T_{2k-1}$ for the parameter $\e>0$ yields 
\begin{align*}
&\,T_{2k-1}\big(\exp\big\{\phi_0+\e \phi_1+\cdots +\e^{2k-1}\phi_{2k-1}+\e^{2k}\phi_{2k}+\cdots\big\}\big)\\
=&\,T_{2k-1}\big(\exp\big\{\phi_0+\e \phi_1+\cdots +\e^{2k-1}\phi_{2k-1}\big\}\big)\\
:=&\,T_{2k-1}(H(\e))\\
=&\,H(0)+\e H'(0)+\cdots + \e^{2k-1}\frac{H^{2k-1}(0)}{(2k-1)!}\\
:=&\,e^{\phi_0}(A_0+\e A_1+\cdots +\e^{2k-1}A_{2k-1}),
\end{align*}
where we have denoted $\exp\big\{\phi_0+\e \phi_1+\cdots +\e^{2k-1}\phi_{2k-1}\big\}$ by $H(\e)$. By direct calculations, we get
\begin{align*}
A_0=&\,1,\,\,\,A_1=\phi_1,\,\,\, A_2=\phi_2+\frac{\phi_1^2}{2},\\
A_3=&\,\phi_3+\phi_1\phi_2+\frac{\phi_1^3}{6},\\
 &  \!\!\! \cdots  \\ 
A_n=&\,\phi_n +\phi_{n-1}\phi_1+\phi_{n-2}\frac{\phi_1^2}{2!}+\cdots +\phi_2 \frac{\phi_1^{n-2}}{(n-2)!}
+\frac{\phi_1^n}{n!}+\phi_{n-2}\phi_2+\cdots\\
=&\, \phi_n +\sum_\varkappa C_\varkappa \prod_{1\leq j\leq n-1}(\phi_j)^{\varkappa_j},
 \quad \sum_{ {1\leq j\leq n-1,  \varkappa_j\in \mathbb{N}_+}} j \varkappa_j=n.
\end{align*}
Here, $\varkappa=[\varkappa_1, \ldots, \varkappa_{n-1}]$ and $C_\varkappa>0$.
\end{rem}

From the $\frac{1}{\e}$ step and $\e^0$ step in \eqref{Fi}, the leading order $F_0$ is given by a local Maxwellian $\mu$:
\begin{align}\label{mu}
F_0(t,x,v)=\mu(t,x,v)=\frac{\rho_0(t,x)}{(2\pi \theta_0(t,x))^{\frac{3}{2}}}\exp\Big\{-\frac{|v-u_0(t,x)|^2}{2\theta_0(t,x)}\Big\}, \quad \theta_0(t,x)=K \rho_0^{\frac{2}{3}}(t,x),
\end{align}
for some constant $K>0$, 
where 
\begin{align*}
\int_{\mathbb{R}^3}F_0 \mathrm{d}v=\rho_0,\quad \int_{\mathbb{R}^3}vF_0 \mathrm{d}v=\rho_0u_0,\quad \int_{\mathbb{R}^3}|v|^2F_0 \mathrm{d}v=\rho_0|u_0|^2+3\rho_0 \theta_0.
\end{align*}
Note that $[\rho_0,u_0]$ satisfies the compressible ionic Euler-Poisson system \eqref{EPS} with $p_0=\rho_0\theta_0=K\rho_0^{\frac{5}{3}}$.

We then determine the coefficients $[F_i,\phi_i]$ $(1\leq i\leq 2k-1)$.
Define the linearized Boltzmann operator $L$ and $\Gamma$ around $F_0=\mu$ as
\begin{align}
Lg=&\,-\frac{1}{\sqrt{\mu}}\{Q(\mu, \sqrt{\mu}g)+ Q(\sqrt{\mu}g, \mu)\}= \nu(\mu)g-K_\mu g,\label{L}\\
\Gamma(g_1,g_2)=&\;\frac{1}{\sqrt{\mu}}Q(\sqrt{\mu}g_1, \sqrt{\mu}g_2),
\end{align}
with $\nu(\mu)= \iint_{\mathbb{R}^3\times \mathbb{S}^2}|(u-v)\cdot \omega| \mu(u)\mathrm{d}\omega \mathrm{d}u$.
The null space of $L$, $\mathcal{N}(L)$, is a five dimensional closed subspace of $L^2(\mathbb{R}_v^3)$ whose orthonormal basis is
\begin{align*}
&\chi_0(v)=\frac{1}{\sqrt{\rho_0}}\sqrt{\mu},\\
&\chi_i(v)=\frac{v^i-u_0^i}{\sqrt{\rho_0 \theta_0}}\sqrt{\mu},\quad i=1,2,3,\\
&\chi_4(v)=\frac{1}{\sqrt{6\rho_0}}\Big\{\frac{|v-u_0|^2}{\theta_0}
-3\Big\}\sqrt{\mu}.
\end{align*}
It holds that $\int_{\mathbb{R}^3}\chi_i \chi_j \mathrm{d}v=\delta_{ij}$ for $0 \leq i,j\leq 4$.
We denote $\mathbf{P}$ as the $L^2$ orthogonal projection on $\mathcal{N}(L)$.

By using the property of $L^{-1}:\mathcal{N}^\perp (L)\rightarrow \mathcal{N}^\perp (L)$, we formally derive the expansion coefficients $F_1,\ldots,F_{2k-1}$  based on \eqref{Fi} and $F_0=\mu$. We first define for each $i\geq 1$ that
\begin{align}\label{macro-i}
\mathbf{P}\Big(\frac{F_i}{\sqrt{\mu}}\Big)
\equiv \frac{\rho_i}{\sqrt{\rho_0}}\chi_0+\sum_{j=1}^3 \sqrt{\frac{\rho_0}{\theta_0}}u_i^j\cdot \chi_j+\sqrt{\frac{3\rho_0}{2}}\frac{\theta_i}{\theta_0}\chi_4.
\end{align}
One can easily check that
\begin{align*}
  \int_{\mathbb{R}^3}F_i \mathrm{d}v=&\rho_i,\\
  \int_{\mathbb{R}^3}(v^k-u_0^k)F_i \mathrm{d}v=&\int_{\mathbb{R}^3}\frac{v^k-u_0^k}{\sqrt{\rho_0 \theta_0}}\sqrt{\mu}\sqrt{\rho_0 \theta_0}\frac{F_i}{\sqrt{\mu}}\mathrm{d}v\\
  =&\sqrt{\rho_0 \theta_0}\int_{\mathbb{R}^3}\chi_k \sqrt{\frac{\rho_0}{\theta_0}}u_i^k \chi_k \mathrm{d}v\\
  =& \rho_0 u_i^k,\quad k=1,2,3,\\
  \int_{\mathbb{R}^3} |v-u_0|^2 F_i \mathrm{d}v=&
  \int_{\mathbb{R}^3}\frac{|v-u_0|^2-3\theta_0}
  {\sqrt{6\rho_0}\theta_0}\sqrt{\mu}\sqrt{6\rho_0}\theta_0 \frac{F_i}{\sqrt{\mu}} \mathrm{d}v+3\theta_0 \rho_i\\
  =& \sqrt{6\rho_0}\theta_0\sqrt{\frac{3\rho_0}{2}}
  \frac{\theta_i}{\theta_0}\int_{\mathbb{R}^3}\chi_4^2 \mathrm{d}v+3\theta_0 \rho_i\\
  =& 3\theta_i \rho_0+3\theta_0 \rho_i.
\end{align*}
Projecting the equation of $F_{k+1}$ in \eqref{Fi} onto $\mathcal{N}(L)$, we obtain the equations for $[\rho_{k+1}, u_{k+1}, \theta_{k+1}]$:
\begin{align}
  &\partial_t \rho_{k+1}+\nabla\cdot (\rho_0u_{k+1}+\rho_{k+1}u_0)=0,\label{mass}\\
  &\rho_0 \{\partial_t u_{k+1}+(u_{k+1}\cdot \nabla)u_0+(u_0\cdot \nabla)u_{k+1}+\nabla \phi_{k+1}\}-\frac{\rho_{k+1}}{\rho_0}\nabla (\rho_0\theta_0)
  +\nabla (\rho_0\theta_{k+1}+\theta_0 \rho_{k+1})=\mathfrak{f}_k,\label{moment}\\
  &\rho_0 \{\partial_t \theta_{k+1}+\frac{2}{3}(\theta_{k+1}\nabla\cdot u_0+\theta_0\nabla\cdot u_{k+1})+u_0 \cdot \nabla \theta_{k+1}+u_{k+1}\cdot \nabla \theta_0\}
  =\frac{1}{3}\mathfrak{g}_k,\label{energy}
\end{align}
where 
\begin{align*}
  \mathfrak{f}_k:=&-\partial_j \int_{\mathbb{R}^3} \Big\{(v^i-u_0^i)(v^j-u_0^j)-\delta_{ij}\frac{|v-u_0|^2}{3}\Big\}F_{k+1}\mathrm{d}v
  -\sum_{\substack{i+j=k+1\\ i, j\geq 1}}\rho_j \nabla_x \phi_i,\\
  \mathfrak{g}_k:=&-\partial_i \bigg\{\int (v^i-u_0^i)(|v-u_0|^2-5\theta_0)F_{k+1} \mathrm{d}v+ 2u_0^j \int \Big((v^i-u_0^i)(v^j-u_0^j)-\delta_{ij}\frac{|v-u_0|^2}{3}\Big)F_{k+1}\mathrm{d}v
 \bigg\}\\
& -2u_0\cdot \mathfrak{f}_k- \sum_{\substack{i+j=k+1\\ i, j\geq 1}}(\rho_0 u_j+ \rho_j u_0)\nabla_x \phi_i.
\end{align*}
Here the subscript $k$ of $\mathfrak{f}_k$ and $\mathfrak{g}_k$ is to emphasize that they depends only on $F_i$ and $\phi_i$ for $0 \leq i \leq k$.
Indeed, the microscopic part of $\frac{F_{k+1}}{\sqrt{\mu}}$ is determined through the equation of $F_k$ in \eqref{Fi}:
\begin{align*}
&\{\mathbf{I}-\mathbf{P}\}\Big(\frac{F_{k+1}}{\sqrt{\mu}}\Big) =L^{-1}\Bigg(-\frac{\{\partial_t+v\cdot \nabla_x\}F_k-\sum_{\substack{i+j=k\\ i, j\geq 0}}\nabla_x \phi_i \cdot \nabla_v F_j
-\sum_{\substack{i+j=k+1\\ i, j\geq 1}}Q(F_i, F_j)}{\sqrt{\mu}}\Bigg).
\end{align*}
In Section \ref{sec2}, we shall discuss the existence and regularity of the coefficients $F_i$.

With \eqref{Fi} in hand, by eliminating the equations of $F_0, \ldots, F_{2k-2}$ on both sides of \eqref{plugging}, we obtain the remainder equations for $[R,\phi_R]$ as
\begin{equation}\label{vpbr}
\left\{ \begin{aligned}
&\{\partial_t+ v\cdot \nabla_x-\nabla_x \phi_0 \cdot \nabla_v\}R-\nabla_x \phi_R \cdot \nabla_v F_0-\frac{1}{\varepsilon}\{Q(F_0, R)+Q(R, F_0)\} \\
&\quad =\varepsilon^{k-1} Q(R, R)+ \sum_{i=1}^{2k-1}\e^{i-1}\{Q(F_i, R)+Q(R, F_i)\}+\e^k \nabla_x \phi_R \cdot \nabla_v R  \\
&\quad \quad\; +\sum_{i=1}^{2k-1}\e^i \{\nabla_x \phi_i \cdot \nabla_v R+\nabla_x \phi_R \cdot \nabla_v F_i\}+\e^{k-1}A,\\
 &\Delta (\e^k \phi_R)= \exp\big\{\phi_0+\e \phi_1+\cdots +\e^{2k-1}\phi_{2k-1}+\e^k \phi_R\big\}
 -\int_{\mathbb{R}^3}\e^k R\, \mathrm{d}v\\
&\quad\quad \qquad\quad
-T_{2k-1}\big(\exp\big\{\phi_0+\e \phi_1+\cdots +\e^{2k-1}\phi_{2k-1}\big\}\big),
\end{aligned}\right.
\end{equation}
where $T_{2k-1}(y(\e))$ is the $(2k-1)$th-order Taylor polynomial of the function $y(\e)$, and
\begin{align}\label{A}
A= -\{\partial_t+v\cdot \nabla_x\}F_{2k-1}+ \sum_{\substack{i+j\geq 2k\\1 \leq i, j\leq 2k-1}}\varepsilon^{i+j-2k}Q(F_i, F_j)+\sum_{\substack{i+j\geq 2k-1\\0 \leq i, j\leq 2k-1}}\varepsilon^{i+j-2k+1}\nabla_x \phi_i \cdot \nabla_v F_j.
\end{align}
It is the core of this article to control the remainder $R$ in the nonlinear dynamics through \eqref{vpbr}. To this end, we introduce a global Maxwellian 
\begin{align}\label{muM}
\mu_M(v)=\frac{1}{(2\pi \theta_M)^{\frac{3}{2}}}\exp\Big\{-\frac{|v|^2}{2\theta_M}\Big\},
\end{align}
where the constant $\theta_M$ satisfies 
$\theta_M \leq \inf_{t,x}\theta_0(t,x)\leq  \sup_{t, x}\theta_0(t,x)\leq 2\theta_M$, 
and define
\begin{align}\label{def-f-h}
  R =& \sqrt{\mu} f =\frac{\sqrt{\mu_M}}{w(v)}h,\quad w(v)=(1+|v|^2)^{\beta} \, \,\,\rm{for}\,\,\, \beta\geq \frac{7}{2}.
\end{align}
We shall establish specific estimates for both $f$ and $h$ to ensure the global validity of the Hilbert expansions \eqref{expansion1} and \eqref{expansion2} for the ionic Vlasov-Poisson-Boltzmann system \eqref{vpb}.

$\mathbf{Notations}.$
Throughout this paper, we denote the H\"{o}lder space by $C^{m,\alpha}(\mathbb{R}^3_x)$ for $(m, \alpha)\in \mathbb{N}\times (0,1)$.
For each nonnegative integer $m$ and real number $1\leq p\leq \infty$, we use $W^{m,p}$ to denote the standard nonhomogeneous Sobolev spaces for $(x,v)\in \mathbb{R}_x^3\times \mathbb{R}_v^3$ or $x\in \mathbb{R}_x^3$ and denote $H^s=W^{s,2}$ with the norm $\|\cdot\|_{H^s}$ or simply $\|\cdot\|_s$. In particular, $L^2=H^0$ is equipped with the norm $\|\cdot\|_{L^2}$ or simply $\|\cdot\|_0$. For the collision frequency $\nu(\mu)$ in \eqref{L}, we define a weighted $L^2$ norm $\|\cdot\|_\nu=\|\sqrt{\nu(\mu)}\cdot\|_0$. For notational simplicity, we denote the norm of $L^\infty$ by $\|\cdot\|_{L^\infty}$ or $\|\cdot\|_\infty$.
The bracket $(\cdot ,\cdot )$ is used to denote the $L^2$ inner product in $\mathbb{R}_x^3$ and $\langle \cdot,\cdot\rangle$ denotes the $L^2$ inner product in $\mathbb{R}_x^3\times \mathbb{R}_v^3$. 

Let $P\in \Psi^{m}$ denote a pseudo-differential operator with symbol $ p(x, \xi) \in S^{m}(\mathbb{R}^3 \times \mathbb{R}^3):=S^{m}_{1,0}(\mathbb{R}^3 \times \mathbb{R}^3) $, defined by
$$
Pu(x) = \frac{1}{(2\pi)^3} \int_{\mathbb{R}^3} e^{i x \cdot \xi} p(x, \xi) \mathcal{F}[u](\xi)  \mathrm{d}\xi,
$$
where $\mathcal{F}[u] $ is the Fourier transform of $ u \in \mathcal{S}(\mathbb{R}^d)$ (the Schwartz class), and $ S^{m}_{1,0} $ denotes the Hörmander symbol class of order $m$.
For $s\in \mathbb{N}$, we also  define the Fourier multiplier  operator $\Lambda^s\in \Psi^s$ with the symbol $(1+|\xi|^2)^{\frac{s}{2}}$.
Denote the commutator of $\Lambda^s\in \Psi^s$ and $P\in\Psi^m$ by $[\Lambda^s,P]=\Lambda^s \circ P-P\circ \Lambda^s\in\Psi^{s+m-1}$.

In what follows, we denote
\begin{align}\label{Hee}
H(\e):=H(t,x,\e)=\exp\big\{\phi_0+\e \phi_1+\cdots +\e^{2k-1}\phi_{2k-1}\big\}.
\end{align}
 And, $T_{N}(y(\e))$ denotes the $N$th-order Taylor polynomial of the function $y(\e)$.
We also define the time-dependent functions $\mathcal{I}_1(t,\e)$ and $\mathcal{I}_2(t,\e)$ as:
\begin{align}
&\mathcal{I}_1(t,\e):=\sum_{i=1}^{2k-1}[\e(1+t)]^{i-1}
+\bigg(\sum_{i=1}^{2k-1}[\e(1+t)]^{i-1}\bigg)^2,\label{I_1}\\
&\mathcal{I}_2(t,\e):=\sum_{2k\leq i+j\leq 4k-2}\e^{i+j-2k}(1+t)^{i+j-2},\label{I_2}
\end{align}
which will be used frequently throughout this paper. 

The letter $C$ denotes a generic positive constant independent of the Knudsen number $\e$ that may change in the estimates line by line. In the sequel, we shall emphasis that $C$ is independent of certain variables if necessary. For any time-dependent function $g(t,\cdot)$, we denote the initial datum $g(0,\cdot)$ by $g^{\rm{in}}(\cdot)$. We always use the superscript $t$ to denote the transpose of a row vector.

\smallskip
Now we present our main result as follows.
\begin{theorem}\label{thm1}
Suppose that the functions $[\rho_0(t,x),u_0(t,x),\phi_0(t,x)]$ are solutions to the compressible ionic Euler-Poisson system \eqref{EPS} constructed by \cite{Guo-ion-2011} provided that the smooth initial data $\rho_0^{\rm{in}}(x)-1$ and $u_0^{\rm{in}}(x)$ are sufficiently small and satisfies
$\nabla\times u_0^{\rm{in}}(x)=0$.
Then for the remainder $[R,\phi_R]$ in \eqref{expansion1}--\eqref{expansion2}, there exists an $\e_0>0$ and a constant $C>0$ independent of $\e$ such that for $0<\e<\e_0$,
\begin{equation}\label{ineq-in-thm1.2}
\begin{split}
&\sup_{0\leq t\leq \e^{-m}}\bigg\{\e^{\frac{3}{2}}\Big\|\frac{(1+|v|)^{2\beta+1}R}{\sqrt{\mu_M}}
\Big\|_\infty+\e^{\frac{3}{2}}\|D_x \phi_R\|_\infty\bigg\}\\
&\quad + \sup_{0\leq t\leq \e^{-m}}\bigg\{\e^{5}\Big\|D_{x,v}\Big(\frac{(1+|v|)^{2\beta}R}
{\sqrt{\mu_M}}\Big)
\Big\|_\infty+\e^{5}\|D^2_{xx} \phi_R\|_\infty\bigg\}\\
&\qquad+ \sup_{0\leq t\leq \e^{-m}}\bigg\{\Big\|\frac{R}{\sqrt{\mu}}\Big\|_0+\|\phi_R\|_0
+\|D_x\phi_R\|_0+\|D^2_{xx}\phi_R\|_0\bigg\}\\
&\quad\qquad\leq  C \bigg\{\e^{\frac{3}{2}}\Big\|\frac{(1+|v|)^{2\beta+1}R^{\rm{in}}}{\sqrt{\mu_M}}
\Big\|_\infty
+\e^{5}\Big\|D_{x,v}\Big(\frac{(1+|v|)^{2\beta}R^{\rm{in}}}{\sqrt{\mu_M}}\Big)
\Big\|_\infty+ \Big\|\frac{R^{\rm{in}}}{\sqrt{\mu^{\rm{in}}}}\Big\|_0+1\bigg\}
\end{split}
\end{equation}
for all $0<m<\frac{1}{2}\frac{2k-3}{2k-2}$ and $\beta\geq \frac{7}{2}$.
\end{theorem}
\begin{rem}
In this paper $($including Theorem \ref{thm1}$)$, for the initial data, 
we assume that $\rho^{\rm{in}}_1, \ldots, \rho^{\rm{in}}_{2k-1}\in H^s(\mathbb{R}^3)$,  $u^{\rm{in}}_1, \ldots, u^{\rm{in}}_{2k-1}\in H^s(\mathbb{R}^3)$ and $\theta^{\rm{in}}_1, \ldots, \theta^{\rm{in}}_{2k-1}\in H^s(\mathbb{R}^3)$ for sufficiently large $s\in \mathbb{N}_+$, where $[\rho_j,u_j,\theta_j]$ $(1\leq j\leq 2k-1)$ satisfies \eqref{mass}--\eqref{energy}. These assumptions are used in Section \ref{sec2}.
In addition, we also assume that
\begin{align}\label{IC-assumption}
\e^{\frac{3}{2}}\Big\|\frac{(1+|v|)^{2\beta+1}R^{\rm{in}}}{\sqrt{\mu_M}}
\Big\|_\infty
+\e^{5}\Big\|D_{x,v}\Big(\frac{(1+|v|)^{2\beta}R^{\rm{in}}}{\sqrt{\mu_M}}\Big)
\Big\|_\infty+\Big\|\frac{R^{\rm{in}}}{\sqrt{\mu^{\rm{in}}}}\Big\|_0<\infty.
\end{align} 
\eqref{IC-assumption} is used to close the main assumptions throughout our paper that
$\|f\|_{L^2}+\|\e^{\frac{3}{2}}h\|_{L^\infty}+\|\e^5 h\|_{W^{1,\infty}}<\infty$, where $f$ and $h$ are defined via \eqref{def-f-h}. 
\end{rem}
\begin{rem}
We see from Theorem \ref{thm1} that the $L^2$ norm of the remainder $R$ is of $O(1)$, the weighted $L^\infty$ norm of the remainder $R$ is of $O(\e^{-\frac{3}{2}})$ and the weighted $L^\infty$ norm of $D_{x,v}R$ is of $O(\e^{-5})$. This is obtained from \eqref{order-5}, which controls the $L^\infty$ norm of $D_{x,v}R$ by $O(\e^{-3})\|R\|_\infty+O(\e^{-4})\|R\|_0$ and other nonsingular terms. If we multiply both sides of \eqref{order-5} by $\e^4$, then $O(\e)\|R\|_\infty$ appears on the right-hand side of \eqref{order-5}, which provides another singular term $O(\e^{-\frac{1}{2}})\|R\|_0$ by the estimate \eqref{order-3/2}. To eliminate this singularity, we should multiply \eqref{order-5} by at least $\e^5$ to get the estimates in Theorem \ref{thm1}. Moreover, to enclose the smallness assumption \eqref{assume}, we choose the expansion parameter $k\geq 6$. In this case, we get
\begin{align*}
  \sup_{0\leq t\leq \e^{-m}}\|F-\mu\|_0=&\sup_{0\leq t\leq \e^{-m}}\bigg\|\sum_{i=1}^{2k-1}\e^i F_i+\e^k R\bigg\|_0\\
  \leq & \sup_{0\leq t\leq \e^{-m}}\bigg\{\sum_{i=1}^{2k-1}\e^i\Big(\Big\|\sqrt{\mu}\mathbf{P}
  \Big(\frac{F_i}{\sqrt{\mu}}\Big)\Big\|_0+\Big\|\sqrt{\mu}
  (\mathbf{I}-\mathbf{P})
  \Big(\frac{F_i}{\sqrt{\mu}}\Big)\Big\|_0\Big)+\e^k \Big\|\sqrt{\mu}\frac{R}{\sqrt{\mu}}\Big\|_0\bigg\}\\
  =& O(\e)
\end{align*}
for all $0<m\leq \frac{1}{2}\frac{2k-3}{2k-2}$, which leads to the compressible ionic Euler-Poisson limit. The restriction of $m$ is due to a technical requirement that the temporal growth of the coefficient term $A$ in \eqref{A} can be controlled uniformly in $\e$ along the dynamics of the remainder system \eqref{vpbr} as the estimate \eqref{time}.
\end{rem}
\begin{rem}
From \eqref{ineq-in-thm1.2} in Theorem \ref{thm1}, we get the quantitative estimates of $\phi_R$ who solves the nonlinear Poisson equation \eqref{vpbr}$_2$. In Section \ref{Sec-phi_R}, we first obtain the uniform boundedness of $\|\e^k \phi_R\|_{L^\infty}$, then, we deduce some refined elliptic estimates as \eqref{psiR-e}, \eqref{DphiR-Linfty} and \eqref{DDphiR-Linfty}.
Roughly speaking, \eqref{psiR-e} implies that:
$$\|\phi_R\|_{L^\infty}\leq C \|\phi_R\|_{H^2}\leq C (\e^{\frac{1}{2}}+\|f\|_{L^2}),$$ 
\eqref{DphiR-Linfty} implies that:
$$\|D_x \phi_R\|_{L^\infty}\leq C\|\phi_R\|_{W^{2,p}}\leq C (\e^{\frac{1}{2}}+\|h\|_{L^\infty}),\quad p>3,$$
\eqref{DDphiR-Linfty} implies that:
$$\|D^2_{xx} \phi_R\|_{L^\infty}\leq C\|\phi_R\|_{C^{2,\alpha}}\leq C (\e^{\frac{1}{2}}+\|h\|_{L^\infty}+\|D_x h\|_{L^\infty}),\quad 0<\alpha<1.$$
Hence $\|\phi_R\|_\infty+\e^{\frac{3}{2}}\|D_x \phi_R\|_\infty+\e^5\|D^2_{xx}\phi_R\|_\infty$ is bounded uniformly in $\e$.
Here, the term $\e^{\frac{1}{2}}$ on the right-hand side of the above inequalities is provided by the coefficient term 
$$\exp\big\{\phi_0+\e \phi_1+\cdots +\e^{2k-1}\phi_{2k-1}\big\}-T_{2k-1}\big(\exp\big\{\phi_0+\e \phi_1+\cdots +\e^{2k-1}\phi_{2k-1}\big\}\big):=H(\e)-T_{2k-1}(H(\e))$$ 
in equation \eqref{vpbr}$_2$. Actually, in Proposition \ref{prop2} we carefully deduce for $t\leq \e^{-m}$, $0<m<\frac{1}{2}$ that
$$\e^{-k}\|H(\e)-T_{2k-1}(H(\e))\|_{L^2}\leq C \e^{\frac{1}{2}},\quad 
\e^{-k}\|H(\e)-T_{2k-1}(H(\e))\|_{W^{1,\infty}}\leq C \e^{\frac{1}{2}},$$
which are essentially due to the fact that $\phi_i$ $(1\leq i\leq 2k-1)$ possesses a polynomial temporal growth in high order $H^s$ norm. 

On the other hand, 
from the remainder equations \eqref{vpbr}, we directly conclude a  closed $L^2$ estimate \eqref{f-DphiR-L2}, involving $\|f\|_{L^2}$ and $\|\phi_R\|_{H^1}$. It should be pointed out that the estimate of the extra term $\|D^2_{xx}\phi_R\|_{L^2}$ in Theorem \ref{thm1} is obtained by elliptic estimates.  
The main reason of the validity of the elliptic estimates is that the 
nonlinear Poisson equation \eqref{vpbr}$_2$ can be regarded as
a linear elliptic equation once we get $\phi_R\in H^2(\mathbb{R}^3)\subset L^\infty(\mathbb{R}^3)$, which provides a new positive zero order term enhancing the coercivity.
\end{rem}

\subsection{Previous works}
\subsubsection{Hilbert expansion}
The rigorous mathematical justification of Hilbert expansions in kinetic theory has witnessed rapid advancements in recent decades. For the classical Boltzmann equation, Caflisch \cite{Caflisch} pioneered the use of truncated Hilbert expansions coupled with remainder decomposition to establish the well-known compressible Euler limit. Building upon this framework, Lachowicz \cite{Lachowicz,Lachowicz-2} incorporated initial layer corrections to derive Euler limits for the Boltzmann, Enskog, and Povzner equations. However, the solutions constructed through these early approaches around Euler equations lacked guaranteed nonnegativity, which is a critical physical requirement.

This limitation was overcome by Guo-Jang-Jiang \cite{krm,cpam} through their innovative $L^2$--$L^\infty$ method, originally developed by Guo \cite{Guo2010} for the well-posedness of Boltzmann equation in bounded domains. 
Their approach is then proved to be robust for multiple kinetic models: such as the Vlasov-Poisson-Boltzmann system for electrons by Guo-Jang \cite{Juhi} with hard-sphere case, then extended to the soft potentials case in \cite{Fucai-siam} by ourselves, 
the relativistic Boltzmann equation by Speck-Strain \cite{Speck}, 
the relativistic Vlasov-Maxwell-Boltzmann system by Guo-Xiao \cite{Guo-xiao},
the two-species Boltzmann equations by Wu-Yang \cite{WuTianfang-JDE}, the quantum Vlasov-Poisson-Boltzmann system by Jiang and Zou \cite{jiang-zou-jde}, 
and the Vlasov-Maxwell-Boltzmann system to the  incompressible Navier-Stokes-Fourier-Maxwell system with Ohm’s
law recently by Jiang-Luo-Zhang \cite{jiangningARMA2023}, see also Zhou et al. \cite{wzgx-2025}.
For the Landau type equations and non-cutoff Boltzmann type equations, Ouyang-Wu-Xiao \cite{Wulei-QAM} proved the global-in-time convergence of the Hilbert expansion in  the relativistic Vlasov-Maxwell-Landau system. Moreover, Lei-Liu-Xiao-Zhao \cite{Leiyuanjie-JLMS} settled the more complicated cases, including the non-relativistic Vlasov-Maxwell-Landau system and non-cutoff Vlasov-Maxwell-Boltzmann system.

For the initial boundary problems in half space,  
Guo-Huang-Wang \cite{wangyong} justified the validity of Hilbert expansion to the Boltzmann equation for specular reflection boundary condition, and Ouyang-Wang \cite{Ouyangjing-JDE} recently extended this argument to the soft potentials case. Recently, Jiang-Luo-Tang successfully worked out for the Maxwell reflection boundary condition case \cite{jiangning-arxiv} and the completely diffuse boundary condition case \cite{jiangningTAMS}. 

\subsubsection{Kinetic models for ions}
Extensive research has been conducted on the global well-posedness of the electronic Vlasov-Poisson-Boltzmann system. Notably, Guo \cite{Guo2002} established classical solutions in $\mathbb{T}^3$, while Yang-Zhao \cite{zhao2006} and Yang-Yu-Zhao \cite{Hongjun} addressed the $\mathbb{R}^3$ case. For the optimal temporal decay rates of strong solutions in $\mathbb{R}^3$, see Duan-Strain \cite{Duan2011} and Li-Yang-Zhong \cite{Zhong2021}. Well-posedness for hard and soft potentials in 
$\mathbb{R}^3$ was further developed by Duan-Yang-Zhao \cite{zhao2012, zhao2013} and Xiao-Xiong-Zhao \cite{huijiang2017}.
For a domain with boundary, interested reader can refer to, for example,  \cite{kim,Fucai-JDE2021,cdl}.

In stark contrast, mathematical studies on kinetic models for ions remain comparatively limited due to the exponential nonlinearities introduced by the Poisson equation. For the ionic Vlasov-Poisson system, Bouchut \cite{Bouchut} first proved the global weak solutions in $\mathbb{R}^3$. Subsequent work by Bardos et al. \cite{Bardos2018} established the existence and uniqueness of weak solutions in a 2D bounded domain  and a 3D torus. Griffin-Iacobelli \cite{Iacobelli2021a, Iacobelli2021b} and Cesbron-Iacobelli \cite{Iacobelli2023} later derived classical solutions for the 3D torus, whole space $\mathbb{R}^3$, and bounded domains, respectively. Regarding the ionic and bipolar Vlasov-Poisson-Boltzmann systems in $\mathbb{R}^3$, Li-Yang-Zhong \cite{Zhong2016} established global well-posedness and optimal decay rates for strong solutions via spectral analysis.
 Flynn and  Guo \cite{fy-24} studied the  massless electron limit of the bipolar Vlasov-Poisson-Landau system
 and obtain the ionic Vlasov-Poisson-Landau system. Flynn \cite{flynn-25} established the  
local well-posedness of the ionic Vlasov-Poisson-Landau system.

\subsubsection{Compressible ionic Euler-Poisson system}
Previous studies on the ionic Euler-Poisson system include \cite{grenier} for one dimensional case by Cordier and Grenier, who investigated the quasi-neutral limit to 1D compressible Euler system, as letting the Debye length for ion flows
tend to zero. Later, G\'{e}rard, Han-Kwan and Rousset  \cite{Gerard-2013, Gerard-2014} studied the quasi-neutral limit for the ionic Euler-Poisson system over the half space with various boundary conditions. Jung-Kwon-Suzuki \cite{Suzuki-2016-M3AS} discussed quasi-neutral limit in the presence of plasma sheaths.
For the asymptotic stability of plasma sheaths of the compressible ionic Euler-Poisson system, one refers to Duan-Yin-Zhu \cite{YinHaiyan2021} and Yao-Yin-Zhu \cite{YinHaiyan-2025}.
Meanwhile,  Guo-Pausader \cite{Guo-ion-2011} established the global well-posedness of the 3D compressible ionic Euler-Poisson system in $\mathbb{R}^3$ with irrotational initial velocity.

However, the compressible ionic Euler-Poisson limit of the ionic Vlasov-Poisson-Boltzmann system has yet to be mathematically validated.
Given that the ionic Vlasov-Poisson-Boltzmann system serves as a fundamental kinetic model for ion-scale phenomena in diluted plasmas, our work focuses on establishing this limit rigorously.  Our analysis aims to bridge the kinetic and fluid descriptions of ion dynamics by rigorously deriving the compressible ionic Euler-Poisson system as the leading-order asymptotic of the ionic Vlasov-Poisson-Boltzmann system.

\subsection{Main difficulties and our strategies in the proof}
\subsubsection{Formations of the model}
Comparing the Hilbert expansion for the electronic Vlasov-Poisson-Boltzmann system issued by Guo-Jang \cite{Juhi} to our ionic Vlasov-Poisson-Boltzmann system, the main difference lies in the nonlinear Poisson equation with the new exponential term $e^{\phi}$.  Formally, 
we take Taylor's expansion of  
$$\exp(\phi)=\exp\bigg(\sum_{i=0}^\infty\e^{i}\phi_i\bigg)$$
in $\e$ around $\e=0$, then the coefficients of different orders are determined, namely, 
\begin{gather*}
\exp\bigg(\sum_{i=0}^\infty\e^{i}\phi_i\bigg)=\tilde{A}_0+\e \tilde{A}_1+\e^2 \tilde{A}_2+\cdots,\\
\tilde{A}_n=\frac{\mathrm{d}^n}{\mathrm{d}\e^n}\Big|_{\e=0}
\bigg\{\exp\bigg(\sum_{i=0}^\infty\e^{i}\phi_i\bigg)\bigg\}
\triangleq A_n\exp(\phi_0), \,\,\,n\geq 0.
\end{gather*}
It is important to notice that, for any $N\in \mathbb{N}_+$,
\begin{align*}
A_n=&\exp(-\phi_0)\tilde{A}_n=\exp(-\phi_0)
\frac{\mathrm{d}^n}{\mathrm{d}\e^n}\Big|_{\e=0}
\bigg\{\exp\bigg(\sum_{i=0}^N\e^{i}\phi_i\bigg)\bigg\},
\,\,\,0\leq n\leq N,
\end{align*}
which implies that
\begin{align*}
\exp(\phi_0) (A_0+\e A_1+\cdots +\e^N A_N)=T_N \bigg(\exp\bigg(\sum_{i=0}^N\e^{i}\phi_i\bigg)\bigg),
\end{align*}
where $T_N(g(\e))$ denotes the $N$-th order Taylor polynomial of $g(\e)$ for the parameter $\e>0$.
Expanding the nonlinear Poisson equation \eqref{vpb}$_2$ with the Hilbert expansion
$$\Delta \bigg(\sum_{i=0}^\infty\e^{i}\phi_i\bigg)
=\exp\bigg(\sum_{i=0}^\infty\e^{i}\phi_i\bigg)
-\int_{\mathbb{R}^3}\bigg(\sum_{i=0}^\infty\e^{i}F_i\bigg)\mathrm{d}v$$
up to the $N$-th order yields
$$\Delta \bigg(\sum_{i=0}^N\e^{i}\phi_i\bigg)
=T_N \bigg(\exp\bigg(\sum_{i=0}^N\e^{i}\phi_i\bigg)\bigg)
-\int_{\mathbb{R}^3}\bigg(\sum_{i=0}^N\e^{i}F_i\bigg)\mathrm{d}v.$$
Taking the difference of the above equation with $N=2k-1$ and the nonlinear Poisson equation \eqref{plugging} with the truncated Hilbert expansions \eqref{expansion1} and \eqref{expansion2}, it remains
$$\Delta (\e^k \phi_R)
=\exp\bigg(\sum_{i=0}^{2k-1}\e^{i}\phi_i\bigg)\exp(\e^k \phi_R)-T_{2k-1} \bigg(\exp\bigg(\sum_{i=0}^{2k-1}\e^{i}\phi_i\bigg)\bigg)
-\int_{\mathbb{R}^3}\e^k R\,\mathrm{d}v,$$
which is exactly the remainder equation \eqref{vpbr}$_2$ for the unknown $\phi_R$ .
\subsubsection{$H^s$ estimates of the coefficients $\phi_n$ $(1\leq n\leq 2k-1)$}
To explore the well-posedness of equation \eqref{vpbr}$_2$, we deduce the energy estimates of the coefficients $\phi_n$ $(1\leq n\leq 2k-1)$, which solves 
\begin{align*}
\Delta \phi_n=A_n\exp(\phi_0)-\rho_n
\end{align*}
coupling with the system \eqref{mass}--\eqref{energy} for $[\rho_n,u_n,\theta_n]$, where 
\begin{align*}
A_n
\equiv& \phi_n +\sum_\varkappa C_\varkappa \prod_{1\leq j\leq n-1}(\phi_j)^{\varkappa_j}, 
\quad  \sum_{ {1\leq j\leq n-1, \varkappa_j\in \mathbb{N}_+}} j \varkappa_j=n,
\end{align*}
with $\varkappa=[\varkappa_1, \ldots, \varkappa_{n-1}]$ and the constants $C_\varkappa>0$. When $[\rho_i,u_i,\theta_i,\phi_i]$ $(1\leq i\leq n-1)$ are determined, $\phi_n$ satisfies a linear 
elliptic equation  and $U_n=[\rho_n,u_n,\theta_n]^t$, $V_n=[0,\nabla \phi_n,0]^t$ satisfy a linear strictly hyperbolic system
\begin{align*}
 \mathcal{A}_0\{\partial_t U_n+V_n\}+\sum_{i=1}^3 \mathcal{A}_i\partial_i U_n+\mathcal{B} U_n=\mathcal{G}_n, 
\end{align*}
as in \eqref{linear-symmetric-system}. In the $H^s$ framework, 
it is crucial to observe that 
the inner product of $\mathcal{A}_0 V_n$ and $U_n$ could recover 
an extra term 
$$\frac{\mathrm{d}}{\mathrm{d}t}\big(\|e^{\frac{\phi_0}{2}}\phi_n\|^2_{H^s}+\|\nabla \phi_n\|^2_{H^s}\big)$$ 
on the left-hand side through the linear elliptic equation of $\phi_n$, rather than
simply a rough upper bound $$\|\mathcal{A}_0\|_{W^{s,\infty}}\|U_n\|_{H^s}\|V_n\|_{H^s},$$
which is lack of sufficient temporal decays like $t^{-16/15}$.

Specifically, invoking the high-order boundedness of the smooth function $\phi_0$, we prove that there exists a $\Psi^{-2}$ pseudo-differential operator $(e^{\phi_0}-\Delta)^{-1}$ as the parametrix of the elliptic operator $e^{\phi_0}-\Delta$, which is used to get the necessary estimates of $\partial_t \phi_0$. The estimates of commutators between the Fourier multiplier $\Lambda^s$ and other pseudo-differential operators are standard for $s\in \mathbb{N}_+$, which are used to conclude that
\begin{align*}
\frac{\mathrm{d}}{\mathrm{d}t}(\|U_n\|^2_{H^s}+\|V_n\|^2_{H^s}+\|\phi_n\|^2_{H^s})
\leq &\,\,C(1+t)^{-\frac{16}{15}}
(\|U_n\|^2_{H^s}+\|V_n\|^2_{H^s}+\|\phi_n\|^2_{H^s}) \\
&+ C(1+t)^{n-2}(\|U_n\|_{H^s}+\|V_n\|_{H^s}+\|\phi_n\|_{H^s}).
\end{align*}  
It follows from Gronwall's inequality that
$$\|U_n\|_{H^s}+\|V_n\|_{H^s}+\|\phi_n\|_{H^s}\leq C (1+t)^{n-1}$$ for $1\leq n\leq 2k-1$, $s\geq 0$, which is the main result of Proposition \ref{propa}.

\subsubsection{The estimates of $H(\e)-T_{2k-1}(H(\e))$ and its first-order derivatives}
We see from the above deductions that $\|\e^k \phi_R\|_{W^{2,6}_x}$ has a uniform-in-$\e$ upper bound. Furthermore, we can prove that $\| \phi_R\|_{L^\infty}$ has a uniform-in-$\e$ upper bound, which implies the smallness of the size of $\e^k \phi_R$. Actually, the key point for the proof is to observe the size of $H(\e)-T_{2k-1}(H(\e))$, where $H(\e)$ is defined in \eqref{Hee}. Formally, by the definition of $T_{2k-1}$,  $H(\e)-T_{2k-1}(H(\e))$ seems to be of $\e^{2k}$ order, but the $L^2$ or $L^\infty$ norms of $\phi_1, \ldots, \phi_{2k-1}$ provide polynomial growths in time, which would cancel some orders of $\e$. Hence we should investigate the Lagrange remainders of $H(\e)-T_{2k-1}(H(\e))$: for each $(t,x)$, there exists a $\vartheta_{t,x}\in (0,1)$ such that
\begin{align*}
H(\e)-T_{2k-1}(H(\e))= \frac{\e^{2k}}{(2k)!}H^{(2k)}(\vartheta_{t,x} \e) 
=  \frac{\e^{2k}}{(2k)!}H(\vartheta_{t,x} \e) \sum_\varpi C_\varpi \prod_{1\leq \ell \leq 2k-1}(g^{(\ell)}(\vartheta_{t,x} \e))^{\varpi_\ell},
\end{align*}
which is reminiscent of the Fa\`{a} di Bruno's formula,
where $\sum_{1\leq \ell\leq 2k-1}\ell \varpi_\ell=2k$ and 
\begin{align*}
g^{(\ell)}(\vartheta_{t,x} \e)=\ell!\phi_\ell+(\ell+1)!(\vartheta_{t,x} \e)^1\phi_{\ell+1}+\cdots + \frac{(2k-1)!}{(2k-\ell-1)!}(\vartheta_{t,x} \e)^{2k-\ell-1}\phi_{2k-1},\quad 0\leq \ell\leq 2k-1.
\end{align*}
Due to the fact that $\|\phi_{\ell+1}\|_{H^s}\leq C (1+t)^{\ell}$, then for $t\leq \e^{-m}$ $(0<m<\frac{1}{2})$, the $L^2$ or $L^\infty$ norms of $g^{(\ell)}(\vartheta_{t,x} \e)$ absorbs $\e^{\frac{\ell-1}{2}}$ factor and releases $\e^{(\frac{1}{2}-m)(\ell-1)}$ factor as shown in \eqref{g-ell-L2} and \eqref{g-ell-Linfty}. It follows that the $L^2$ or $L^\infty$ norms of $H(\e)-T_{2k-1}(H(\e))$ contributes $O(\e^{k+\frac{1}{2}})$. With this estimate in hand, we can find a constant $C>0$ independent of $\e$ and $\phi_R$ such that
$\|\phi_R\|_{H^2}\leq C (\e^{\frac{1}{2}}+\|f\|_{L^2})$, which is crucial for getting the $L^2$ estimate of $f$ in Section \ref{sec3}. We also consider the $L^2$ norm of $\partial_t (H(\e)-T_{2k-1}(H(\e)))$ and $L^\infty$ norm of $\nabla_x (H(\e)-T_{2k-1}(H(\e)))$, which are used to get refined estimates about $\phi_R$. The expressions of $\partial_t (H(\e)-T_{2k-1}(H(\e)))$ and $\nabla_x (H(\e)-T_{2k-1}(H(\e)))$ is explicit, noticing that $\partial_t T_{2k-1}(H(\e))=T_{2k-1}(\partial_t H(\e))$ and $\nabla_x T_{2k-1}(H(\e))=T_{2k-1}(\nabla_x H(\e))$, which rely on the continuity of mixed partial derivatives $D_t D_\e H(\e)$ and $D_x D_\e H(\e)$ proved in Section \ref{Sec-phi_R}.
\subsubsection{The interplay between $L^2$ and $W^{1,\infty}$ norms}
In Sections \ref{sec3} and \ref{sec4}, we do not repeat the details of those estimates similar to Guo-Jang \cite{Juhi} for electronic Vlasov-Poisson-Boltzmann system, and we emphasis here the differences raised from the ion dynamics. 
Note that the $t-$derivative of the nonlinear Poisson equation 
$$\Delta \psi_R=  (e^{\psi_R}-1)H(\e)+H(\e)-T_{2k-1}(H(\e))-\int_{\mathbb{R}^3}\e^k R\, \mathrm{d}v$$
provides a new energy structure $\frac{\mathrm{d}}{\mathrm{d}t}\int_{\mathbb{R}^3} (\psi_R e^{\psi_R}-e^{\psi_R}+1)H(\e)\mathrm{d}x$, where $\psi_R e^{\psi_R}-e^{\psi_R}+1$ is nonnegative. Moreover, when $|\psi_R|\ll 1$, it holds
$\psi_R e^{\psi_R}-e^{\psi_R}+1\sim \psi_R^2$, thus in essence the $L^2$ energy of the ionic Vlasov-Poisson-Boltzmann system involves 
$\|f\|_{L^2}$, $\|\nabla \phi_R\|_{L^2}$ and $\|\phi_R\|_{L^2}$.

To close the $L^2$ estimate \eqref{low}, the weighted $W^{1,\infty}$ estimates are needed. At first, Lemma \ref{1} (similar to Guo-Jang \cite{Juhi} for electrons) is revisited for the electrostatic potential for ions. Next, we use the method of characteristic to estimate $\|h\|_{L^\infty}$ through the expression \eqref{h-formula}. Note that, on the seventh line of \eqref{h-formula}, there is a source involving the potential $\nabla_x \phi_R$ for ions. We should recover $\|h\|_{L^\infty}$ from $\|\nabla_x \phi_R\|_{L^\infty}$, which requires new elliptic estimates. Our strategy lies in $(H(\e)-\Delta)^{-1}\in \Psi^{-2}: L^p\rightarrow W^{2,p}$, $p>3$, Morrey's embedding and $\|H(\e)-T_{2k-1}(H(\e))\|_{L^p}\leq C \e^k\e^{\frac{1}{2}}$ to deduce $\|\nabla_x \phi_R\|_{L^\infty}\leq C \|\phi_R\|_{W^{2,p}}\leq C \e^{\frac{1}{2}}+C\|h\|_{L^\infty}$. It contributes an $\e^{\frac{3}{2}}\e^{1}O(\e^{\frac{1}{2}})=O(\e^3)$ term  in the final $L^\infty$ estimate \eqref{order-3/2}, which proves Lemma \ref{lemma-h-Linfty}.

For the $L^\infty$ norm of $D_{x,v}h$, we need to recover $\|h\|_{W^{1,\infty}}$ from $D_xD_x \phi_R$. By using 
$(H(\e)-\Delta)^{-1}\in \Psi^{-2}:C^{0,\alpha}\rightarrow C^{2,\alpha}$ for $0<\alpha<1$, $C^{0,\alpha}\subset W^{1,\infty}$ and $\|H(\e)-T_{2k-1}(H(\e))\|_{W^{1,\infty}}\leq C \e^k\e^{\frac{1}{2}}$, we obtain that
$\|D_xD_x\phi_R\|_{L^\infty}\leq C \e^{\frac{1}{2}}+C\|h\|_{W^{1,\infty}}$. It contributes an $O(\e^{\frac{13}{2}})=\e^{5}\e^{1}O(\e^{\frac{1}{2}})$ term. Notice however that, in the final $W^{1,\infty}$ estimate \eqref{order-5}, the term $\e^2 \|h\|_{L^\infty}$ on the right-hand side provides a lowest order term $\e^{\frac{1}{2}}O(\e^3)=O(\e^{\frac{7}{2}})$, which is presented in Lemma \ref{lemma-Dh-Linfty}.
Based on these fundamental estimates, the next step is to combine the $L^\infty$ and $W^{1,\infty}$ norms, which can be carried out in the same lines of Guo-Jang \cite{Juhi}.

\subsection{Outline of the paper}
This paper is organized as follows. In Section \ref{sec2}, we estimate the temporal growth of the coefficients $\phi_i$ $(1\leq i\leq 2k-1)$ within $H^s$ norms. 
Section \ref{Sec-phi_R} is devoted to studying the elliptic estimates of the remainder $\phi_R$ of the nonlinear Poisson equation by controlling its source terms within $L^2\cap W^{1,\infty}$ norm. 
In Section \ref{sec3} and Section \ref{sec4}, we respectively establish the $L^2$ and weighted $W^{1,\infty}$ estimates of $R$ via the remainder system \eqref{vpbr}. 
Section \ref{sec5} encloses the $L^2$--$L^\infty$ estimates
and complete the proof of Theorem \ref{thm1}.

\section{Coefficients of the Hilbert Expansion}\label{sec2}
For each given integer $k\geq 1$, let $[\rho_k,u_k,\theta_k]$ satisfy the system 
\eqref{mass}--\eqref{energy}. We denote 
\begin{align}\label{Uk-Vk}
  U_k=\begin{bmatrix}
        \rho_k \\
        (u_k)^t \\
         \theta_k 
      \end{bmatrix}, \quad V_k=\begin{bmatrix}
                                 0 \\
                                 (\nabla \phi_k)^t \\
                                 0 
                               \end{bmatrix},
\end{align}
then \eqref{mass}--\eqref{energy} are equivalent to 
\begin{align}\label{linear-system}
 \partial_t U_k+V_k+\sum_{i=1}^3 \overline{\mathcal{A}}_i\partial_i U_k+\overline{\mathcal{B}}U_k=\overline{\mathcal{G}}_k, 
\end{align}
where 
\begin{align*}
  \overline{\mathcal{A}}_i=\begin{bmatrix}
                   u_0^i & \rho_0 e_i & 0 \\
                   \frac{\theta_0}{\rho_0}(e_i)^t & u_0^i I_{3\times 3} & (e_i)^t \\
                   0 & \frac{2}{3}\theta_0 e_i & u_0^i 
                 \end{bmatrix},\quad
  \overline{\mathcal{B}}=\begin{bmatrix}
                   \nabla\cdot u_0 & \nabla\rho_0  & 0 \\
                   \frac{-\theta_0}{\rho^2_0}(\nabla \rho_0)^t & \nabla (u_0)^t & (\nabla \ln \rho_0)^t \\
                   0 & \nabla \theta_0 & \frac{2}{3}\nabla\cdot u_0 
                 \end{bmatrix},\quad
  \overline{\mathcal{G}}_k=\begin{bmatrix}
                 0 \\
                 \frac{1}{\rho_0}\mathfrak{f}_{k-1} \\
                 \frac{1}{3\rho_0}\mathfrak{g}_{k-1} 
               \end{bmatrix}.
\end{align*}
It is crucial to observe that $\overline{\mathcal{A}}_i$ is symmetrizable. By introducing  
\begin{align*}
  \mathcal{A}_0=\begin{bmatrix}
                   \frac{\theta_0}{\rho_0} \\
                   & \rho_0 I_{3\times 3} \\
                   && \frac{3\rho_0}{2\theta_0}
                 \end{bmatrix},
\end{align*}
\eqref{linear-system} is equivalent to
\begin{align}\label{linear-symmetric-system}
 \mathcal{A}_0\{\partial_t U_k+V_k\}+\sum_{i=1}^3 \mathcal{A}_i\partial_i U_k+\mathcal{B} U_k=\mathcal{G}_k, 
\end{align}
where
\begin{align*}
  \mathcal{A}_i=\begin{bmatrix}
                   \frac{\theta_0}{\rho_0}u_0^i & \theta_0 e_i & 0 \\
                   \theta_0(e_i)^t & \rho_0 u_0^i I_{3\times 3} & \rho_0(e_i)^t \\
                   0 & \rho_0 e_i & \frac{3\rho_0}{2\theta_0}u_0^i 
                 \end{bmatrix},\quad
  \mathcal{B}=\begin{bmatrix}
                   \frac{\theta_0}{\rho_0}\nabla\cdot u_0 & \frac{\theta_0}{\rho_0}\nabla\rho_0  & 0 \\
                   -\theta_0(\nabla \ln \rho_0)^t & \rho_0 \nabla (u_0)^t & (\nabla  \rho_0)^t \\
                   0 & \frac{3\rho_0}{2\theta_0}\nabla\theta_0 & \frac{\rho_0}{\theta_0}\nabla\cdot u_0 
                 \end{bmatrix},\quad
  \mathcal{G}_k=\begin{bmatrix}
                 0 \\
                 \mathfrak{f}_{k-1} \\
                 \frac{1}{2\theta_0}\mathfrak{g}_{k-1} 
               \end{bmatrix}.
\end{align*}
Here $e_i$, $i=1,2,3$, are the unit row base vectors in $\mathbb{R}^3$.

The well-posedness of \eqref{linear-symmetric-system} follows from the standard linear theory (see e.g. \cite{Majda}). Below we shall establish more delicate energy estimates of \eqref{linear-symmetric-system} to control the temporal growth of the $H^s$ norm $\|U_k(t)\|_s+\|V_k(t)\|_s+\|\phi_k(t)\|_s$, which will be used in the arguments in Sections \ref{Sec-phi_R} and \ref{sec3}.  

\begin{proposition}\label{propa}
For the integer $k\geq 1$, assume that $U_k$ and $V_k$ in \eqref{Uk-Vk} solve the system \eqref{linear-symmetric-system}, then it holds
\begin{align*}
  \|U_k(t)\|_s+\|V_k(t)\|_s+\|\phi_k(t)\|_s\leq C (1+t)^{k-1}\quad \mathrm{for}\,\,\,s\geq 0,
\end{align*}
where $C>0$ is independent of $t$.
\end{proposition}
\begin{proof}
We first consider \eqref{linear-symmetric-system} for $k=1$:
\begin{align}\label{linear-symmetric-system-1}
 \mathcal{A}_0\{\partial_t U_1+V_1\}+\sum_{i=1}^3 \mathcal{A}_i\partial_i U_1+\mathcal{B} U_1=\mathcal{G}_1.
\end{align}
Here, we define
\begin{align}\label{def-P}
\mathcal{P}U_1:=\mathcal{A}_0 \partial_t U_1+\sum_{i=1}^3 \mathcal{A}_i\partial_i U_1,
\end{align}
and, as mentioned before, we denote the norm of $H^s(\mathbb{R}^3)$ by $\|\cdot\|_s$ for $s\in \mathbb{N}$. In particular, the norm of $L^2(\mathbb{R}^3)$ is denoted by $\|\cdot\|_0$.

\emph{Step 1}. For the $\|\cdot\|_0$ norm, we have
\begin{align*}
\frac{\mathrm{d}}{\mathrm{d}t}(U_1, \mathcal{A}_0 U_1)
=& \,(\partial_t U_1,\mathcal{A}_0 U_1)+(U_1, \partial_t \mathcal{A}_0U_1)+(U_1, \mathcal{A}_0\partial_t U_1)\\
=&\,2\bigg(U_1, \mathcal{P}U_1-\sum_{i=1}^3 \mathcal{A}_i\partial_i U_1\bigg)+(U_1, \partial_t \mathcal{A}_0U_1)\\
=&\,2(\mathcal{P}U_1,U_1)+\bigg(U_1, \bigg(\partial_t \mathcal{A}_0+\sum_{i=1}^3 \partial_i\mathcal{A}_i\bigg)U_1\bigg)\\
=& \,2(-\mathcal{A}_0 V_1,U_1)+(-\mathcal{B}U_1+\mathcal{G}_1,U_1)+\bigg(U_1, \bigg(\partial_t \mathcal{A}_0+\sum_{i=1}^3 \partial_i\mathcal{A}_i\bigg)U_1\bigg).
\end{align*}
The first term on the last line reads
\begin{align*}
  (-\mathcal{A}_0 V_1,U_1)=&\int -\rho_0 \nabla \phi_1 \cdot u_1 \mathrm{d}x\\
  =& \int \phi_1 \nabla\cdot (\rho_0 u_1)\mathrm{d}x\\
  =& \int \phi_1 (-\partial_t \rho_1-\nabla\cdot (\rho_1 u_0))\mathrm{d}x  \quad (\rm{using}\,\,\,\eqref{mass} )\\
  =& \int -\phi_1 \partial_t \rho_1 \mathrm{d}x+\int \nabla \phi_1\cdot \rho_1 u_0 \mathrm{d}x.
\end{align*}
By using \eqref{Fi} of $\e^1$ step, we get
\begin{align*}
\int \phi_1 \partial_t \rho_1 \mathrm{d}x
=&\,\int \phi_1 (\partial_t(e^{\phi_0})\phi_1+e^{\phi_0}\partial_t \phi_1-\Delta (\partial_t \phi_1))\mathrm{d}x\\
=&\, \frac{1}{2}\frac{\mathrm{d}}{\mathrm{d}t}\big(\|\nabla \phi_1\|^2_0+\big\|\phi_1 e^{\frac{\phi_0}{2}}\big\|_0^2\big)+\frac{1}{2}\int \phi_1^2 \partial_t (e^{\phi_0})\mathrm{d}x.
\end{align*}
In summary, we have
\begin{align*}
& \frac{\mathrm{d}}{\mathrm{d}t}\Big(\Big\|\sqrt{U_1^t \mathcal{A}_0 U_1}\Big\|_0^2+\|\nabla \phi_1\|^2_0+\big\|\phi_1 e^{\frac{\phi_0}{2}}\big\|_0^2\Big)\\
 &\quad=  2 \int \nabla \phi_1\cdot \rho_1 u_0 \mathrm{d}x-\int \phi_1^2 \partial_t (e^{\phi_0})\mathrm{d}x+(-\mathcal{B}U_1+\mathcal{G}_1,U_1)+\bigg(U_1, \bigg(\partial_t \mathcal{A}_0+\sum_{i=1}^3 \partial_i\mathcal{A}_i\bigg)U_1\bigg).
\end{align*}
Note that $\mathcal{B}$, $\mathcal{G}_1$, $\partial_t \mathcal{A}_0$ and $\partial_i\mathcal{A}_i$ consist of $\rho_0-1,u_0, \theta_0$ and their first spatial derivatives, which decay at the rate of $t^{-16/15}$ by Proposition \ref{result-of-guoyan},  thus we have
\begin{align*}
&\left|2 \int \nabla \phi_1\cdot \rho_1 u_0 \mathrm{d}x+(-\mathcal{B}U_1+\mathcal{G}_1,U_1)+\bigg(U_1, \bigg(\partial_t \mathcal{A}_0+\sum_{i=1}^3 \partial_i\mathcal{A}_i\bigg)U_1\bigg)\right|\\
&\quad\leq \frac{C}{(1+t)^{\frac{16}{15}}}\{\|U_1\|_0^2+\|V_1\|_0^2\}
+\frac{C}{(1+t)^{\frac{16}{15}}}\{\|U_1\|_0+\|V_1\|_0\}.
\end{align*}
To deal with the term $\int \phi_1^2 \partial_t (e^{\phi_0})\mathrm{d}x$, we study the estimates of $\partial_t \phi_0$. Due to the $\e^0$ step of \eqref{Fi}, it holds $\Delta \phi_0=e^{\phi_0}-\rho_0$ and thus
$$\Delta \partial_t\phi_0=e^{\phi_0}\partial_t \phi_0-\partial_t \rho_0.$$
From Remark \ref{phi_0}, we know that $\phi_0\in C_b^\infty(\mathbb{R}^3)$ of smooth bounded functions with bounded $x-$derivatives of all orders. Hence $e^{\phi_0}\in C^\infty(\mathbb{R}^3)$ and there exists constants $C_0,C_1>0$ such that $0<C_0\leq e^{\phi_0}\leq C_1$. Therefore, the symbol of $e^{\phi_0(x)}-\Delta$, which is $e^{\phi_0(x)}+|\xi|^2$, satisfies
\begin{align*}
  |\partial_\xi^\alpha \partial_x^\beta (e^{\phi_0(x)}+|\xi|^2)|
  \leq\, & C_{\alpha, \beta} (1+|\xi|)^{2-|\alpha|},\\
  e^{\phi_0(x)}+|\xi|^2\geq\, &C_0(1+|\xi|)^2,
\end{align*} 
where constants $C_{\alpha,\beta}$ and $C_0$ are uniform for all  $t\in [0,\infty)$ as a parameter in $\phi_0(t,x)$ due to $\phi_0\in L^\infty(\mathbb{R}_t^+; H^s(\mathbb{R}_x^3))$.
It implies that
$e^{\phi_0(x)}-\Delta$ is an elliptic pseudo-differential operator of $\Psi^2$, which possesses a parametrix of $\Psi^{-2}$ (see e.g. \cite{Alinhac}).
Note that, by Lax-Milgram theorem, there exists a bounded linear operator $(e^{\phi_0(x)}-\Delta)^{-1}:L^2\rightarrow H^1$.  
Thus the operator $(e^{\phi_0(x)}-\Delta)^{-1}$ is a  $\Psi^{-2}$ parametrix of $e^{\phi_0(x)}-\Delta$.  By the H\"{o}lder boundedness properties of the $\Psi^{-2}$ pseudo-differential operators (see e.g. \cite{stein}), we have
$(e^{\phi_0}-\Delta)^{-1}:C^{k,\gamma}\rightarrow C^{k+2,\gamma}$ for $k\in \mathbb{N}_+$ and $\gamma\in (0,1)$, with operator norm constants uniform in $t$.  
It follows that 
\begin{align}\label{partial_t phi_0}
  \|\partial_t \phi_0\|_{C^{k,\gamma}}
  =&\,\|(e^{\phi_0}-\Delta)^{-1} \partial_t \rho_0\|_{C^{k,\gamma}}\nonumber\\
  =&\,\|-(e^{\phi_0}-\Delta)^{-1} \nabla\cdot (\rho_0 u_0)\|_{C^{k,\gamma}}\nonumber\\
  \leq &\, C \|\rho_0 u_0\|_{C^{k-1,\gamma}}\nonumber\\
  \leq & \,C \|\rho_0 u_0\|_{W^{k,\infty}}\nonumber\\
  \leq & \,\frac{C}{(1+t)^{\frac{16}{15}}}.
\end{align}
Thus we get 
\begin{align*}
\int \phi_1^2 \partial_t (e^{\phi_0})\mathrm{d}x\leq \frac{C}{(1+t)^{\frac{16}{15}}}\big\|\phi_1 e^{\frac{\phi_0}{2}}\big\|_0^2.
\end{align*}
We conclude that
\begin{align*}
 & \frac{\mathrm{d}}{\mathrm{d}t}\Big(\Big\|\sqrt{U_1^t \mathcal{A}_0 U_1}\Big\|_0^2+\|\nabla \phi_1\|^2_0+\big\|\phi_1 e^{\frac{\phi_0}{2}}\big\|_0^2\Big)\\
 &\,\,\,\leq \frac{C}{(1+t)^{\frac{16}{15}}}\big\{\|U_1\|_0^2+\|V_1\|_0^2+\big\|\phi_1 e^{\frac{\phi_0}{2}}\big\|_0^2\big\}
+\frac{C}{(1+t)^{\frac{16}{15}}}\{\|U_1\|_0+\|V_1\|_0\}.
\end{align*}

\emph{Step 2}. For the integer $s>0$, define Fourier multiplier  operator $\Lambda^s$ whose symbol is $(1+|\xi|^2)^{\frac{s}{2}}$. If $U_1\in H^s$, then $\Lambda^s U_1\in L^2$. We have
\begin{align*}
  \frac{\mathrm{d}}{\mathrm{d}t}(\Lambda^s U_1,\mathcal{A}_0 \Lambda^s U_1)
  =&\,2(\mathcal{P}\Lambda^s U_1,\Lambda^s U_1)+\bigg(\Lambda^s U_1, \bigg(\partial_t \mathcal{A}_0+\sum_{i=1}^3 \partial_i\mathcal{A}_i\bigg)\Lambda^s U_1\bigg)\\
  =&\,2(\Lambda^s \mathcal{P} U_1,\Lambda^s U_1)+2([\mathcal{P},\Lambda^s] U_1,\Lambda^s U_1)
  +\bigg(\Lambda^s U_1, \bigg(\partial_t \mathcal{A}_0+\sum_{i=1}^3 \partial_i\mathcal{A}_i\bigg)\Lambda^s U_1\bigg)\\
  =&\,-2(\Lambda^s (\mathcal{A}_0 V_1),\Lambda^s U_1)+2(\Lambda^s (-\mathcal{B}U_1+\mathcal{G}_1),\Lambda^s U_1)\\
  &\,+2([\mathcal{P},\Lambda^s] U_1,\Lambda^s U_1)+\bigg(\Lambda^s U_1, \bigg(\partial_t \mathcal{A}_0+\sum_{i=1}^3 \partial_i\mathcal{A}_i\bigg)\Lambda^s U_1\bigg)\\
  =&\,-2(\mathcal{A}_0\Lambda^s  V_1,\Lambda^s U_1)-2([\Lambda^s, \mathcal{A}_0]V_1, \Lambda^s U_1)+2(\Lambda^s (-\mathcal{B}U_1+\mathcal{G}_1),\Lambda^s U_1)\\
  &\,+2([\mathcal{P},\Lambda^s] U_1,\Lambda^s U_1)+\bigg(\Lambda^s U_1, \bigg(\partial_t \mathcal{A}_0+\sum_{i=1}^3 \partial_i\mathcal{A}_i\bigg)\Lambda^s U_1\bigg).
\end{align*}
The first term on the right-hand side of the above equality reads 
\begin{align*}
(\mathcal{A}_0\Lambda^s  V_1,\Lambda^s U_1)
=\,& \int \rho_0 \Lambda^s \nabla \phi_1 \cdot \Lambda^s u_1 \mathrm{d}x\\
=\,&  \int \Lambda^s \nabla \phi_1 \cdot ([\rho_0,\Lambda^s] u_1+\Lambda^s (\rho_0 u_1)) \mathrm{d}x\\
=\,& \int \Lambda^s \nabla \phi_1 \cdot [\rho_0,\Lambda^s] u_1 \mathrm{d}x
-\int \Lambda^s  \phi_1  \Lambda^s( \nabla \cdot (\rho_0 u_1)) \mathrm{d}x\\
=\,& \int \Lambda^s \nabla \phi_1 \cdot [\rho_0,\Lambda^s] u_1 \mathrm{d}x
+\int \Lambda^s  \phi_1  \Lambda^s( \partial_t \rho_1) \mathrm{d}x
-\int \Lambda^s  \nabla\phi_1  \cdot \Lambda^s(  \rho_1 u_0) \mathrm{d}x.
\end{align*}
Routine calculations yield
\begin{align*}
\int \Lambda^s  \phi_1  \Lambda^s( \partial_t \rho_1) \mathrm{d}x
=\,& \int \Lambda^s  \phi_1  \Lambda^s( -\Delta \partial_t \phi_1+\partial_t (e^{\phi_0})\phi_1+e^{\phi_0}\partial_t \phi_1) \mathrm{d}x\\
=\,&\frac{1}{2}\frac{\mathrm{d}}{\mathrm{d}t}\big(\|\Lambda^s \nabla \phi_1\|_0^2
+\big\|e^{\frac{\phi_0}{2}}\Lambda^s  \phi_1\big\|_0^2 \big)
+\frac{1}{2}(\Lambda^s \phi_1, \partial_t (e^{\phi_0})\Lambda^s \phi_1 )\\
&\,+ (\Lambda^s \phi_1, [\Lambda^s,\partial_t (e^{\phi_0})] \phi_1)
+(\Lambda^s \phi_1, [\Lambda^s,e^{\phi_0}] \partial_t\phi_1 ).
\end{align*}
Therefore, it holds
\begin{align}\label{Hs-estimate-1}
&\frac{\mathrm{d}}{\mathrm{d}t}\big(\big\|\sqrt{(\Lambda^s U_1)^t \mathcal{A}_0 (\Lambda^s U_1)}\big\|_0^2+\|\Lambda^s \nabla \phi_1\|_0^2
+\big\|e^{\frac{\phi_0}{2}}\Lambda^s  \phi_1\big\|_0^2\big)\nonumber\\
  &\quad =
-(\Lambda^s \phi_1, \partial_t (e^{\phi_0})\Lambda^s \phi_1 )
-2 (\Lambda^s \phi_1, [\Lambda^s,\partial_t (e^{\phi_0})] \phi_1)
-2(\Lambda^s \phi_1, [\Lambda^s,e^{\phi_0}] \partial_t\phi_1 )\nonumber\\
  &\qquad-2\int \Lambda^s \nabla \phi_1 \cdot [\rho_0,\Lambda^s] u_1 \mathrm{d}x
+2\int \Lambda^s  \nabla\phi_1  \cdot \Lambda^s(  \rho_1 u_0) \mathrm{d}x
  -2([\Lambda^s, \mathcal{A}_0]V_1, \Lambda^s U_1)\nonumber\\
  &\qquad+2(\Lambda^s (-\mathcal{B}U_1+\mathcal{G}_1),\Lambda^s U_1)+2([\mathcal{P},\Lambda^s] U_1,\Lambda^s U_1)+\bigg(\Lambda^s U_1, \bigg(\partial_t \mathcal{A}_0+\sum_{i=1}^3 \partial_i\mathcal{A}_i\bigg)\Lambda^s U_1\bigg).
\end{align}
Recalling  \eqref{partial_t phi_0}, the first term on the right-hand side of \eqref{Hs-estimate-1} is controlled by 
$C \big\|e^{\frac{\phi_0}{2}}\Lambda^s  \phi_1\big\|_0^2$.
To deal with the other terms in \eqref{Hs-estimate-1}, we study the properties of the commutators $[\Lambda^s,\partial_t (e^{\phi_0})]$, $[\Lambda^s,e^{\phi_0}]$, $[\rho_0,\Lambda^s]$, $[\Lambda^s, \mathcal{A}_0]$ and $[\mathcal{P},\Lambda^s]$ for $s\in \mathbb{N}_+$.
Since $\partial_t (e^{\phi_0})$, $e^{\phi_0}$, $\rho_0$ and $\mathcal{A}_0$ are smooth and bounded, they are $\Psi^0$ pseudo-differential operators. By the definition of $\mathcal{P}$ in \eqref{def-P}, $\mathcal{P}$ is a $\Psi^1$ pseudo-differential operator. Clearly, $\Lambda^s \in \Psi^s$. Thus we obtain from the standard theory of commutators (see e.g. \cite{Alinhac}) that
\begin{align*}
  [\Lambda^s,\partial_t (e^{\phi_0})],\,\,[\Lambda^s,e^{\phi_0}], \,\,[\rho_0,\Lambda^s], \,\,[\Lambda^s, \mathcal{A}_0]\in \Psi^{s-1},\quad 
  [\mathcal{P}, \Lambda^s]\in \Psi^s
\end{align*}
whose asymptotic expansions are finite due to $s\in \mathbb{N}_+$.
Since the linear operator $\Psi^s:H^s\rightarrow L^2$ is bounded, we have
\begin{align*}
  ([\mathcal{P},\Lambda^s] U_1,\Lambda^s U_1)\leq &\,\|[\mathcal{P},\Lambda^s] U_1\|_0 \|\Lambda^s U_1\|_0\\
  \leq &\, \frac{C}{(1+t)^{\frac{16}{15}}} \|U_1\|_s \|\Lambda^s U_1\|_0\\
  =& \,\frac{C}{(1+t)^{\frac{16}{15}}}\|\Lambda^s U_1\|_0^2.
\end{align*}
Similarly, we get
\begin{align*}
&-2 (\Lambda^s \phi_1, [\Lambda^s,\partial_t (e^{\phi_0})] \phi_1)
  -2\int \Lambda^s \nabla \phi_1 \cdot [\rho_0,\Lambda^s] u_1 \mathrm{d}x
  -2([\Lambda^s, \mathcal{A}_0]V_1, \Lambda^s U_1)\\
\,& \quad\leq \frac{C}{(1+t)^{\frac{16}{15}}}(\|\Lambda^s U_1\|_0^2+\|\Lambda^s V_1\|_0^2+\|\Lambda^s  \phi_1\|_0^2) 
  \leq    \frac{C}{(1+t)^{\frac{16}{15}}}\big(\|\Lambda^s U_1\|_0^2+\|\Lambda^s V_1\|_0^2+\big\|e^{\frac{\phi_0}{2}}\Lambda^s  \phi_1\big\|_0^2\big).
\end{align*}
We also have
\begin{align*}
&2\int \Lambda^s  \nabla\phi_1  \cdot \Lambda^s(  \rho_1 u_0) \mathrm{d}x
  +2(\Lambda^s (-\mathcal{B}U_1+\mathcal{G}_1),\Lambda^s U_1)
  +\bigg(\Lambda^s U_1, \bigg(\partial_t \mathcal{A}_0+\sum_{i=1}^3 \partial_i\mathcal{A}_i\bigg)\Lambda^s U_1\bigg)\\
&\quad\leq \,\frac{C}{(1+t)^{\frac{16}{15}}}(\|U_1\|_s^2+\|V_1\|_s^2)
+\frac{C}{(1+t)^{\frac{16}{15}}}(\|U_1\|_s+\|V_1\|_s).
\end{align*}
The remainder term on the right-hand side of  \eqref{Hs-estimate-1} is $(\Lambda^s \phi_1, [\Lambda^s,e^{\phi_0}] \partial_t\phi_1 )$. From the $\e^1$ step of \eqref{Fi}, we know that $\phi_1=(e^{\phi_0}-\Delta)^{-1}\rho_1$ and 
\begin{align*}
\Delta \partial_t \phi_1=\partial_t (e^{\phi_0})\phi_1+e^{\phi_0}\partial_t \phi_1-\partial_t \rho_1.
\end{align*}
Hence, it holds
\begin{align*}
\partial_t \phi_1=&\,(e^{\phi_0}-\Delta)^{-1}\left(\partial_t \rho_1-\partial_t (e^{\phi_0})(e^{\phi_0}-\Delta)^{-1}\rho_1\right)\\
=& \,-(e^{\phi_0}-\Delta)^{-1}\left(\nabla\cdot ( \rho_0 u_1+\rho_1u_0)+\partial_t (e^{\phi_0})(e^{\phi_0}-\Delta)^{-1}\rho_1\right).
\end{align*}
It follows by $(e^{\phi_0}-\Delta)^{-1}(\in \Psi^{-2}):H^s\rightarrow H^{s+2}$ that
\begin{align*}
  \|\partial_t \phi_1\|_s\leq C \|\rho_0 u_1+\rho_1u_0\|_{s-1}+C\|\rho_1\|_{s-4}\leq C \|U_1\|_s.
\end{align*}
Then we obtain
$$(\Lambda^s \phi_1, [\Lambda^s,e^{\phi_0}] \partial_t\phi_1 )\leq \frac{C}{(1+t)^{\frac{16}{15}}}\big(\| U_1\|_s^2+\big\|e^{\frac{\phi_0}{2}}\Lambda^s  \phi_1\big\|_0^2\big).$$
In summary, for any integer $s\geq 0$, we deduce that
\begin{align}\label{U1-s}
  &\frac{\mathrm{d}}{\mathrm{d}t}
\big(\|\sqrt{(\Lambda^s U_1)^t \mathcal{A}_0 (\Lambda^s U_1)}\|_0^2+\|\Lambda^s \nabla \phi_1\|_0^2
+\big\|e^{\frac{\phi_0}{2}}\Lambda^s  \phi_1\big\|_0^2\big)\nonumber\\
 &\quad \leq \frac{C}{(1+t)^{\frac{16}{15}}}\big(\|\Lambda^s U_1\|_0^2+\|\Lambda^s V_1\|_0^2+\big\|e^{\frac{\phi_0}{2}}\Lambda^s  \phi_1\big\|_0^2\big)\nonumber\\
  &\qquad +\frac{C}{(1+t)^{\frac{16}{15}}}\big(\|\Lambda^s U_1\|_0+\|\Lambda^s V_1\|_0+\big\|e^{\frac{\phi_0}{2}}\Lambda^s  \phi_1\big\|_0\big).
\end{align}
Utilizing Gronwall's inequality, we arrive at
$\|U_1(t)\|_s+\|V_1(t)\|_s+\|\phi_1(t)\|_s\leq C(1+t)^0$.

\smallskip
Now, we consider the equation \eqref{linear-symmetric-system} for $k=2$:
\begin{align}\label{linear-symmetric-system-2}
 \mathcal{A}_0\{\partial_t U_2+V_2\}+\sum_{i=1}^3 \mathcal{A}_i\partial_i U_2+\mathcal{B} U_2=\mathcal{G}_2.
\end{align}
There are two main differences from the $k=1$ case \eqref{linear-symmetric-system-1}.

(i). For the $\|\cdot\|_0$ norm, the term
\begin{align*}
  (-\mathcal{A}_0 V_2,U_2)=&\int -\rho_0 \nabla \phi_2 \cdot u_2 \mathrm{d}x=\int -\phi_2 \partial_t \rho_2 \mathrm{d}x+\int \nabla \phi_2\cdot \rho_2 u_0 \mathrm{d}x.
\end{align*}
Due to \eqref{Fi} of $\e^2$ step, we get
\begin{align*}
\int \phi_2 \partial_t \rho_2 \mathrm{d}x=\,&\int \phi_2 (\partial_t(e^{\phi_0})\phi_2+e^{\phi_0}\partial_t \phi_2-\Delta (\partial_t \phi_2))\mathrm{d}x\\
&+ \int \phi_2 \Big(\partial_t(e^{\phi_0})\frac{\phi_1^2}{2}+e^{\phi_0}\partial_t \Big(\frac{\phi_1^2}{2}\Big)\Big)\mathrm{d}x\\
=\,& \frac{1}{2}\frac{\mathrm{d}}{\mathrm{d}t}\big(\|\nabla \phi_2\|^2_0+\|\phi_2 e^{\frac{\phi_0}{2}}\|_0^2\big)+\frac{1}{2}\int \phi_2^2 \partial_t (e^{\phi_0})\mathrm{d}x\\
&+ \frac{1}{2}\int \phi_2 (\partial_t(e^{\phi_0})\phi_1^2+e^{\phi_0}\partial_t (\phi_1^2))\mathrm{d}x,
\end{align*}
here, the last line on the right-hand side is new. Clearly, it holds
\begin{align*}
  \int \phi_2 (\partial_t(e^{\phi_0})\phi_1^2+e^{\phi_0}\partial_t (\phi_1^2))\mathrm{d}x
  \leq &\, C \|\phi_2\|_0(\|\phi_1\|_0\|\phi_1\|_\infty
  +\|\partial_t \phi_1\|_0\|\phi_1\|_\infty)\\
  \leq & \,C \|\phi_2\|_0 \leq  C \|U_2\|_0.
\end{align*}
The cases for the integer $s>0$ can be similarly  handled.

(ii). To deal with the term $\mathcal{G}_2=(0, \mathfrak{f}_1, \frac{1}{2\theta_0}\mathfrak{g}_1)^t$ in \eqref{linear-symmetric-system-2}, we note that $L^{-1}$ preserves decay in $v$ and 
\begin{align*}
&\{\mathbf{I}-\mathbf{P}\}\Big(\frac{F_{2}}{\sqrt{\mu}}\Big)\\
=&\,L^{-1}\bigg(-\frac{\{\partial_t+v\cdot \nabla_x\}F_1-\sum_{ {i, j\geq 0 ,i+j=1  }}\nabla_x \phi_i \cdot \nabla_v F_j
-Q(F_1, F_1)}{\sqrt{\mu}}\bigg)\\
=&\, L^{-1}\bigg(-\frac{\{\partial_t+v\cdot \nabla_x\}\big(\sqrt{\mu}\mathbf{P}\big(\frac{F_1}{\sqrt{\mu}}\big)
+\sqrt{\mu}\{\mathbf{I}-\mathbf{P}\}\big(\frac{F_1}{\sqrt{\mu}}\big)\big)
-\sum_{ {i, j\geq 0 ,i+j=1  }}\nabla_x \phi_i \cdot \nabla_v F_j
-Q(F_1, F_1)}{\sqrt{\mu}}\bigg),
\end{align*}
where the term $$L^{-1}\Big(\frac{\{\partial_t+v\cdot \nabla_x\}\sqrt{\mu}}{\sqrt{\mu}}
\{\mathbf{I}-\mathbf{P}\}\Big(\frac{F_1}{\sqrt{\mu}}\Big)\Big)$$ is dominant. 
Since 
\begin{align*}
\{\mathbf{I}-\mathbf{P}\}\Big(\frac{F_{1}}{\sqrt{\mu}}\Big)=&\,L^{-1}\Big(-\frac{\{\partial_t+v\cdot \nabla_x-\nabla_x \phi_0\cdot \nabla_v\}\mu}{\sqrt{\mu}}\Big)\\
\leq& \, C(\|\partial \rho_0\|_\infty+\|\partial u_0\|_\infty+\|\partial \theta_0\|_\infty+\|\nabla \phi_0\|_\infty)(1+|v|^3)\sqrt{\mu}\\
\leq &\, C (1+t)^{-\frac{16}{15}}(1+|v|^3)\sqrt{\mu}, \qquad (\textrm{here}\,\partial=\partial_t\,\mathrm{or}\,\partial_x)
\end{align*}
we have
$$L^{-1}\Big(\frac{\{\partial_t+v\cdot \nabla_x\}\sqrt{\mu}}{\sqrt{\mu}}
\{\mathbf{I}-\mathbf{P}\}\Big(\frac{F_1}{\sqrt{\mu}}\Big)\Big)\leq C (1+|v|^3)^2\sqrt{\mu}.$$
Hence, we get $$\Big|\{\mathbf{I}-\mathbf{P}\}\Big(\frac{F_{2}}{\sqrt{\mu}}\Big)\Big|\leq C (1+t)^0(1+|v|^3)^2\sqrt{\mu}.$$

In conclusion, similarly to \eqref{U1-s}, we obtain that
\begin{align}\label{U2-s}
  \frac{\mathrm{d}}{\mathrm{d}t}&\big(\|\sqrt{(\Lambda^s U_2)^t \mathcal{A}_0 (\Lambda^s U_2)}\|_0^2+\|\Lambda^s \nabla \phi_2\|_0^2
+\big\|e^{\frac{\phi_0}{2}}\Lambda^s  \phi_2\big\|_0^2\big)\nonumber\\
  &\leq \,\frac{C}{(1+t)^{\frac{16}{15}}}\big(\|\Lambda^s U_2\|_0^2+\|\Lambda^s V_2\|_0^2+\big\|e^{\frac{\phi_0}{2}}\Lambda^s  \phi_2\big\|_0^2\big)\nonumber\\
  &\quad\,+C(1+t)^0\big(\|\Lambda^s U_2\|_0+\|\Lambda^s V_2\|_0+\big\|e^{\frac{\phi_0}{2}}\Lambda^s  \phi_2\big\|_0\big),
\end{align}
which implies that $$\|U_2(t)\|_s+\|V_2(t)\|_s+\|\phi_2(t)\|_s\leq C(1+t)^1$$ by using Gronwall's inequality.

Next, inductively, we have
$$\Big|\{\mathbf{I}-\mathbf{P}\}\Big(\frac{F_{n+1}}{\sqrt{\mu}}\Big)\Big|\leq C (1+t)^{n-1}(1+|v|^3)^{n+1}\sqrt{\mu} \,\,\,\,\, \text{for}\,\,\,\, n\in \mathbb{N}_+.$$ 
On the other hand, from the $\e^{n+1}$ step of \eqref{Fi}, we get
\begin{align*}
\int \phi_{n+1} \partial_t \rho_{n+1} \mathrm{d}x=\,&\int \phi_{n+1} (\partial_t(e^{\phi_0})\phi_{n+1}+e^{\phi_0}\partial_t \phi_{n+1}-\Delta (\partial_t \phi_{n+1}))\mathrm{d}x\\
&+ \int \phi_{n+1} (\partial_t(e^{\phi_0})(A_{n+1}-\phi_{n+1})+e^{\phi_0}\partial_t (A_{n+1}-\phi_{n+1}))\mathrm{d}x\\
=\,&\frac{1}{2}\frac{\mathrm{d}}{\mathrm{d}t}\big(\|\nabla \phi_{n+1}\|^2_0+\big\|\phi_{n+1} e^{\frac{\phi_0}{2}}\big\|_0^2\big)+\frac{1}{2}\int \phi_{n+1}^2 \partial_t (e^{\phi_0})\mathrm{d}x\\
&+ \int \phi_{n+1} (\partial_t(e^{\phi_0})(A_{n+1}-\phi_{n+1})+e^{\phi_0}\partial_t (A_{n+1}-\phi_{n+1}))\mathrm{d}x.
\end{align*}
Notice that the term $A_{n+1}-\phi_{n+1}$ on the last line consists of the products of at least two terms among $\{\phi_1,\ldots, \phi_n\}$, for example $\phi_{n-m}\phi_m$, whose $L^2$ norm satisfies controllable temporal growth:
$$\|\phi_{n-m}\phi_m\|_0\leq C \|\phi_{n-m}\|_0\|\phi_m\|_\infty\leq C (1+t)^{n-m-1}(1+t)^{m-1}\leq C (1+t)^{n-2}\leq C (1+t)^{n-1}.$$  
It follows that
\begin{align*}
\int \phi_{n+1} \big(\partial_t(e^{\phi_0})(A_{n+1}-\phi_{n+1})+e^{\phi_0}\partial_t (A_{n+1}-\phi_{n+1})\big)\mathrm{d}x\leq C (1+t)^{n-1} \|\phi_{n+1}\|_0.
\end{align*}

Similarly to \eqref{U1-s} and \eqref{U2-s}, we deduce that
\begin{align}
  \frac{\mathrm{d}}{\mathrm{d}t}&\big(\big\|\sqrt{(\Lambda^s U_{n+1})^t \mathcal{A}_0 (\Lambda^s U_{n+1})}\big\|_0^2+\|\Lambda^s \nabla \phi_{n+1}\|_0^2
+\big\|e^{\frac{\phi_0}{2}}\Lambda^s  \phi_{n+1}\big\|_0^2\big)\nonumber\\
  &\leq\,\frac{C}{(1+t)^{\frac{16}{15}}}\big(\|\Lambda^s U_{n+1}\|_0^2+\|\Lambda^s V_{n+1}\|_0^2+\big\|e^{\frac{\phi_0}{2}}\Lambda^s  \phi_{n+1}\big\|_0^2\big)\nonumber\\
  &\quad\,+C(1+t)^{n-1}\big(\|\Lambda^s U_{n+1}\|_0+\|\Lambda^s V_{n+1}\|_0+\big\|e^{\frac{\phi_0}{2}}\Lambda^s  \phi_{n+1}\big\|_0\big),
\end{align}
which implies that $$\|U_{n+1}(t)\|_s+\|V_{n+1}(t)\|_s+\|\phi_{n+1}(t)\|_s\leq C(1+t)^n, \quad n\in \mathbb{N}_+.$$
This completes the proof of Proposition  \ref{propa}.
\end{proof}

\section{The solution $\phi_R$ to the remainder equation \eqref{vpbr}$_2$}
\label{Sec-phi_R}
First, from Proposition \ref{result-of-LiYangZhong2016}, we see that there exists a unique strong solution to the ionic Vlasov-Poisson-Boltzmann system \eqref{vpb} with $\sup_{t\geq 0}\|\phi(t)\|_{H^2(\mathbb{R}^3)}=o(1),$
where $\phi(t,x)$ satisfies the Poisson equation \eqref{vpb}$_2$: 
$\Delta \phi=e^{\phi}-\int_{\mathbb{R}^3}F\, \mathrm{d}v.$

Second, from Proposition \ref{result-of-guoyan} and Remark \ref{phi_0}, we see that $\sup_{t\geq 0}\|\phi_0(t)\|_{L^\infty(\mathbb{R}^3)}=o(1),$
where $\phi_0(t,x)$ satisfies the Poisson equation in the compressible ionic Euler-Poisson system \eqref{EPS}$_2$: 
$\Delta \phi_0=e^{\phi_0}-\rho_0.$

Furthermore, in Proposition \ref{propa}, we have proved that $\|\phi_i\|_{H^s}\leq C(1+t)^{i-1}$ for $1\leq i\leq 2k-1$ and $s\in \mathbb{N}$. For $t\leq \e^{-m}$, $0<m<\frac{1}{2}$, it holds $\e \|\phi_1\|_{L^\infty}+\e^{2k-1}\|\phi_{2k-1}\|_{L^\infty}\leq C \e \mathcal{I}_1(t,\e)=o(1)$, where $\mathcal{I}_1(t,\e)$ is defined in \eqref{I_1}.

\emph{Below we denote $\e^k\phi_R$ by $\psi_R$ for presentation simplicity}. From the  expansion \eqref{expansion2}, we have
\begin{align*}
\sup_{t\geq 0}\|\psi_R\|_{L^\infty}\leq \sup_{t\geq 0}\big\{\|\phi\|_{L^\infty}+\|\phi_0\|_{L^\infty}+\e\|\phi_1\|_{L^\infty}
+\cdots +\e^{2k-1}\|\phi_{2k-1}\|_{L^\infty}\big\}
=o(1),
\end{align*}
which implies that $\|\psi_R\|_{L^\infty}$ has a uniform upper bound independent of $\psi_R$, $\e$ and $t$.

We next establish a sharper estimate of $\psi_R$, which is necessary in Section \ref{sec3}. 

Recall that we have denoted $\exp\{\phi_0+\cdots+ \e^{2k-1}\phi_{2k-1}\}$ by $H(\e)$ and the $(2k-1)$th-order Taylor polynomial of $H(\e)$ by $T_{2k-1}(H(\e))$. 
The nonlinear Poisson equation \eqref{vpbr}$_2$ for the remainder $\psi_R$ is equivalent to
\begin{align}\label{psi_R}
  \Delta \psi_R= & (e^{\psi_R}-1)H(\e)+H(\e)-T_{2k-1}(H(\e))-\int_{\mathbb{R}^3}\e^k R\, \mathrm{d}v.
\end{align}

\begin{proposition}\label{prop1}
Suppose that $\psi_R\in H^2(\mathbb{R}^3)$ is the unique solution to the equation \eqref{psi_R}. Assuming that $\|f\|_{0}+\e^{\frac{3}{2}}\|h\|_{\infty}$ is bounded for $R = \sqrt{\mu} f =\frac{\sqrt{\mu_M}}{w(v)}h$ in \eqref{def-f-h}, we have, for sufficiently small $\e$,
\begin{align}
\|\psi_R\|_2\leq C\e^k(\e^{\frac{1}{2}}+\|f\|_0),\quad \mathrm{i.e.}, \,\, \|\phi_R\|_2\leq C(\e^{\frac{1}{2}}+\|f\|_0),\label{psiR-e}
\end{align}
where the constant $C>0$ is independent of $\e$ and $\phi_R$.
\end{proposition}
The proof of \eqref{psiR-e} relies on the following \eqref{H-e}. 
To uniformly handle the term $H-T_{2k-1}(H)$ and its derivatives, we summary them together as the following proposition. Note that \eqref{dtH-e} and \eqref{dxH-e} are crucial in Sections \ref{sec3} and \ref{sec4}. 
\begin{proposition}\label{prop2}
For $\e$ sufficiently small, we have
  \begin{align}
   & \|H(\e)-T_{2k-1}(H(\e))\|_0\leq C\e^{k+\frac{1}{2}},\quad \|H(\e)-T_{2k-1}(H(\e))\|_\infty\leq C\e^{k+\frac{1}{2}}.\label{H-e}
\end{align}
Moreover, for the first order derivatives of $t$ and $x$ of $H(\e)-T_{2k-1}(H(\e))$, we have 
\begin{align}
 \|\partial_t (H(\e)-T_{2k-1}(H(\e)))\|_0\,& \leq C\e^{k+\frac{1}{2}},\label{dtH-e}\\
    \|\nabla_x (H(\e)-T_{2k-1}(H(\e)))\|_\infty\,& \leq C\e^{k+\frac{1}{2}}.\label{dxH-e}
\end{align}
\end{proposition}
\begin{proof}[Proof of Proposition \ref{prop1}]
Once we establish \eqref{H-e} in Proposition \ref{prop2}, we readily conclude Proposition \ref{prop1}. In fact, by 
using the mean value theorem, for any $x\in \mathbb{R}^3$, the equality $$e^{\psi_R}-1=e^{\psi_\xi}\psi_R$$ holds for a function $\psi_\xi\in L^\infty(\mathbb{R}^3)$ due to the fact that $\|\psi_R\|_2$ has a bound independent of $\psi_R$. From the above equality, one easily gets
\begin{align*}
\|\psi_R\|_1\leq &\,C e^{\|\psi_R\|_\infty}\Big(\|H(\e)-T_{2k-1}(H(\e))\|_0+\Big\|\int_{\mathbb{R}^3}\e^k R\, \mathrm{d}v\Big\|_0\Big)\\
\leq &\,Ce^{\|\psi_R\|_\infty}\e^k(\e^{\frac{1}{2}}+\|f\|_0).
\end{align*}
Moreover, because $\psi_R$ satisfies 
$$-(I-\Delta)\psi_R=H(\e) e^{\psi_\xi}\psi_R-\psi_R+H(\e)-T_{2k-1}(H(\e))-\int_{\mathbb{R}^3}\e^k R\, \mathrm{d}v,$$
it holds
$$\psi_R=-(I-\Delta)^{-1}\Big(H(\e) e^{\psi_\xi}\psi_R-\psi_R+H(\e)-T_{2k-1}(H(\e))-\int_{\mathbb{R}^3}\e^k R\, \mathrm{d}v\Big),$$
and there exists a constant $C>0$ such that
\begin{align*}
  \|\psi_R\|_2\leq &\,C \Big\|H(\e) e^{\psi_\xi}\psi_R-\psi_R+H(\e)-T_{2k-1}(H(\e))-\int_{\mathbb{R}^3}\e^k R\, \mathrm{d}v\Big\|_0\\
  \leq &\,Ce^{\|\psi_R\|_\infty}\|\psi_R\|_0+C\e^k\e^{\frac{1}{2}}+\e^k \|f\|_0 \\
  \leq &\,Ce^{2\|\psi_R\|_\infty}\e^k\big(\e^{\frac{1}{2}}+\|f\|_0\big),
\end{align*}
which implies that $$\|\psi_R\|_\infty\leq Ce^{2\|\psi_R\|_\infty}\e^k\big(\e^{\frac{1}{2}}+\|f\|_0\big).$$
By the boundedness of $\|f\|_0$ and $\|\psi_R\|_{\infty}=o(1)$, when $\e$ is sufficiently small, it holds $\|\psi_R\|_\infty=O(\e^k)$. Moreover, there exists a constant $C>0$ independent of $\psi_R$ such that
$$\|\psi_R\|_2\leq C\e^k\big(\e^{\frac{1}{2}}+\|f\|_0\big).$$
Thus Proposition \ref{prop1} is proved.
\end{proof}
\smallskip
Next, we verify Proposition \ref{prop2}.
\begin{proof}[Proof of Proposition \ref{prop2}]

\emph{Step 1. Proof of \eqref{H-e}}.
Note that for any fixed $(t,x)\in [0,+\infty)\times \mathbb{R}^3$,
\begin{align*}
H(\e)-T_{2k-1}(H(\e))=\frac{H^{(2k)}(\vartheta_{t,x} \e)}{(2k)!}\e^{2k}, \quad 0<\vartheta_{t,x} <1.
\end{align*}
To estimate $H^{(2k)}(\vartheta_{t,x} \e)$, we denote $$g(\e):=\phi_0+\cdots +\e^{2k-1}\phi_{2k-1}$$
 and observe that
\begin{align}\label{simple-Faadi}
H^{(1)}(\e)=&\,g^{(1)}(\e)H(\e),\nonumber\\
H^{(2)}(\e)=&\,g^{(2)}(\e)H(\e)+g^{(1)}(\e)g^{(1)}(\e) H(\e),\nonumber\\
H^{(3)}(\e)=&\,g^{(3)}(\e)H(\e)+3g^{(2)}(\e)g^{(1)}(\e) H(\e)+g^{(1)}(\e)g^{(1)}(\e) g^{(1)}(\e)H(\e),\nonumber\\
\cdots &\nonumber\\
H^{(2k)}(\e)=&\, H(\e) \sum_\varpi C_\varpi \prod_{1\leq \ell \leq 2k}(g^{(\ell)}(\e))^{\varpi_\ell}= H(\e) \sum_\varpi C_\varpi \prod_{1\leq \ell \leq 2k-1}(g^{(\ell)}(\e))^{\varpi_\ell},
\end{align}
where the coefficients $C_\varpi$ are nonnegative integers, and the sum is taken over those $\varpi$ such that $\varpi_\ell\in \mathbb{N}_+$, $\sum_{1\leq \ell\leq 2k-1}\ell \varpi_\ell=2k$.
Note that \eqref{simple-Faadi} is a simple version of celebrated Fa\`{a} di Bruno's formula. 
It is clear that $g^{(2k)}(\e)=0$ and 
\begin{align*}
g^{(\ell)}(\e)=\ell!\phi_\ell+(\ell+1)!\e^1\phi_{\ell+1}+\cdots + \frac{(2k-1)!}{(2k-\ell-1)!}\e^{2k-\ell-1}\phi_{2k-1},\quad 1\leq \ell\leq 2k-1.
\end{align*}
For $t\leq \e^{-m}$, $0<m<\frac{1}{2}$ and $1\leq \ell\leq 2k-1$, we have 
\begin{align*}
\|g^{(\ell)}(\vartheta_{t,x}\e)\|_0   \leq &\, \ell!\|\phi_\ell\|_0+(\ell+1)!
\|(\vartheta_{t,x}\e)^1\phi_{\ell+1}\|_0
+\cdots + \frac{(2k-1)!}{(2k-\ell-1)!}
\|(\vartheta_{t,x}\e)^{2k-\ell-1}\phi_{2k-1}\|_0\\
\leq &\, \ell!\|\phi_\ell\|_0+(\ell+1)!\e^1\|\phi_{\ell+1}\|_0
+\cdots + \frac{(2k-1)!}{(2k-\ell-1)!}
\e^{2k-\ell-1}\|\phi_{2k-1}\|_0\\
\leq &\,
C_\ell ((1+t)^{\ell-1}+\e^1 (1+t)^\ell+\cdots + \e^{2k-\ell-1}(1+t)^{2k-2})\\
\leq &\,
C_\ell ((\e^{-m})^{\ell-1}+\e^1 (\e^{-m})^\ell+\cdots + \e^{2k-\ell-1}(\e^{-m})^{2k-2}),
\end{align*}
and 
\begin{align*}
\|g^{(\ell)}(\vartheta_{t,x}\e)\|_\infty   \leq & \, \ell!\|\phi_\ell\|_\infty
+(\ell+1)!\|(\vartheta_{t,x}\e)^1\phi_{\ell+1}\|_\infty
+\cdots + \frac{(2k-1)!}{(2k-\ell-1)!}
\|(\vartheta_{t,x}\e)^{2k-\ell-1}\phi_{2k-1}\|_\infty\\
\leq & \, \ell!\|\phi_\ell\|_\infty+(\ell+1)!\e^1\|\phi_{\ell+1}\|_\infty
+\cdots + \frac{(2k-1)!}{(2k-\ell-1)!}
\e^{2k-\ell-1}\|\phi_{2k-1}\|_\infty\\
\leq &\,
C_\ell ((1+t)^{\ell-1}+\e^1 (1+t)^\ell+\cdots + \e^{2k-\ell-1}(1+t)^{2k-2})\\
\leq &\,
C_\ell ((\e^{-m})^{\ell-1}+\e^1 (\e^{-m})^\ell+\cdots + \e^{2k-\ell-1}(\e^{-m})^{2k-2}).
\end{align*}
Thus, 
\begin{align}\label{g-ell-L2}
(\e^{\frac{1}{2}})^{\ell-1}\|g^{(\ell)}(\vartheta_{t,x}\e)\|_0  
\leq &\,
C_\ell(\e^{\frac{1}{2}})^{\ell-1} ((\e^{-m})^{\ell-1}+\e^1 (\e^{-m})^\ell+\cdots + \e^{2k-\ell-1}(\e^{-m})^{2k-2})\nonumber\\
= &\,
C_\ell ((\e^{\frac{1}{2}-m})^{\ell-1}
+\e^{\frac{1}{2}}(\e^{\frac{1}{2}-m})^{\ell}
+\cdots 
+\e^{\frac{2k-\ell-1}{2}}(\e^{\frac{1}{2}-m})^{2k-2})\nonumber\\
\leq &\,2 C_\ell (\e^{\frac{1}{2}-m})^{\ell-1},
\end{align}
and 
\begin{align}\label{g-ell-Linfty}
(\e^{\frac{1}{2}})^{\ell-1}\|g^{(\ell)}(\vartheta_{t,x}\e)\|_\infty  
\leq 2 C_\ell (\e^{\frac{1}{2}-m})^{\ell-1}.
\end{align}
Hence we get a crucial estimate:
\begin{align}\label{product-estimate}
  &\Big\|\e^{2k}\prod_{1\leq \ell \leq 2k-1}(g^{(\ell)}(\vartheta_{t,x} \e))^{\varpi_\ell}\Big\|_0\nonumber\\
  \leq & \, C(\e^{\frac{1}{2}})^{2k}(\e^{\frac{1}{2}})^{2k-\sum_{1\leq \ell\leq 2k-1}(\ell-1) \varpi_\ell}
  (\e^{\frac{1}{2}-m})^{\sum_{1\leq \ell\leq 2k-1}(\ell-1) \varpi_\ell}\nonumber\\
  \leq & \,C(\e^{\frac{1}{2}})^{2k}(\e^{\frac{1}{2}})^{\sum_{1\leq \ell\leq 2k-1} \varpi_\ell}\nonumber\\
  \leq & \,C(\e^{\frac{1}{2}})^{2k}(\e^{\frac{1}{2}})^1= C \e^k \e^{\frac{1}{2}}.
\end{align}
The second inequality above is due to the fact that $\sum_{1\leq \ell\leq 2k}\ell \varpi_\ell=2k$ and $\varpi_\ell\in \mathbb{N}_+$.
Because $\|H(\vartheta_{t,x} \e)\|_\infty<\infty$ for $t\leq \e^{-m}$,  we obtain
\begin{align*}
\|H(\e)-T_{2k-1}(H(\e))\|_0= &\, \Big\|\frac{H^{(2k)}(\vartheta_{t,x} \e)}{(2k)!}\e^{2k}\Big
\|_0\\
= &\, \Big\| H(\vartheta_{t,x} \e) \e^{2k}\frac{1}{(2k)!}\sum_\varpi C_\varpi \prod_{1\leq \ell \leq 2k-1}(g^{(\ell)}(\vartheta_{t,x} \e))^{\varpi_\ell}\Big\|_0\\
\leq   &\,C\Big\|\e^{2k}\prod_{1\leq \ell \leq 2k-1}(g^{(\ell)}(\vartheta_{t,x} \e))^{\varpi_\ell}\Big\|_0\\
\leq   &\,   C \e^k \e^{\frac{1}{2}}.
\end{align*}
Similarly, we have
\begin{align*}
\|H(\e)-T_{2k-1}(H(\e))\|_\infty\leq C\e^{k+\frac{1}{2}}.
\end{align*}
Thus \eqref{H-e} is proved.

\emph{Step 2. Proof of \eqref{dtH-e}}.
We claim that $$\partial_t (T_{2k-1}(H)(\e))= T_{2k-1}(\partial_t H)(\e).$$
Indeed, it directly follows from the equality of mixed partial derivatives. 
It suffices to verify the
continuity of $\partial_t H(\e)$ (thus the
continuity of $\partial_{t} \partial_\e H(\e)$) for each $t>0$. Since $$\partial_t H(\e)=H(\e)\partial_t g(\e)=(\partial_t \phi_0+\cdots +\e^{2k-1}\partial_t \phi_{2k-1})H(\e),$$
 we shall prove that $\partial_{tt}\phi_i$ is bounded in each time interval. To this end, 
we first prove that $$\|\partial_t \phi_i\|_s\leq C (1+t)^{i-1}$$ for $1\leq i\leq 2k-1$ and
 $s\in \mathbb{N}_+$. 

From the $\e^0$ step of \eqref{Fi}: $\Delta \phi_0=e^{\phi_0}-\rho_0$, we get
$$\Delta (\partial_t\phi_0)=\partial_t \phi_0 e^{\phi_0}-\partial_t \rho_0.$$
Thus, it holds
\begin{align*}
  \partial_t\phi_0=&\,(e^{\phi_0}-\Delta)^{-1}(\partial_t \rho_0)
  =\, -(e^{\phi_0}-\Delta)^{-1}(\nabla\cdot (\rho_0u_0)).
\end{align*}
Due to the facts that $\|\nabla\cdot (\rho_0u_0)\|_s\leq C$ and $(e^{\phi_0}-\Delta)^{-1}\in \Psi^{-2}$, we deduce that
$\|\partial_t\phi_0\|_{s+2}\leq C$.

From the $\e^1$ step of \eqref{Fi}: $\Delta \phi_1=e^{\phi_0}\phi_1-\rho_1$, we get
$$\Delta (\partial_t\phi_1)=\partial_t \phi_0 e^{\phi_0}\phi_1+e^{\phi_0}\partial_t \phi_1-\partial_t \rho_1.$$
Thus, it holds
\begin{align*}
  \partial_t\phi_1=&\,(e^{\phi_0}-\Delta)^{-1}(\partial_t \rho_1-\partial_t \phi_0 e^{\phi_0}\phi_1)\\
  =&\, -(e^{\phi_0}-\Delta)^{-1}(\nabla\cdot (\rho_1u_0+\rho_0u_1)
  +\partial_t \phi_0 e^{\phi_0}\phi_1).
\end{align*}
Since the multiplication by $e^{\phi_0}\in C^\infty_b(\mathbb{R}^3)$ is a continuous map from $H^s(\mathbb{R}^3)$ into itself, $\|\partial_t \phi_0\|_s+\|\phi_1\|_s\leq C$, and $H^s(\mathbb{R}^3)$ is a Banach algebra for $s>\frac{3}{2}$, we have
$$\|\nabla\cdot (\rho_1u_0+\rho_0u_1)
  +\partial_t \phi_0 e^{\phi_0}\phi_1\|_s\leq C.$$
By $(e^{\phi_0}-\Delta)^{-1}\in \Psi^{-2}$, we deduce that
$\|\partial_t\phi_1\|_{s+2}\leq C$.

Inductively, assume that $$\|\partial_t \phi_j\|_s\leq C (1+t)^{j-1}\quad \mathrm{for} \,\,\,\,1\leq j\leq i.$$
 For $j=i+1$, we obtain from the $\e^j$ step of \eqref{Fi} that
\begin{align*}
  \partial_t \phi_{i+1}=-(e^{\phi_0}-\Delta)^{-1}(\nabla\cdot (\rho_{i+1}u_0+\rho_0u_{i+1})+e^{\phi_0}\partial_t (A_{i+1}-\phi_{i+1})+\partial_t \phi_0 e^{\phi_0}A_{i+1}),
\end{align*}
where $$A_{i+1}-\phi_{i+1}=\sum_\varkappa C_\varkappa \prod_{1\leq j\leq i}(\phi_j)^{\varkappa_j}$$ with $\sum_{1\leq j\leq i} j \varkappa_j=i+1$ and $C_\varkappa>0$.
Because it holds $\| \phi_j\|_s\leq C (1+t)^{j-1}$ for $1\leq j\leq 2k-1$, we have
$$\|\nabla\cdot (\rho_{i+1}u_0+\rho_0u_{i+1})+e^{\phi_0}\partial_t (A_{i+1}-\phi_{i+1})+\partial_t \phi_0 e^{\phi_0}A_{i+1}\|_s\leq C (1+t)^i.$$
Hence we deduce
$\|\partial_t\phi_{i+1}\|_{s+2}\leq C(1+t)^i$ by $(e^{\phi_0}-\Delta)^{-1}\in \Psi^{-2}$. 
Therefore, we conclude that $$\|\partial_t \phi_i\|_s\leq C (1+t)^{i-1}\quad \mathrm{for}\,\,\,1\leq i\leq 2k-1\,\,\, \mathrm{and}\,\,\,s\in \mathbb{N}_+.$$

Next, we investigate $\partial_{tt}\phi_i$ for $0\leq i\leq 2k-1$.
From the $\e^0$ step of \eqref{Fi}, we get
\begin{align*}
-(e^{\phi_0}-\Delta)\partial_{tt}\phi_0
=&\,(\partial_{t}\phi_0)^2 e^{\phi_0}-\partial_{tt}\rho_0\\
=&\,(\partial_{t}\phi_0)^2 e^{\phi_0}+\partial_t \nabla_x \cdot (\rho_0 u_0)\\
=&\,(\partial_{t}\phi_0)^2 e^{\phi_0}+\partial_t u_0\cdot \nabla\rho_0+u_0\cdot \nabla \partial_t \rho_0+\partial_t \rho_0\nabla\cdot u_0+\rho_0\nabla\cdot \partial_t u_0\\
=&\,(\partial_{t}\phi_0)^2 e^{\phi_0}-((u_0\cdot \nabla)u_0+\frac{1}{\rho_0}\nabla (\rho_0\theta_0)+\nabla \phi_0)
\cdot \nabla\rho_0\\
&-(u_0\cdot \nabla) \nabla\cdot (\rho_0 u_0)-\nabla\cdot (\rho_0 u_0)\nabla\cdot u_0\\
&-\rho_0\nabla\cdot ((u_0\cdot \nabla)u_0+\frac{1}{\rho_0}\nabla (\rho_0\theta_0)+\nabla \phi_0).
\end{align*}
By elliptic estimates, $\sup_{0\leq t \leq \eta}\|\partial_{tt}\phi_0\|_s$ is bounded for any $\eta>0$, $s\geq 2$.
From the $\e^1$ step of \eqref{Fi}, we get
\begin{align*}
-(e^{\phi_0}-\Delta)\partial_{tt}\phi_1
=&\,((\partial_{t}\phi_0)^2 \phi_1+\partial_{tt}\phi_0\phi_1+2\partial_t \phi_0\partial_t \phi_1) e^{\phi_0}-\partial_{tt}\rho_1\\
=&\, ((\partial_{t}\phi_0)^2 \phi_1+\partial_{tt}\phi_0\phi_1+2\partial_t \phi_0\partial_t \phi_1) e^{\phi_0}+\nabla_x \cdot \partial_t(\rho_0u_1+\rho_1u_0)\\
=&\, ((\partial_{t}\phi_0)^2 \phi_1+\partial_{tt}\phi_0\phi_1+2\partial_t \phi_0\partial_t \phi_1) e^{\phi_0}\\
&\,+\partial_t u_1\cdot \nabla\rho_0+u_1\cdot \nabla \partial_t \rho_0+\partial_t \rho_1\nabla\cdot u_0+\rho_1\nabla\cdot \partial_t u_0,
\end{align*}
where \eqref{mass} and \eqref{moment} substitute $\partial_t \rho_1$ and $\partial_t u_1$ on the right-hand side above into $x$-derivatives of $\rho_1,u_1,\rho_0$ and $u_0$.
By elliptic estimates, $\sup_{0\leq t \leq \eta}\|\partial_{tt}\phi_0\|_s$ is bounded for any $\eta>0$, $s\geq 2$. 

Similarly, we derive that $\partial_t\phi_i\in W^{1,\infty}([0,\eta])$ for any $\eta>0$, which implies that $\partial_t\phi_i$ $(0\leq i\leq 2k-1)$ is continuous in $t$.
Hence we conclude that $\partial_t (T_{2k-1}(H(\e)))= T_{2k-1}(\partial_t H)(\e)$.

Note that
for any fixed $(t,x)\in [0,+\infty)\times\mathbb{R}^3$,
\begin{align*}
\partial_t (H(\e)-T_{2k-1}(H(\e)))=&\, \partial_t H(\e) -T_{2k-1}(\partial_t H)(\e)\\
=&\,\frac{ (\partial_t H)^{(2k)}(\vartheta'_{t,x} \e)}{(2k)!}\e^{2k}, \quad 0<\vartheta'_{t,x} <1.
\end{align*}
According to \eqref{simple-Faadi}, it holds
\begin{align*}
(\partial_t H)^{(2k)}(\e)=\partial_t(H^{(2k)}(\e))= \partial_t H (\e)\sum_\varpi C_\varpi \prod_{1\leq \ell \leq 2k-1}(g^{(\ell)}(\e))^{\varpi_\ell}+ H (\e)\sum_\varpi C_\varpi \partial_t\Big(\prod_{1\leq \ell \leq 2k-1} (g^{(\ell)}(\e))^{\varpi_\ell}\Big).
\end{align*}
We now handle $\|\partial_t H(\e)\|_\infty$ and $\|\partial_t g^{(\ell)}(\e)\|_0$.
Note that
\begin{align}
\partial_t H(\e)=&\, (\partial_t \phi_0+\cdots +\e^{2k-1}\partial_t\phi_{2k-1})\exp\{\phi_0+\cdots +\e^{2k-1}\phi_{2k-1}\},\label{dtH-biaodashi}\\
 \partial_t g^{(\ell)}(\e)=& \, \ell!\partial_t\phi_\ell+(\ell+1)!\e^1\partial_t\phi_{\ell+1}
 +\cdots + \frac{(2k-1)!}{(2k-\ell-1)!}\e^{2k-\ell-1}\partial_t\phi_{2k-1},\quad 1\leq \ell\leq 2k-1.\label{dtg-biaodashi}
\end{align}

From \eqref{partial_t phi_0}, \eqref{dtH-biaodashi} and \eqref{dtg-biaodashi}, we obtain, for $t\leq \e^{-m}$, $0<m<\frac{1}{2}$, that
\begin{align}\label{dtH-Linfty}
\|\partial_t H(\e)\|_\infty\leq &\, (\|\partial_t \phi_0\|_\infty+\cdots +\e^{2k-1}\|\partial_t\phi_{2k-1}\|_\infty)\exp\{\|\phi_0\|_\infty+\cdots +\e^{2k-1}\|\phi_{2k-1}\|_\infty\},\nonumber\\
\leq &\, C\Big(\frac{1}{(1+t)^{\frac{16}{15}}} +\e \mathcal{I}_1(t,\e)\Big)e^{C +C\e \mathcal{I}_1(t,\e)}\leq C,
\end{align}
and for $1\leq \ell\leq 2k-1$,
\begin{align*}
 \|\partial_t g^{(\ell)}(\e)\|_0\leq & \, \ell!\|\partial_t\phi_\ell\|_0+(\ell+1)!\e^1
 \|\partial_t\phi_{\ell+1}\|_0
 +\cdots + \frac{(2k-1)!}{(2k-\ell-1)!}\e^{2k-\ell-1}
 \|\partial_t\phi_{2k-1}\|_0\\
 \leq & \,C_\ell ((1+t)^{\ell-1}+\e^1 (1+t)^\ell+\cdots + \e^{2k-\ell-1}(1+t)^{2k-2})\\
\leq &\,
C_\ell ((\e^{-m})^{\ell-1}+\e^1 (\e^{-m})^\ell+\cdots + \e^{2k-\ell-1}(\e^{-m})^{2k-2}),
\end{align*}
which implies that
\begin{align*}
(\e^{\frac{1}{2}})^{\ell-1}\|\partial_t g^{(\ell)}(\vartheta'_{t,x}\e)\|_0  
\leq &\,
C_\ell(\e^{\frac{1}{2}})^{\ell-1} ((\e^{-m})^{\ell-1}+\e^1 (\e^{-m})^\ell+\cdots + \e^{2k-\ell-1}(\e^{-m})^{2k-2})\\
= &\,
C_\ell ((\e^{\frac{1}{2}-m})^{\ell-1}
+\e^{\frac{1}{2}}(\e^{\frac{1}{2}-m})^{\ell}
+\cdots 
+\e^{\frac{2k-\ell-1}{2}}(\e^{\frac{1}{2}-m})^{2k-2})\\
\leq &\,2 C_\ell (\e^{\frac{1}{2}-m})^{\ell-1}.
\end{align*}
Combining the above inequality with the estimates \eqref{g-ell-L2} and \eqref{g-ell-Linfty}, we get, similarly to \eqref{product-estimate}, that
\begin{align*}
&\|\partial_t (H(\e)-T_{2k-1}(H(\e)))\|_0\\
\leq  &\, \Big\| \partial_t H(\vartheta'_{t,x} \e) \e^{2k}\frac{1}{(2k)!}\sum_\varpi C_\varpi \prod_{1\leq \ell \leq 2k-1}(g^{(\ell)}(\vartheta'_{t,x} \e))^{\varpi_\ell}\Big\|_0\\
&+ \Big\|  H(\vartheta'_{t,x} \e) \e^{2k}\frac{1}{(2k)!}\sum_\varpi C_\varpi \partial_t \Big(\prod_{1\leq \ell \leq 2k-1}(g^{(\ell)})^{\varpi_\ell}\Big)(\vartheta'_{t,x} \e)\Big\|_0\\
\leq   &\,C\Big\|\e^{2k}\prod_{1\leq \ell \leq 2k-1}(g^{(\ell)}(\vartheta'_{t,x} \e))^{\varpi_\ell}\Big\|_0+C\Big\|\e^{2k}\partial_t\Big(\prod_{1\leq \ell \leq 2k-1}(g^{(\ell)})^{\varpi_\ell}\Big)(\vartheta'_{t,x} \e)\Big\|_0\\
\leq   & \,  C \e^k \e^{\frac{1}{2}}.
\end{align*}
Thus \eqref{dtH-e} is proved.

\emph{Step 3. Proof of \eqref{dxH-e}}.
Since $\nabla_x H(\e)$ is continuous in $x$, we commute the mixed partial derivatives of $x$ and $\e$ to get
\begin{align*}
 \nabla_x (H(\e)-T_{2k-1}(H(\e)))=&\,\nabla_x H(\e)-T_{2k-1}(\nabla_x H)(\e)\\
 =&\,\frac{ (\nabla_x H)^{(2k)}(\vartheta''_{t,x} \e)}{(2k)!}\e^{2k}, \quad 0<\vartheta''_{t,x} <1,
\end{align*}
for any fixed $(t,x)\in [0,+\infty)\times\mathbb{R}^3$.
Here,
\begin{align*}
(\nabla_x H)^{(2k)}(\e)=\nabla_x(H^{(2k)}(\e))= \nabla_x H (\e)\sum_\varpi C_\varpi \prod_{1\leq \ell \leq 2k-1}(g^{(\ell)}(\e))^{\varpi_\ell}+ H(\e) \sum_\varpi C_\varpi \nabla_x\Big(\prod_{1\leq \ell \leq 2k-1} (g^{(\ell)}(\e))^{\varpi_\ell}\Big).
\end{align*}
Similar to the deduction of \eqref{dtH-e}, we get
\begin{align*}
&\|\nabla_x (H(\e)-T_{2k-1}(H(\e)))\|_\infty\\
\leq  & \,\Big\| \nabla_x H(\vartheta''_{t,x} \e) \e^{2k}\frac{1}{(2k)!}\sum_\varpi C_\varpi \prod_{1\leq \ell \leq 2k-1}(g^{(\ell)}(\vartheta''_{t,x} \e))^{\varpi_\ell}\Big\|_0\\
&+ \Big\|  H(\vartheta''_{t,x} \e) \e^{2k}\frac{1}{(2k)!}\sum_\varpi C_\varpi \nabla_x\Big(\prod_{1\leq \ell \leq 2k-1}(g^{(\ell)}(\vartheta''_{t,x} \e))^{\varpi_\ell}\Big)\Big\|_0\\
\leq   &\,C\Big\|\e^{2k}\prod_{1\leq \ell \leq 2k-1}(g^{(\ell)}(\vartheta''_{t,x} \e))^{\varpi_\ell}\Big\|_0+C\Big\|\e^{2k}\nabla_x\Big(\prod_{1\leq \ell \leq 2k-1}(g^{(\ell)}(\vartheta''_{t,x} \e))^{\varpi_\ell}\Big)\Big\|_0\\
\leq   & \,  C \e^k \e^{\frac{1}{2}}.
\end{align*}
Thus \eqref{dxH-e} follows.
This completes the proof of Proposition \ref{prop2}.
\end{proof}

\smallskip
\section{$L^2$ Estimates for the Remainders $R$ and $\phi_R$}\label{sec3}
This section is devoted to performing the $L^2$ estimates of the remainders $f=\frac{R}{\sqrt{\mu}}$ and $\phi_R$.
From \eqref{vpbr}, $f$ satisfies
\begin{align}
&\{\partial_t+v\cdot \nabla_x-\nabla_x\phi_0 \cdot \nabla_v\}f+\frac{1}{\varepsilon}Lf\nonumber\\
 &\quad =
-\frac{v-u_0}{\theta_0}\sqrt{\mu}\cdot \nabla_x \phi_R -
\frac{\{\partial_t+v\cdot \nabla_x-\nabla_x\phi_0 \cdot \nabla_v\}\sqrt{\mu}}{\sqrt{\mu}}f\nonumber\\
&\qquad +\varepsilon^{k-1}\Gamma(f,f)+\sum_{i=1}^{2k-1}\varepsilon^{i-1}\Big\{\Gamma\Big(\frac{F_i}
{\sqrt{\mu}},f\Big)+\Gamma\Big(f, \frac{F_i}{\sqrt{\mu}}\Big)\Big\}\nonumber\\
&\qquad +\varepsilon^k \nabla_x \phi_R \cdot \nabla_v f-\e^k\nabla_x \phi_R \cdot \frac{v-u_0}{2\theta}f+\sum_{i=1}^{2k-1}\e^i \frac{\nabla_v F_i}{\sqrt{\mu}}\cdot \nabla_x \phi_R\nonumber\\
&\qquad -\sum_{i=1}^{2k-1}\e^i \nabla_x \phi_i \cdot \Big\{\frac{v-u_0}{2\theta_0}f-\nabla_v f\Big\}
+\varepsilon^{k-1} \frac{A}{\sqrt{\mu}}.\label{eqn4.1}
\end{align}
Take the $L^2$ inner product with $\theta_0 f$ on both sides of \eqref{eqn4.1} to get
\begin{align}\label{L2}
&\frac{1}{2}\frac{\mathrm{d}}{\mathrm{d}t}\|\sqrt{\theta_0}f\|_0^2+\frac{1}{\e}\langle Lf, \theta_0 f\rangle+\iint_{\mathbb{R}^3\times \mathbb{R}^3} v\sqrt{\mu}f \cdot \nabla_x \phi_R \mathrm{d}x\mathrm{d}v\nonumber\\
&\quad   =\iint_{\mathbb{R}^3\times \mathbb{R}^3} u_0 \sqrt{\mu}f \cdot \nabla_x \phi_R \mathrm{d}x\mathrm{d}v+\varepsilon^{k-1}\langle \Gamma(f,f), \theta_0f\rangle \nonumber\\
&\qquad+\Big\langle \frac{1}{2}\{\partial_t+v \cdot \nabla_x\}\theta_0 f-\theta_0 \frac{\{\partial_t+v\cdot \nabla_x-\nabla_x\phi_0 \cdot \nabla_v\}\sqrt{\mu}}{\sqrt{\mu}}f, f\Big\rangle \nonumber\\
&\qquad +\Big\langle \theta_0\sum_{i=1}^{2k-1}\varepsilon^{i-1}\Big\{\Gamma\Big(\frac{F_i}
{\sqrt{\mu}},f\Big)+\Gamma\Big(f, \frac{F_i}{\sqrt{\mu}}\Big)\Big\}, f\Big\rangle \nonumber\\
&\qquad -\e^k \Big\langle\nabla_x \phi_R \cdot \frac{v-u_0}{2}f, f\Big\rangle+\Big\langle \theta_0\sum_{i=1}^{2k-1}\e^i \frac{\nabla_v F_i}{\sqrt{\mu}}\cdot \nabla_x \phi_R, f\Big\rangle\nonumber\\
&\qquad -\Big\langle\sum_{i=1}^{2k-1}\e^i \nabla_x \phi_i \cdot \frac{v-u_0}{2}f, f\Big\rangle
+\varepsilon^{k-1} \Big\langle\frac{A}{\sqrt{\mu}}, \theta_0f\Big\rangle.
\end{align}

We recall the time-dependent functions $\mathcal{I}_1(t,\e)$ and $\mathcal{I}_2(t,\e)$ as:
\begin{align}
&\mathcal{I}_1(t,\e)=\sum_{i=1}^{2k-1}[\e(1+t)]^{i-1}
+\Big(\sum_{i=1}^{2k-1}[\e(1+t)]^{i-1}\Big)^2,\label{I1}\\
&\mathcal{I}_2(t,\e)=\sum_{2k\leq i+j\leq 4k-2}\e^{i+j-2k}(1+t)^{i+j-2}.\label{I2}
\end{align}

To deal with the term $\iint_{\mathbb{R}^3\times \mathbb{R}^3} v\sqrt{\mu}f \cdot \nabla_x \phi_R \mathrm{d}x\mathrm{d}v$ on the left-hand side of \eqref{L2}, we recall the equation of $\e^k\phi_R=\psi_R$ \eqref{psi_R}:
 \begin{align*}
  \Delta \psi_R= H(\e) (e^{\psi_R}-1)+H(\e)-T_{2k-1}(H(\e))-\int_{\mathbb{R}^3}\e^k R\, \mathrm{d}v,
\end{align*}
which implies that
\begin{align}\label{partial-t-psi-R}
  \Delta( \partial_t \psi_R)= \partial_t H(\e) (e^{\psi_R}-1)+H (\e)\partial_t (e^{\psi_R}-1)
  +\partial_t (H(\e)-T_{2k-1}(H(\e)))-\int_{\mathbb{R}^3}\e^k \partial_t R\, \mathrm{d}v.
\end{align}
Since it holds the identity
$$\partial_t[ H(\e) (e^{\psi_R}-1)]\psi_R=\partial_t [H(\e)(\psi_Re^{\psi_R}-e^{\psi_R}+1)]+\partial_t H (\e) (e^{\psi_R}-\psi_R-1),$$
we multiply \eqref{partial-t-psi-R} by $\psi_R$ to get
\begin{align}\label{pre-L2-psi_R}
  \frac{\mathrm{d}}{\mathrm{d}t}&\int_{\mathbb{R}^3}(\psi_Re^{\psi_R}-e^{\psi_R}+1)H(\e) \mathrm{d}x+\frac{1}{2}\frac{\mathrm{d}}{\mathrm{d}t}\int_{\mathbb{R}^3}|\nabla \psi_R|^2 \mathrm{d}x\nonumber\\
  &= -\int_{\mathbb{R}^3}\big[\partial_t (H(\e)-T_{2k-1}(H(\e))) \psi_R
  + (e^{\psi_R}-\psi_R-1)\partial_t H(\e)\big]\mathrm{d}x\nonumber\\
  &\quad\,+ \iint_{\mathbb{R}^3\times \mathbb{R}^3}\e^k \partial_t R \psi_R \mathrm{d}x\mathrm{d}v.
\end{align}
By \eqref{vpbr}$_1$, the last term on the right-hand side of \eqref{pre-L2-psi_R} equals
\begin{align*}
  & \e^{2k} \iint_{\mathbb{R}^3\times \mathbb{R}^3} \partial_t R \phi_R \mathrm{d}x\mathrm{d}v
  = \e^{2k} \iint_{\mathbb{R}^3\times \mathbb{R}^3} -v\cdot \nabla_x R \phi_R \mathrm{d}x\mathrm{d}v
  = \e^{2k} \iint_{\mathbb{R}^3\times \mathbb{R}^3} v \sqrt{\mu}f \cdot\nabla_x\phi_R \mathrm{d}x\mathrm{d}v.
\end{align*}
Hence we get the identity:
\begin{align}\label{L2-psi_R}
  \frac{\mathrm{d}}{\mathrm{d}t}&\int_{\mathbb{R}^3}(\psi_Re^{\psi_R}-e^{\psi_R}+1)H(\e) \mathrm{d}x+\frac{1}{2}\frac{\mathrm{d}}{\mathrm{d}t}\int_{\mathbb{R}^3}|\nabla \psi_R|^2 \mathrm{d}x\nonumber\\
  & = -\int_{\mathbb{R}^3}\big[\partial_t (H(\e)-T_{2k-1}(H(\e))) \psi_R
  + (e^{\psi_R}-\psi_R-1)\partial_t H(\e)\big]\mathrm{d}x\nonumber\\
  &\quad\, + \iint_{\mathbb{R}^3\times \mathbb{R}^3}\e^k v \sqrt{\mu}f \cdot\nabla_x \psi_R \mathrm{d}x\mathrm{d}v.
\end{align}
By \eqref{dtH-e}, the first term on the right-hand side of \eqref{L2-psi_R} is controlled by 
\begin{align*}
&\left|\int_{\mathbb{R}^3}\partial_t (H(\e)-T_{2k-1}(H(\e))) \psi_R \mathrm{d}x\right|
\leq  \|\partial_t (H(\e)-T_{2k-1}(H(\e)))\|_0 \|\psi_R \|_0
\leq  C\e^{2k}\e^{\frac{1}{2}}\|\phi_R\|_0.
\end{align*}
By \eqref{psiR-e} and \eqref{dtH-Linfty}, the second term on the right-hand side of \eqref{L2-psi_R} is controlled by 
\begin{align*}
\left|\int_{\mathbb{R}^3} (e^{\psi_R}-\psi_R-1)\partial_t H(\e) \mathrm{d}x\right|
\leq & \|\partial_t H(\e)\|_\infty 
\int_{\mathbb{R}^3}|\psi_R|^2 \Big(\frac{1}{2!}+\frac{|\psi_R|}{3!}
+\frac{|\psi_R|^2}{4!}+\cdots\Big)\mathrm{d}x\\
\leq & C\Big(\frac{1}{(1+t)^{\frac{16}{15}}} +\e \mathcal{I}_1(t,\e)\Big)
\int_{\mathbb{R}^3}|\psi_R|^2 \cdot e \mathrm{d}x\\
\leq & C\Big(\frac{1}{(1+t)^{\frac{16}{15}}} +\e \mathcal{I}_1(t,\e)\Big)\|\psi_R\|_0^2\\
= & C\e^{2k}\Big(\frac{1}{(1+t)^{\frac{16}{15}}} +\e \mathcal{I}_1(t,\e)\Big)\|\phi_R\|_0^2.
\end{align*}
Thus, \eqref{L2-psi_R} implies that
\begin{align}\label{L2-psiR-estimate}
 & \frac{\mathrm{d}}{\mathrm{d}t}\int_{\mathbb{R}^3}\e^{-2k}H(\e)(\psi_Re^{\psi_R}-e^{\psi_R}+1) \mathrm{d}x+\frac{1}{2}\frac{\mathrm{d}}{\mathrm{d}t}\int_{\mathbb{R}^3}|\nabla \phi_R|^2 \mathrm{d}x\nonumber\\
  &\quad -\iint_{\mathbb{R}^3\times \mathbb{R}^3}  v \sqrt{\mu}f \cdot\nabla_x\phi_R \mathrm{d}x\mathrm{d}v\leq  C\Big(\frac{1}{(1+t)^{\frac{16}{15}}} +\e \mathcal{I}_1(t,\e)+\e^{\frac{1}{2}}\Big)(\|\phi_R\|_0^2+\|\phi_R\|_0).
\end{align}

We now deal with the right-hand side of \eqref{L2} term by term.

(1). For the first term, from \eqref{psi_R}, we have
\begin{align*}
  \iint_{\mathbb{R}^3\times \mathbb{R}^3} u_0 \sqrt{\mu}f \cdot \nabla_x \phi_R \mathrm{d}x\mathrm{d}v
  =&-\int_{\mathbb{R}^3}\Delta \phi_R u_0 \cdot \nabla_x \phi_R\mathrm{d}x\\
  &+ \int_{\mathbb{R}^3}\e^{-k}H(\e)\big(e^{\e^k\phi_R}-1\big) u_0 \cdot \nabla_x \phi_R\mathrm{d}x\\
  &+\int_{\mathbb{R}^3}\e^{-k}(H(\e)-T_{2k-1}(H(\e))) u_0 \cdot \nabla_x \phi_R\mathrm{d}x.
\end{align*}
For the terms on the right-hand side above, we deduce that
\begin{align*}
-\int_{\mathbb{R}^3}\Delta \phi_R u_0 \cdot \nabla \phi_R\mathrm{d}x
=\,&\int_{\mathbb{R}^3} (\nabla \phi_R \cdot \nabla) u_0 \cdot \nabla \phi_R\mathrm{d}x-\frac{1}{2}\int_{\mathbb{R}^3}(\nabla\cdot u_0)|\nabla \phi_R|^2 \mathrm{d}x\\
\leq \,& \frac{C}{(1+t)^{\frac{16}{15}}}\|\nabla_x \phi_R\|_0^2,
\end{align*}
and by \eqref{psiR-e} and mean value theorem again,
\begin{align*}
\int_{\mathbb{R}^3}\e^{-k}H(\e)\big(e^{\e^k\phi_R}-1\big) u_0 \cdot \nabla_x \phi_R\mathrm{d}x
=\,\,& \int_{\mathbb{R}^3}H(\e) e^{\psi_\xi}\phi_R u_0 \cdot \nabla_x \phi_R\mathrm{d}x\\
\leq \,\,& \frac{C}{(1+t)^{\frac{16}{15}}}\| \phi_R\|_0\|\nabla_x \phi_R\|_0,
\end{align*}
and by \eqref{H-e},
\begin{align*}
  \int_{\mathbb{R}^3}\e^{-k}(H(\e)-T_{2k-1}(H(\e))) u_0 \cdot \nabla_x \phi_R\mathrm{d}x 
  \leq & \,\frac{C}{(1+t)^{\frac{16}{15}}}
  \|\e^{-k}(H(\e)-T_{2k-1}(H(\e)))\|_0\|\nabla_x \phi_R\|_0\\
  \leq & \,\frac{C\e^{\frac{1}{2}}}{(1+t)^{\frac{16}{15}}}
  \|\nabla_x \phi_R\|_0.
\end{align*}
Hence, we obtain
\begin{align*}
  \Big|\iint_{\mathbb{R}^3\times \mathbb{R}^3} u_0 \sqrt{\mu}f \cdot \nabla_x \phi_R \mathrm{d}x\mathrm{d}v\Big|
  &\leq  \frac{C}{(1+t)^{\frac{16}{15}}}(\| \phi_R\|^2_0+\|\nabla_x \phi_R\|^2_0+\|\nabla_x \phi_R\|_0).
\end{align*}

(2). Since the term $$\frac{1}{2}\{\partial_t+v \cdot \nabla_x\}\theta_0- \theta_0\frac{\{\partial_t+v\cdot \nabla_x-\nabla_x\phi_0 \cdot \nabla_v\}\sqrt{\mu}}{\sqrt{\mu}}$$ behaves like $|v|^3$ for large $|v|$, we notice that
$(1+|v|^2)^{3/2}f\leq (1+|v|^2)^{-2}h$ for $\beta\geq 7/2$ and get 
\begin{align*}
&\Big\langle \frac{1}{2}\{\partial_t+v \cdot \nabla_x\}\theta_0f- \theta_0\frac{\{\partial_t+v\cdot \nabla_x-\nabla_x\phi_0 \cdot \nabla_v\}\sqrt{\mu}}{\sqrt{\mu}} f, f \Big\rangle\\
\leq  & \; C (\|\partial \rho_0\|_0+\|\partial u_0\|_0+\|\partial \theta_0\|_0+\|\nabla \phi_0\|_0)\big\|(1+|v|^2)^{3/2}f \mathbf{1}_{\{|v|\geq \frac{\kappa}{\sqrt{\e}}\}}\big\|_\infty \|f\|_0\\
&\;+ C (\|\partial \rho_0\|_\infty+\|\partial u_0\|_\infty+\|\partial \theta_0\|_\infty+\|\nabla \phi_0\|_\infty)
\big\|(1+|v|^2)^{3/4}f \mathbf{1}_{\{|v|\leq \frac{\kappa}{\sqrt{\e}}\}}\big\|^2_0\\
\leq &\; C_\kappa \e^2 \|h\|_\infty \|f\|_0+\frac{C}{(1+t)^{\frac{16}{15}}}\big(\|f\|_0^2
+\frac{\kappa^2}{\e}\|(\mathbf{I}-\mathbf{P})f\|_\nu^2\big).
\end{align*}

(3). It is clear to deduce that
\begin{align*}
\left\langle\varepsilon^{k-1} \Gamma(f,f),\theta_0 f\right\rangle \leq C   \varepsilon^{k-1} \|h\|_\infty \|f\|_0^2,
\end{align*}
and
\begin{align*}
&  \Big\langle \theta_0\sum_{i=1}^{2k-1}\varepsilon^{i-1}\Big\{\Gamma\Big(\frac{F_i}
{\sqrt{\mu}},f\Big)+\Gamma\Big(f, \frac{F_i}{\sqrt{\mu}}\Big)\Big\}, 
f\Big\rangle \leq C_\kappa \e \mathcal{I}_1 (t,\e)\|f\|^2_0 + \frac{\kappa^2}{\e}\|(\mathbf{I}-\mathbf{P})f\|^2_\nu,
\end{align*}
and
\begin{align*}
&-\e^k \Big\langle\nabla_x \phi_R \cdot \frac{v-u_0}{2}f, f\Big\rangle+\Big\langle \theta_0\sum_{i=1}^{2k-1}\e^i \frac{\nabla_v F_i}{\sqrt{\mu}}\cdot \nabla_x \phi_R, f\Big\rangle\\
&\,\quad\leq C\e^k \|h\|_\infty \|\nabla_x \phi_R\|_0 \|f\|_0
+C\e \mathcal{I}_1(t,\e) (\|\nabla_x \phi_R\|^2_0 +\|f\|^2_0),
\end{align*}
and
\begin{align*}
&\Big\langle \sum_{i=1}^{2k-1}\e^{i} \nabla_x \phi_i \cdot \frac{v-u_0}{2}f, f\Big\rangle
\leq \;C_\kappa\e \mathcal{I}_1(t,\e)\|f\|_0^2+ \frac{\kappa^{2}}{\e}\|(\mathbf{I}-\mathbf{P})f\|^2_\nu,
\end{align*}
and
\begin{align*}
\e^{k-1}\Big\langle \frac{A}{\sqrt{\mu}}, \theta_0 f\Big\rangle \leq C \mathcal{I}_2(t,\e) \e^{k-1}\|f\|_0.
\end{align*}

It is well-known that there exists a $\delta_0>0$ such that $\langle Lf, f\rangle\geq \delta_0 \|(\mathbf{I-P})f\|^2_\nu$.
Choosing $\kappa$ sufficiently small, and combining \eqref{L2} with \eqref{L2-psiR-estimate} yields
\begin{align}\label{low}
&\frac{\mathrm{d}}{\mathrm{d}t}\int_{\mathbb{R}^3}\e^{-2k}H(\e)(\psi_Re^{\psi_R}-e^{\psi_R}+1) \mathrm{d}x+\frac{1}{2}\frac{\mathrm{d}}{\mathrm{d}t}\big(\|\nabla \phi_R\|_0^2 +\|\sqrt{\theta_0}f\|_0^2\big)+\frac{\delta_0}{2\e}\theta_M \|(\mathbf{I-P})f\|^2_\nu\nonumber\\
& \quad \leq C\Big(\frac{1}{(1+t)^{\frac{16}{15}}} +\e \mathcal{I}_1(t,\e)+\e^{\frac{1}{2}}\Big)(\|\phi_R\|_0^2+\|\phi_R\|_0
+\|\nabla_x\phi_R\|_0^2+\|\nabla_x\phi_R\|_0+\|f\|_0^2)
\nonumber\\
&\qquad +C\mathcal{I}_2(t,\e) \e^{k-1}\|f\|_0+ C\big(\e^2 \|h\|_\infty \|f\|_0+\e^{k-1}\|h\|_\infty \|f\|_0^2
+\e^k \|h\|_\infty\|f\|_0\|\nabla_x\phi_R\|_0\big).
\end{align}

\smallskip

\section{$W^{1,\infty}$ Estimates for the Remainder $R$}\label{sec4}
We first present an important lemma as below.
For any function $\phi\in L^\infty([0, T];C^{2,\alpha}(\mathbb{R}^3))$, we define the characteristics $[X(\tau;t,x,v), V(\tau;t,x,v)]$ passing through $(t,x,v)$ such that
\begin{align*}
&\frac{\mathrm{d}X(\tau;t,x,v)}{\mathrm{d} \tau}=V(\tau;t,x,v),\,\,\,X(t;t,x,v)=x,\\
&\frac{\mathrm{d}V(\tau;t,x,v)}{\mathrm{d} \tau}=-\nabla_x \phi(\tau,X(\tau;t,x,v)),\,\,\,V(t;t,x,v)=v.
\end{align*}
For simplify, we shall denote $[X(\tau;t,x,v), V(\tau;t,x,v)]$ by $[X(\tau),V(\tau)]$.
\begin{lemma}\label{1}
Recall that $h(t,x,v)= w(v) \frac{R(t,x,v)}{\sqrt{\mu_M(v)}}$. Assume that
$0 \leq T \leq \e^{-m}$, $0<m<\frac{1}{2}$, and
\begin{align}\label{assume}
\sup_{0\leq \tau \leq T}\e^k \|h(\tau)\|_{W^{1,\infty}}\leq \e^{\frac{1}{2}}.
\end{align}
Then, it holds
\begin{align}\label{DxX}
\sup_{0\leq \tau \leq T}\{\|\partial_x X(\tau)\|_\infty+\|\partial_v V(\tau)\|_\infty\}\leq C.
\end{align}
Moreover, there exists a sufficiently small $T_0\in (0, T]$ such that for each $0\leq \tau \leq t \leq T_0$, we have
\begin{align}
&\frac{1}{2}|t-\tau|^3\leq \Big|\det \Big(\frac{\partial X(\tau)}{\partial v}\Big)\Big|\leq 2|t-\tau|^3,\label{1-shi}\\
&|\partial_v X(\tau)|\leq 2 |t-\tau|,\label{2-shi}\\
& \frac{1}{2}\leq \Big|\det \Big(\frac{\partial V(\tau)}{\partial v}\Big)\Big|\leq 2,\,\,\,
\frac{1}{2}\leq \Big|\det \Big(\frac{\partial X(\tau)}{\partial x}\Big)\Big|\leq 2,\label{3-shi}\\
&\sup_{0\leq \tau\leq T_0,\,x_0 \in \mathbb{R}^3,\,|v|\leq N}\Big(\int_{\{|x-x_0|\leq N\}}( |\partial_{xv}X(\tau)|^2+  |\partial_{vv}X(\tau)|^2   )\mathrm{d}x\Big)^{\frac{1}{2}}\leq C_N, \,\,\,\mathrm{for}\,\,N\geq 1.\label{4-shi}
\end{align}
\end{lemma}
\begin{proof}
For any $\alpha\in (0,\frac{1}{2})$, $H^2(\mathbb{R}^3)\hookrightarrow C^{0,\alpha}(\mathbb{R}^3)$. By \eqref{psiR-e}, we get
\begin{align*}
\|\phi\|_{C^{0,\alpha}}\leq & \,\|\phi_0\|_{C^{0,\alpha}}
+ \e \|\phi_1\|_{C^{0,\alpha}}+\cdots +\e^{2k-1}\|\phi_{2k-1}\|_{C^{0,\alpha}}+\e^k \|\phi_R\|_{C^{0,\alpha}}\\
\leq & \,C \left(\|\phi_0\|_{2}
+ \e \|\phi_1\|_{2}+\cdots +\e^{2k-1}\|\phi_{2k-1}\|_{2}+\e^k \|\phi_R\|_{2}
\right)\\
\leq & \,C \big(1
+ \e +\cdots +\e^{2k-1}(1+t)^{2k-2}+\e^k\big(\e^{\frac{1}{2}}+\|f\|_0\big)
\big)\\
\leq & \,C +C \e \mathcal{I}_1(t,\e)\leq C.
\end{align*}
From \eqref{vpb}$_2$, we know that
\begin{align*}
-(I-\Delta) \phi= &\,e^\phi-\phi-1-(\rho-1)\\
= &\,e^\phi-\phi-1-\Big(\rho_0-1+\e\rho_1+\cdots +\e^{2k-1}\rho_{2k-1}+\int_{\mathbb{R}^3} \e^k R\,\mathrm{d}v\Big).
\end{align*}
Hence,
\begin{align*}
\|(I-\Delta) \phi\|_{C^{0,\alpha}}\leq  &\,\|e^\phi-\phi-1\|_{C^{0,\alpha}}+\Big(\|\rho_0-1\|_{C^{0,\alpha}}
+\cdots +\e^{2k-1}\|\rho_{2k-1}\|_{C^{0,\alpha}}+\Big\|\int_{\mathbb{R}^3} \e^k R\,\mathrm{d}v\Big\|_{C^{0,\alpha}}\Big)\\
\leq  &\,e^{\|\phi\|_{C^{0,\alpha}}}+ C\left(\|\rho_0-1\|_2
+\cdots +\e^{2k-1}\|\rho_{2k-1}\|_2+\e^k \|h\|_{W^{1,\infty}}\right)\\
\leq &\, e^C + C (1+C\e\mathcal{I}_1(t,\e)+\e^{\frac{1}{2}})\leq C.
\end{align*}
Since $(I-\Delta)^{-1}\in \Psi^{-2}$, there exists a constant $C>0$, independent of $\phi$, such that
\begin{align}\label{D2phi}
\|\phi\|_{C^{2,\alpha}}\leq C \|(I-\Delta) \phi\|_{C^{0,\alpha}}\leq C.
\end{align}
Noting that for $\partial=\partial_x$ or $\partial_v$, it holds
\begin{align*}
\frac{\mathrm{d}^2\partial X(\tau)}{\mathrm{d}\tau^2}=-\nabla_{xx}\phi(\tau, X(\tau))\partial X.
\end{align*}
By using $\|\nabla_{xx}\phi\|_\infty\leq C$ and integrating in time, one deduces that
\begin{align*}
\sup_{0\leq \tau \leq T}\{\|\partial_x X(\tau)\|_\infty+\|\partial_v V(\tau)\|_\infty\}\leq C.
\end{align*}

Moreover,
consider the Taylor expansions of 
$\frac{\partial X(\tau)}{\partial v}$, $\frac{\partial X(\tau)}{\partial x}$ and $\frac{\partial V(\tau)}{\partial v}$ in $\tau$ around $t$:
\begin{align*}
\frac{\partial X(\tau)}{\partial v}=&\,(\tau-t)I+\frac{(\tau-t)^2}{2}\frac{\mathrm{d}^2}{\mathrm{d}\tau^2}\frac{\partial X(\bar{\tau})}{\partial v},\\
\frac{\partial X(\tau)}{\partial x}=&\,I+\frac{(\tau-t)^2}{2}\frac{\mathrm{d}^2}{\mathrm{d}\tau^2}\frac{\partial X(\bar{\tau})}{\partial x},\\
\frac{\partial V(\tau)}{\partial v}=&\,I+(\tau-t)\frac{\mathrm{d}}{\mathrm{d}\tau}\frac{\partial V(\bar{\tau})}{\partial v}\\
=& \,I- (\tau-t)\nabla_{xx} \phi(\bar{\tau})
\frac{\partial X(\bar{\tau})}{\partial v},
\end{align*}
for $\tau\leq \bar{\tau}\leq t$.
Note that
\begin{align*}
\Big|\frac{\mathrm{d}^2}{\mathrm{d}\tau^2}\frac{\partial X(\bar{\tau})}{\partial v}\Big|=
\Big|\frac{\partial}{\partial v}\nabla_x \phi(\bar{\tau}, X(\bar{\tau}))\Big|
\leq  |\nabla_{xx} \phi(\bar{\tau})|\Big|\frac{\partial X(\bar{\tau})}{\partial v}\Big|
\leq  C \sup_{0\leq \bar{\tau}\leq T}\Big\|\frac{\partial X(\bar{\tau})}{\partial v}\Big\|_\infty,
\end{align*}
and 
\begin{align*}
\Big|\frac{\mathrm{d}^2}{\mathrm{d}\tau^2}\frac{\partial X(\bar{\tau})}{\partial x}\Big|=
\Big|\frac{\partial}{\partial x}\nabla_x \phi(\bar{\tau}, X(\bar{\tau}))\Big|
\leq  |\nabla_{xx} \phi(\bar{\tau})|\Big|\frac{\partial X(\bar{\tau})}{\partial x}\Big|
\leq C \sup_{0\leq \bar{\tau}\leq T}\Big\|\frac{\partial X(\bar{\tau})}{\partial x}\Big\|_\infty,
\end{align*}
we thus choose $T_0$ sufficiently small such that for $0\leq \tau \leq t \leq T_0$, \eqref{1-shi}, \eqref{2-shi} and \eqref{3-shi} hold.

To prove \eqref{4-shi}, we notice that
\begin{align*}
\frac{\mathrm{d}^2}{\mathrm{d}\tau^2}\partial \partial_v X(\tau)=& \partial \partial_v \Big(\frac{\mathrm{d}^2}{\mathrm{d}\tau^2}X(\tau)\Big)\\
=& -\partial \partial_v (\nabla_x \phi(\tau, X(\tau)))\\
=& -\partial (\nabla_x^2 \phi(\tau, X(\tau))\partial_v X(\tau))\\
=& -\nabla_x^3 \phi(\tau, X(\tau))\partial_v X(\tau) \partial X(\tau)\\
&-\nabla_x^2 \phi(\tau, X(\tau))\partial \partial_v X(\tau).
\end{align*}
By integrating the above equation twice in time, we see that
\begin{align*}
\|\partial \partial_v X(\tau)\|_{L^2(\{|x-x_0|\leq N\})}
& \leq\,\frac{T_0^2 \|\nabla_{x,v}X\|_\infty^2}{2}\sup_{0\leq \tau\leq T_0}\|\nabla_x^3 \phi(\tau, X(\tau))\|_{L^2(\{|x-x_0|\leq N\})}\\
&\quad\,+ \frac{T_0^2 \|\nabla_{x,v}\phi\|_\infty^2}{2}\sup_{0\leq \tau\leq T_0}\|\partial \partial_v X(\tau)\|_{L^2(\{|x-x_0|\leq N\})}.
\end{align*}
For $|v|\leq N$, we easily deduce that
\begin{align*}
\{|x-x_0|\leq N\}\cap \{|v|\leq N\}\subset \{|X(\tau)-x_0|\leq CN\}
\end{align*}
for $0\leq \tau\leq T_0\ll 1$ and $N\geq 1$.
According to \eqref{DxX}, \eqref{2-shi}, \eqref{3-shi} and \eqref{D2phi}, we get
\begin{align*}
&\sup_{0\leq \tau\leq T_0,\,x_0 \in \mathbb{R}^3,\,|v|\leq N}\|\partial \partial_v X(\tau)\|_{L^2(\{|x-x_0|\leq N\})}\\
\leq &\, C \sup_{0\leq \tau\leq T_0,\,x_0 \in \mathbb{R}^3,\,|v|\leq N}\|\nabla_x^3 \phi(\tau, X(\tau))\|_{L^2(\{|X(\tau)-x_0|\leq CN\})}\\
\leq &\, C \sup_{0\leq \tau\leq T_0,\,x_0 \in \mathbb{R}^3,\,|v|\leq N}\|\nabla_x^3 \phi(\tau)\|_{L^2(\{|x-x_0|\leq CN\})}
\Big|\det \Big\{\frac{\partial X(\tau)}{\partial x}\Big\}\Big|^{-\frac{1}{2}}\\
\leq &\, C \sup_{0\leq \tau\leq T_0,\,x_0 \in \mathbb{R}^3,\,|v|\leq N}\|\nabla_x^3 \phi(\tau)\|_{L^2(\{|x-x_0|\leq CN\})}.
\end{align*}
To control $\|\nabla_x^3 \phi(\tau)\|_{L^2(\{|x-x_0|\leq CN\})}$, we define $\chi(x)$ as a smooth cutoff function of $x$ on the domain $\{|x-x_0|\leq CN+1\}$.
By using Calder\'{o}n-Zygmund's inequality, we get
\begin{align*}
  \|\nabla_x^3 \phi(\tau)\|_{L^2(\{|x-x_0|\leq CN\})}
  \leq & \,\|\partial_{ij}\partial_k (\chi \phi)\|_{L^2}\\
  \leq & \,C \|\Delta D_x (\chi \phi)\|_{L^2}\\
  \leq & \,C \|\chi D_x \phi e^\phi\|_{L^2}+C \|\chi (D_x \rho_0+\cdots + \e^{2k-1}D_x \rho_{2k-1})\|_{L^2}\\
  &\,+ C \Big\|\chi \e^k \int_{\mathbb{R}^3}\frac{\sqrt{\mu_M}}{w}D_x h \mathrm{d}v \Big\|_{L^2}
  +C \Big\|\sum_{|\alpha|+|\beta|=3,\,|\beta|\leq 2}\partial^\alpha \chi \partial^\beta \phi \Big\|_{L^2}\\
  \leq &\, C N^{\frac{3}{2}}\|\phi\|_{C^{2,\alpha}}+C+C\e\mathcal{I}_1
  +\e^k N^{\frac{3}{2}}\|h\|_{W^{1,\infty}}\\
  \leq &\,C N^{\frac{3}{2}}.
\end{align*}
Then \eqref{4-shi} follows.
This completes the proof of Lemma \ref{1}.
\end{proof}

\subsection{$L^\infty$ Estimates}

Recall the global Maxwellian $\mu_M$:
\begin{align*}
\mu_M(v)=\frac{1}{(2\pi \theta_M)^{\frac{3}{2}}}\exp\Big\{-\frac{|v|^2}{2\theta_M}\Big\},
\end{align*}
and the local Maxwellian $\mu$:
\begin{align*}
\mu(t,x,v)=\frac{\rho_0(t,x)}{(2\pi \theta_0(t,x))^{\frac{3}{2}}}\exp\Big\{-\frac{|v-u_0(t,x)|^2}
{2\theta_0(t,x)}\Big\}, \quad \theta_0(t,x)= K \rho_0^{\frac{2}{3}}(t,x).
\end{align*}
Now, we define 
\begin{align*}
L_M g=-\frac{1}{\sqrt{\mu_M}}\{Q(\mu, \sqrt{\mu_M}g)+Q( \sqrt{\mu_M}g, \mu)\}=(\nu(\mu)+K_M)g,
\end{align*}
and
\begin{align*}
K_{M,w}g=wK_M \Big(\frac{g}{w}\Big).
\end{align*}
According to \eqref{vpbr}, the equation of $h=\frac{wR}{\sqrt{\mu_M}}$ reads
\begin{align}\label{weighted remainder eqn}
&\Big\{\partial_t+v\cdot \nabla_x-\nabla_x \phi \cdot \nabla_v+ \frac{\nu(\mu)}{\varepsilon}\Big\}h+ \frac{1}{\varepsilon}K_{M,w}h\nonumber\\
&\quad=\,\frac{\varepsilon^{k-1}w}{\sqrt{\mu_M}}
Q\Big(\frac{h\sqrt{\mu_M}}{w},\frac{h\sqrt{\mu_M}}{w}\Big)
+ \nabla_x \phi\cdot  \frac{w}{\sqrt{\mu_M}}\nabla_v \Big(\frac{\sqrt{\mu_M}}{w}\Big)h\nonumber\\
&\qquad+\sum_{i=1}^{2k-1}\frac{\e^{i-1}w}{\sqrt{\mu_M}}
\Big\{Q\Big(F_i, \frac{h\sqrt{\mu_M}}{w}\Big)+Q\Big(\frac{h\sqrt{\mu_M}}{w}, F_i\Big) \Big\}\nonumber\\
&\qquad+\nabla_x \phi_R \cdot \frac{w}{\sqrt{\mu_M}}\nabla_v \Big(\mu+\sum_{i=1}^{2k-1}\e^iF_i\Big) +\varepsilon^{k-1} \frac{w}{\sqrt{\mu_M}}A,
\end{align}
where $\phi=\phi_0+\cdots +\e^{2k-1}\phi_{2k-1}+\e^k \phi_R$. 
We have 
\begin{lemma}\label{lemma-h-Linfty}
  Assume that \eqref{assume} is valid. There exists a $T_0>0$ such that for $0\leq T_0\leq T\leq \e^{-m}$ $(0<m<\frac{1}{2})$ and $\e$ sufficiently small,
\begin{align*}
  \sup_{0\leq s \leq T_0}\{\e^{\frac{3}{2}}\|h(s)\|_\infty\}
  \leq &\,C \{\|\e^{\frac{3}{2}}h_0\|_\infty+\sup_{0\leq s \leq T}\|f(s)\|_0+\e^3\},\\
  \e^{\frac{3}{2}}\|h(T_0)\|_\infty\leq &\, \frac{1}{2}\|\e^{\frac{3}{2}}h_0\|_\infty+C \sup_{0\leq s \leq T}\|f(s)\|_0+C\e^3.
\end{align*}
\end{lemma}
\begin{proof}
 Along the trajectory $[X(\tau), V(\tau)]$, it holds
\begin{align}\label{h-formula}
h(t,x,v) &= \exp\Big\{-\frac{1}{\varepsilon}\int_{0}^{t}\nu(\tau)\mathrm{d}\tau\Big\}
h(0,X(0;t,x,v),V(0;t,x,v)) \nonumber\\
&\quad - \int_{0}^{t} \exp\Big\{-\frac{1}{\varepsilon}\int_{s}^{t}\nu(\tau)\mathrm{d}\tau\Big\} \frac{1}{\varepsilon}K_{M,w}h(s,X(s),V(s))\mathrm{d}s\nonumber \\
&\quad + \int_{0}^{t} \exp\Big\{-\frac{1}{\varepsilon}\int_{s}^{t}\nu(\tau)\mathrm{d}\tau\Big\} \frac{\varepsilon^{k-1}w}{\sqrt{\mu_M}} Q\Big(\frac{h\sqrt{\mu_M}}{w},\frac{h\sqrt{\mu_M}}{w}\Big)
(s,X(s),V(s))\mathrm{d}s \nonumber\\
&\quad + \int_{0}^{t} \exp\Big\{-\frac{1}{\varepsilon}\int_{s}^{t}\nu(\tau)\mathrm{d}\tau\Big\} \sum_{i=1}^{2k-1}\frac{\varepsilon^{i-1}w}{\sqrt{\mu_M}} Q\Big(F_{i},\frac{h\sqrt{\mu_M}}{w}\Big)(s,X(s),V(s))\mathrm{d}s \nonumber\\
&\quad + \int_{0}^{t} \exp\Big\{-\frac{1}{\varepsilon}\int_{s}^{t}\nu(\tau)\mathrm{d}\tau\Big\} \sum_{i=1}^{2k-1}\frac{\varepsilon^{i-1}w}{\sqrt{\mu_M}} Q\Big(\frac{h\sqrt{\mu_M}}{w},F_{i}\Big)(s,X(s),V(s))\mathrm{d}s \nonumber\\
&\quad + \int_{0}^{t} \exp\Big\{-\frac{1}{\varepsilon}\int_{s}^{t}\nu(\tau)\mathrm{d}\tau\Big\} \nabla_{x}\phi \cdot \frac{w}{\sqrt{\mu_M}} \nabla_{v}\Big(\frac{\sqrt{\mu_M}}{w}\Big)h(s,X(s),V(s))\mathrm{d}s \nonumber\\
&\quad + \int_{0}^{t} \exp\Big\{-\frac{1}{\varepsilon}\int_{s}^{t}\nu(\tau)\mathrm{d}\tau\Big\} \nabla_{x}\phi_{R} \cdot \frac{w}{\sqrt{ \mu_M}} \nabla_{v}\Big(\mu + \sum_{i=1}^{2k-1}\varepsilon^{i}F_{i}\Big)(s,X(s),V(s))\mathrm{d}s \nonumber\\
&\quad + \int_{0}^{t} \exp\Big\{-\frac{1}{\varepsilon}\int_{s}^{t}\nu(\tau)\mathrm{d}\tau\Big\} \frac{\varepsilon^{k-1}w}{\sqrt{\mu_M}} A(s,X(s),V(s))\mathrm{d}s.
\end{align}
Since $\nu(\mu)\sim \int |v-u|\mu \mathrm{d}v\sim (1+|v|)\rho_0(t,x)\geq 2\nu_0>0$, we see that
\begin{align*}
\int_{0}^{t} \exp\Big\{-\frac{1}{\varepsilon}\int_{s}^{t}\nu(\tau)\mathrm{d}\tau\Big\}\mathrm{d}s
\leq &\,e^{-\frac{\nu_0}{\e}(t-s)} \int_{0}^{t} \exp\Big\{-\frac{1}{2\varepsilon}\int_{s}^{t}\nu(\tau)\mathrm{d}\tau\Big\}\mathrm{d}s\\
\leq &\, e^{-\frac{\nu_0}{\e}(t-s)} \int_{0}^{t} \exp\Big\{-\frac{1}{2\varepsilon}\int_{s}^{t}\nu(\tau)\mathrm{d}\tau\Big\}
\frac{\nu}{2\e}\cdot\frac{2\e}{\nu}\mathrm{d}s\\
\leq &\, C \e (1+|v|)^{-1}e^{-\frac{\nu_0}{\e}(t-s)}.
\end{align*}

Now, we handle the terms on the right-hand side of \eqref{h-formula}.

For the sixth line in \eqref{h-formula}, by using 
\eqref{D2phi}, it can be controlled by
$C \e e^{-\frac{\nu_0 t}{2\e}}\sup_{0\leq s\leq t}\big\{e^{\frac{\nu_0 s}{2\e}}\|h(s)\|_\infty\big\}$. 

For the seventh line in \eqref{h-formula}, we investigate $\|\nabla_x \phi_R\|_\infty$ as follows. 
Recall the equation \eqref{psi_R} for $\psi_R=\e^k \phi_R$, it is equivalent to
\begin{align}\label{psiR-modified}
  -(H(\e)-\Delta)\psi_R=H(\e)(e^{\psi_R}-1-\psi_R)+H(\e)-T_{2k-1}(H(\e))
  -\int_{\mathbb{R}^3}\e^k R\, \mathrm{d}v.
\end{align}
Noting that $\phi_0,\ldots, \phi_{2k-1}$ are smooth functions satisfying 
$$\sup_{t\leq \e^{-m}}\|\phi_0+\cdots+ \e^{2k-1}\phi_{2k-1}\|_\infty\leq C+C\e\sup_{t\leq \e^{-m}}\mathcal{I}_1(t,\e)\leq C,$$
thus $H(\e)-\Delta$ is an elliptic pseudo-differential operator. From Lax-Milgram's Theorem, the inverse operator $(H(\e)-\Delta)^{-1}:L^2\rightarrow H^1$ exists as a bounded linear operator, thus $(H(\e)-\Delta)^{-1}\in \Psi^{-2}$ is a parametrix of $H(\e)-\Delta$. For any $p>3$, we have
\begin{align*}
\|\psi_R\|_{W^{2,p}(\mathbb{R}^3)}\leq C \|(H(\e)-\Delta)\psi_R\|_{L^p(\mathbb{R}^3)}.
\end{align*}
Recall \eqref{psiR-e} that $\|\psi_R\|_\infty\leq C \e^k$. By Taylor's formula, 
for each $x\in \mathbb{R}^3$, there exists a function $\psi_\xi\in L^\infty(\mathbb{R}^3)$ with $\|\psi_\xi\|_\infty\leq C \e^k$ such that
$e^{\psi_R}-1-\psi_R=\frac{e^{\psi_\xi}}{2}\psi_R^2$.
By interpolation, $H(\e)\in L^1\cap L^\infty\subset L^p$, we also notice by \eqref{H-e} that
$\|H(\e)-T_{2k-1}(H(\e))\|_{L^p}\leq C \e^{k+\frac{1}{2}}$.
Hence we get
\begin{align*}
\|\psi_R\|_{W^{2,p}(\mathbb{R}^3)}\leq&\, C \Big\|H(\e)(e^{\psi_R}-1-\psi_R)+H(\e)-T_{2k-1}(H(\e))
  -\int_{\mathbb{R}^3}\e^k R\, \mathrm{d}v\Big\|_{L^p(\mathbb{R}^3)}\\
  \leq &\, C \Big\|H(\e)\frac{e^{\psi_\xi}}{2}\psi_R^2\Big\|_{L^p(\mathbb{R}^3)}
  +C\|H(\e)-T_{2k-1}(H(\e))\|_{L^p(\mathbb{R}^3)}
  +C\e^k\Big\|\int_{\mathbb{R}^3} R\, \mathrm{d}v\Big\|_{L^p(\mathbb{R}^3)}\\
  \leq &\, C\e^{2k}\|H(\e)\|_{L^p(\mathbb{R}^3)}+C \e^{k+\frac{1}{2}}
  +C\e^k\Big\|\int_{\mathbb{R}^3} h \frac{\sqrt{\mu_M}}{w} \mathrm{d}v\Big\|_{L^p(\mathbb{R}^3)}\\
  \leq &\, C\e^{2k}+C \e^{k+\frac{1}{2}}+C\e^k \|h\|_{L^\infty(\mathbb{R}^3\times \mathbb{R}^3)},
\end{align*}
and thus for some $0<\bar{\alpha}<1$,
\begin{align}\label{DphiR-Linfty}
  \|\nabla_x \phi_R\|_{L^\infty}\leq \|\phi_R\|_{C^{1,\bar{\alpha}}}\leq C \|\phi_R\|_{W^{2,p}}\leq 
  C \e^{\frac{1}{2}}+ C\|h\|_{L^\infty}.
\end{align}
It implies that the seventh line in \eqref{h-formula} is controlled by $$C\e^{\frac{3}{2}}+ C\e e^{-\frac{\nu_0 t}{2\e}}\sup_{0\leq s\leq t}\big\{e^{\frac{\nu_0 s}{2\e}}\|h(s)\|_\infty\big\}.$$

For  the other terms on the right-hand side of \eqref{h-formula},  their $L^\infty$ estimates  are similar to \cite{Juhi} (see  pp. 489--493), where $L^\infty$ interplays with $L^2$. For brevity, we omit the details here. 

In summary, 
we obtain, under the  assumption \eqref{assume}, that
\begin{align*}
\sup_{0\leq s\leq T_0}\big\{e^{\frac{\nu_0 s}{2\e}}\|h(s)\|_\infty\big\}
\leq C\|h_0\|_\infty+\frac{C_N}{\e^{\frac{3}{2}}}e^{\frac{\nu_0 T_0}{2\e}}\sup_{0\leq s\leq T}\|f(s)\|_0+ C \e^{\frac{3}{2}}e^{\frac{\nu_0 T_0}{2\e}}+C \e^{k}e^{\frac{\nu_0 T_0}{2\e}}.
\end{align*}
Letting $s=T_0$ and multiplying the result by $\e^{\frac{3}{2}}e^{-\frac{\nu_0 T_0}{2\e}}$, for sufficiently small $\e$, we deduce that
\begin{align}\label{order-3/2}
  \e^{\frac{3}{2}}\|h(T_0)\|_\infty\leq \frac{1}{2}\|\e^{\frac{3}{2}}h(0)\|_\infty+C \sup_{0\leq s \leq T}\|f(s)\|_0+C\e^3,
\end{align}
since $k\geq 6$. Thus, the proof of Lemma \ref{lemma-h-Linfty} is complete.
\end{proof}

\subsection{$W^{1, \infty}$ Estimates}
To confirm the assumption \eqref{assume}, we further investigate the $L^\infty$ norm of $D_x h$. Take $D_x$ of \eqref{weighted remainder eqn} to get
\begin{align}\label{dx-weighted-remainder}
\Big\{\partial_t+v\cdot \nabla_x-\nabla_x \phi \cdot \nabla_v+ \frac{\nu(\mu)}{\varepsilon}\Big\}(D_xh)
=& \,\,\nabla_x(D_x\phi)\cdot \nabla_v h-\frac{D_x \nu(\mu)}{\e}h\nonumber\\
&- \frac{1}{\varepsilon}D_x(K_{M,w}h)+\frac{\varepsilon^{k-1} w}{\sqrt{\mu_M}}D_x
\Big(Q\Big(\frac{h\sqrt{\mu_M}}{w},\frac{h\sqrt{\mu_M}}{w}\Big)\Big)
\nonumber\\
&+\sum_{i=1}^{2k-1}\frac{\e^{i-1}w}{\sqrt{\mu_M}}D_x\Big\{
Q\Big(F_i, \frac{h\sqrt{\mu_M}}{w}\Big)+ Q\Big( \frac{h\sqrt{\mu_M}}{w},F_i\Big)\Big\}\nonumber\\
&+D_x \Big(\nabla_x \phi \cdot \frac{w}{\sqrt{\mu_M}}\nabla_v \Big(\frac{\sqrt{\mu_M}}{w}\Big)h\Big)\nonumber\\
&+D_x\Big(\nabla_x \phi_R \cdot \frac{w}{\sqrt{\mu_M}}\nabla_v \Big(\mu+\sum_{i=1}^{2k-1}\e^iF_i\Big)\Big)\nonumber\\
&+\varepsilon^{k-1} \frac{w}{\sqrt{\mu_M}}D_x A.
\end{align}
\begin{lemma}\label{lemma-Dh-Linfty}
For $T_0$ obtained in Lemma \ref{lemma-h-Linfty}, and for all sufficiently small $\e>0$, it holds 
\begin{align*}
\sup_{0\leq s\leq T_0}\{\e^5 \|D_{x,v}h(s)\|_\infty\}\leq & \, C\e^5\{\|D_{x,v}h(0)\|_\infty+\|(1+|v|)h(0)\|_\infty\}
+C\e^2 \|h(0)\|_\infty\\
&+C\e^{\frac{1}{2}}\sup_{0\leq s\leq T}\|f(s)\|_0+C \e^{\frac{7}{2}},\\
\e^5 \|D_{x,v}h(T_0)\|_\infty\leq & \, \frac{1}{2}\e^5\{\|D_{x,v}h(0)\|_\infty+\|(1+|v|)h(0)\|_\infty\}
+\frac{1}{2}\e^2 \|h(0)\|_\infty\\
&+C\e^{\frac{1}{2}}\sup_{0\leq s\leq T}\|f(s)\|_0+C \e^{\frac{7}{2}}.
\end{align*}
\end{lemma}
\begin{proof}
 According to \eqref{D2phi}, we see that $\|D_x^2 \phi\|_\infty\leq C$. Along the trajectory, the contribution of the fourth line on the right-hand side of \eqref{dx-weighted-remainder} is 
$$C\e e^{-\frac{\nu_0 t}{2\e}}\sup_{0\leq s\leq t}\{e^{\frac{\nu_0 s}{2\e}}\|h\|_{W^{1,\infty}}\}.$$
We next estimate $\|D^2_x \phi_R\|_\infty$ to deal with the fifth line on the right-hand side of \eqref{dx-weighted-remainder}.

Recall the equation \eqref{psiR-modified} for $\psi_R:=\e^k \phi_R$,
\begin{align*}
  -(H(\e)-\Delta)\psi_R=H(\e)(e^{\psi_R}-1-\psi_R)+H(\e)-T_{2k-1}(H(\e))
  -\int_{\mathbb{R}^3}\e^k R\, \mathrm{d}v,
\end{align*}
where $H(\e)$ is defined in \eqref{Hee}.
 Since $(H(\e)-\Delta)^{-1}\in \Psi^{-2}$, we have for any $0<\tilde{\alpha}<1$, 
\begin{align*}
\|\psi_R\|_{C^{2,\tilde{\alpha}}}
\leq &\,C \Big\|H(\e)(e^{\psi_R}-1-\psi_R)+H(\e)-T_{2k-1}(H(\e))
  -\int_{\mathbb{R}^3}\e^k R\, \mathrm{d}v\Big\|_{C^{0,\tilde{\alpha}}}\\
\leq &\, C \|H(\e)(e^{\psi_R}-1-\psi_R)\|_{W^{1,\infty}}
+C\|H(\e)-T_{2k-1}(H(\e))\|_{W^{1,\infty}}+C\Big\|\int_{\mathbb{R}^3}\e^k R\, \mathrm{d}v\Big\|_{W^{1,\infty}}. 
\end{align*} 
By \eqref{psiR-e}: $\|\psi_R\|_\infty\leq C \e^{k}$ and \eqref{DphiR-Linfty}:
$\|\nabla_x \psi_R\|_\infty\leq C \e^{k+\frac{1}{2}}+C\e^k \|h\|_\infty$, it holds
\begin{align*}
  &\|H(\e)(e^{\psi_R}-1-\psi_R)\|_{W^{1,\infty}}\\
  =&\,\|H(\e)(e^{\psi_R}-1-\psi_R)\|_\infty
  +\|\nabla_x H(\e)(e^{\psi_R}-1-\psi_R)\|_\infty
  +\|H(\e)(e^{\psi_R}-1)\nabla_x \psi_R\|_\infty\\
  \leq &\, \frac{1}{2}(\|H(\e)\|_\infty+\|\nabla_xH(\e)\|_\infty)\big\|e^{|\psi_R|}
  \psi_R^2\big\|_\infty+
  \|H(\e)\|_\infty \|\nabla_x \psi_R\|_\infty\|e^{|\psi_R|}
  \psi_R\|_\infty\\
  \leq &\,C \e^{2k}+C\e^{2k}(\e^{\frac{1}{2}}+\|h\|_\infty)\leq C \e^{2k}(1+\|h\|_\infty). 
\end{align*}
From the crucial estimates \eqref{H-e} and \eqref{dxH-e}, we get $\|H(\e)-T_{2k-1}(H(\e))\|_{W^{1,\infty}}\leq C \e^{k+\frac{1}{2}}$.
Hence we arrive at
\begin{align*}
\|\psi_R\|_{C^{2,\tilde{\alpha}}}
\leq & C \|H(\e)(e^{\psi_R}-1-\psi_R)\|_{W^{1,\infty}}
+C\|H(\e)-T_{2k-1}(H(\e))\|_{W^{1,\infty}}+C\Big\|\int_{\mathbb{R}^3}\e^k R\, \mathrm{d}v\Big\|_{W^{1,\infty}}\\
\leq & C \e^{2k}(1+\|h\|_\infty)+C \e^{k+\frac{1}{2}}+ C\e^k (\|h\|_{\infty}+\|D_x h\|_\infty),
\end{align*} 
which implies that
\begin{align}\label{DDphiR-Linfty}
\|D_x^2 \phi_R\|_\infty\leq \|\phi_R\|_{C^{2,\tilde{\alpha}}}
\leq & C \e^{k}\|h\|_\infty+C \e^{\frac{1}{2}}+ C (\|h\|_{\infty}+\|D_x h\|_\infty)\leq C \e^{\frac{1}{2}}+ C (\|h\|_{\infty}+\|D_x h\|_\infty).
\end{align} 
Therefore, the contribution of the fifth line on the right-hand side of \eqref{dx-weighted-remainder} is
$$C\e^{\frac{3}{2}}+ C\e e^{-\frac{\nu_0 t}{2\e}}\sup_{0\leq s\leq t}\big\{e^{\frac{\nu_0 s}{2\e}}(\|h\|_{\infty}+\|D_x h\|_\infty)\big\}.$$

The $W^{1,\infty}$ estimates of the other terms on the right-hand side of \eqref{dx-weighted-remainder} are similar to \cite{Juhi} (see pp. 494--499),
where Lemma \ref{1} plays a key role. Here we omit the details for brevity.

In summary, we have
\begin{align*}
\sup_{0\leq s\leq T_0}\big\{e^{\frac{\nu_0 s}{2\e}}\|D_x h(s)\|_\infty\big\}
\leq &\,C \|(1+|v|)h(0)\|_\infty+C \|D_xh(0)\|_\infty\\
& +C\e \sup_{0\leq s\leq T_0}\big\{e^{\frac{\nu_0 s}{2\e}}\|D_v h(s)\|_\infty\big\}+\frac{C}{\e^3}\sup_{0\leq s\leq T_0}\big\{e^{\frac{\nu_0 s}{2\e}}\| h(s)\|_\infty\big\}\\
&+ (C\e+o(1))\sup_{0\leq s\leq T_0}\big\{e^{\frac{\nu_0 s}{2\e}}\|D_x h(s)\|_\infty+e^{\frac{\nu_0 s}{2\e}}\|(1+|v|) h(s)\|_\infty\big\}\\
&+ C\e^k \sup_{0\leq s\leq T_0}\big\{\big(e^{\frac{\nu_0 s}{2\e}}\|D_x h(s)\|_\infty\big)^2+\big(e^{\frac{\nu_0 s}{2\e}}\|(1+|v|) h(s)\|_\infty\big)^2\big\}\\
&+ \frac{C}{\e^4}e^{\frac{\nu_0 T_0}{2\e}}\sup_{0\leq s\leq T}\|f(s)\|_0+C e^{\frac{\nu_0 T_0}{2\e}}\e^{\frac{3}{2}}.
\end{align*}
Moreover, to close the estimates of $\|D_v h\|_\infty$, we deduce that
\begin{align}\label{order-5}
\sup_{0\leq s\leq T_0}\big\{e^{\frac{\nu_0 s}{2\e}}\|D_{x,v} h(s)\|_\infty\big\}
\leq &\,C \|(1+|v|)h(0)\|_\infty+C \|D_{x,v}h(0)\|_\infty\nonumber\\
&+ (C\e+o(1))\sup_{0\leq s\leq T_0}\big\{e^{\frac{\nu_0 s}{2\e}}\|D_{x,v} h(s)\|_\infty+e^{\frac{\nu_0 s}{2\e}}\|(1+|v|) h(s)\|_\infty\big\}\nonumber\\
&+ C\e^k \sup_{0\leq s\leq T_0}\big\{\big(e^{\frac{\nu_0 s}{2\e}}\|D_{x,v} h(s)\|_\infty\big)^2+\big(e^{\frac{\nu_0 s}{2\e}}\|(1+|v|) h(s)\|_\infty\big)^2\big\}\nonumber\\
&+\frac{C}{\e^3}\sup_{0\leq s\leq T_0}\big\{e^{\frac{\nu_0 s}{2\e}}\| h(s)\|_\infty\big\}+ \frac{C}{\e^4}e^{\frac{\nu_0 T_0}{2\e}}\sup_{0\leq s\leq T}\|f(s)\|_0+C e^{\frac{\nu_0 T_0}{2\e}}\e^{\frac{3}{2}}.
\end{align}
Letting $s=T_0$ and multiplying the result by $\e^{5}e^{-\frac{\nu_0 T_0}{2\e}}$, for sufficiently small $\e$, we deduce that
\begin{align*}
\e^5 \|D_{x,v}h(T_0)\|_\infty\leq & \, \frac{1}{2}\e^5\{\|D_{x,v}h(0)\|_\infty+\|(1+|v|)h(0)\|_\infty\}
+\frac{1}{2}\e^2 \|h(0)\|_\infty\\
&+C\e^{\frac{1}{2}}\sup_{0\leq s\leq T}\|f(s)\|_0+C \e^{\frac{7}{2}}.
\end{align*}
Thus, the proof of Lemma \ref{lemma-Dh-Linfty} is complete.
\end{proof}
Combining Lemma \ref{lemma-h-Linfty} and Lemma \ref{lemma-Dh-Linfty} up, by bootstrapping the time interval into each given time $T\leq \e^{-m}$, $0<m<\frac{1}{2}$, we deduce the following
\begin{proposition}
  Given $0<T\leq \e^{-m}$, $0<m<\frac{1}{2}$. Assume  that \eqref{assume} holds. For any sufficiently small $\e$, there exists a constant $C>0$ independent of $T$ and $\e$ such that
\begin{align}
&\sup_{0\leq s\leq T}\big\{\e^{\frac{3}{2}}\|h(s)\|_\infty\big\}\leq C \|\e^{\frac{3}{2}}h_0\|_\infty+C \sup_{0\leq s \leq T}\|f(s)\|_0+C\e^3,\label{high}\\
&\sup_{0\leq s\leq T}\big\{\e^{\frac{3}{2}}\|(1+|v|)h(s)\|_\infty+\e^5 \|D_{x,v}h(s)\|_\infty\big\}\nonumber\\
&\quad\leq C \e^{\frac{3}{2}}\|(1+|v|)h_0\|_\infty+C\e^5\| D_{x,v}h_0\|_\infty
+C \sup_{0\leq s \leq T}\|f(s)\|_0+C\e^3.\label{nabla-high}
\end{align} 
\end{proposition}

\section{$L^2$--$L^\infty$ Estimates of the Remainder $R$}\label{sec5}
Plugging \eqref{high} into \eqref{low}, we obtain
\begin{align*}
&\frac{\mathrm{d}}{\mathrm{d}t}\int_{\mathbb{R}^3}\e^{-2k}H(\e)(\psi_Re^{\psi_R}-e^{\psi_R}+1) \mathrm{d}x+\frac{1}{2}\frac{\mathrm{d}}{\mathrm{d}t}(\|\nabla_x \phi_R\|_0^2 +\|\sqrt{\theta_0}f\|_0^2)+\frac{\delta_0}{2\e}\theta_M \|(\mathbf{I-P})f\|^2_\nu\nonumber\\
\leq &\, C\Big(\frac{1}{(1+t)^{\frac{16}{15}}} +\e \mathcal{I}_1(t,\e)+\e^{\frac{1}{2}}+\mathcal{I}_2(t,\e) \e^{k-1}\Big)(\|\phi_R\|_0^2
+\|\nabla_x\phi_R\|_0^2+\|f\|_0^2+1)
\\
&\,+ C\e^{\frac{1}{2}}[\e^{\frac{3}{2}}\|h\|_\infty]( \|f\|_0+\e^{k-3} \|f\|_0^2
+\e^{k-2}\|f\|_0\|\nabla_x\phi_R\|_0 )\\
\leq &\, C\Big(\frac{1}{(1+t)^{\frac{16}{15}}} +\e \mathcal{I}_1(t,\e)+\e^{\frac{1}{2}}+\mathcal{I}_2(t,\e) \e^{k-1}\Big)(\|\phi_R\|_0^2
+\|\nabla_x\phi_R\|_0^2+\|f\|_0^2+1)
\\
&\,+ C\e^{\frac{1}{2}}\Big[\e^{\frac{3}{2}}\|h_0\|_\infty+\sup_{0\leq s\leq T}\|f(s)\|_0+\e^3\Big]
( \|f\|^2_0+\|\nabla_x\phi_R\|^2_0+1)\\
\leq &\, C\Big(\frac{1}{(1+t)^{\frac{16}{15}}} +\e \mathcal{I}_1(t,\e)+\e^{\frac{1}{2}}+\mathcal{I}_2(t,\e) \e^{k-1}+\e^{\frac{1}{2}}\Big[\e^{\frac{3}{2}}\|h_0\|_\infty+\sup_{0\leq s\leq T}\|f(s)\|_0+\e^3\Big]\Big)\\
&\,\times  (\|\phi_R\|_0^2
+\|\nabla_x\phi_R\|_0^2+\|f\|_0^2+1).
\end{align*}

It is crucial to notice that 
$$3x^2\geq xe^x-e^x+1\geq \frac{1}{12}x^2,\quad \mathrm{for}\,\,|x|\leq \ln 6.$$
According to \eqref{psiR-e}, it holds
$\|\psi_R\|_\infty\ll 1$, and there exists two constants $C_1,C_2>0$ such that
\begin{align}\label{C1C2}
C_1 \e^{2k}\|\phi_R\|^2_0
&\geq 3\int_{\mathbb{R}^3} H(\e) |\psi_R|^2 \mathrm{d}x
\geq \int_{\mathbb{R}^3}H(\e)(\psi_Re^{\psi_R}-e^{\psi_R}+1) \mathrm{d}x\nonumber\\
&\geq \frac{1}{12}\int_{\mathbb{R}^3} H(\e) |\psi_R|^2 \mathrm{d}x\geq C_2 \e^{2k}\|\phi_R\|^2_0,
\end{align}
where $H(\e)\in [e^{-C-C\e\mathcal{I}_1(t,\e)}, e^{C+C\e\mathcal{I}_1(t,\e)}]$.
Therefore, it holds
\begin{align*}
\frac{\mathrm{d}}{\mathrm{d}t}&\Big\{\int_{\mathbb{R}^3}\e^{-2k}H(\e)(\psi_Re^{\psi_R}-e^{\psi_R}+1) \mathrm{d}x+\|\nabla_x \phi_R\|_0^2 +\|\sqrt{\theta_0}f\|_0^2+1\Big\}\\
\leq & \,C\Big(\frac{1}{(1+t)^{\frac{16}{15}}} +\e \mathcal{I}_1(t,\e)+\e^{\frac{1}{2}}+\e^{k-1}\mathcal{I}_2 (t,\e) +\e^{\frac{1}{2}}\big[\e^{\frac{3}{2}}\|h_0\|_\infty+\sup_{0\leq s\leq T}\|f(s)\|_0+\e^3\big]\Big)\\
&\quad\times \Big(\int_{\mathbb{R}^3}\e^{-2k}H(\e)(\psi_Re^{\psi_R}-e^{\psi_R}+1) \mathrm{d}x
+\|\nabla_x\phi_R\|_0^2+\|\sqrt{\theta_0}f\|_0^2+1\Big).
\end{align*}

By using Gronwall's inequality, we have
\begin{align*}
&\int_{\mathbb{R}^3}\e^{-2k}H(t,x,\e) [(\psi_Re^{\psi_R}-e^{\psi_R}+1)(t,\e)] \mathrm{d}x+\|\nabla_x \phi_R\|_0^2 +\|\sqrt{\theta_0}f\|_0^2+1\\
\leq &\, \Big\{\int_{\mathbb{R}^3}\e^{-2k}H(0,x,\e) [(\psi_Re^{\psi_R}-e^{\psi_R}+1)(0,\e)] \mathrm{d}x+\|\nabla_x \phi_R(0)\|_0^2 +\|\sqrt{\theta_0}f(0)\|_0^2+1\Big\}\\
& \times \exp\Big\{\int_0^t C\Big(\frac{1}{(1+t)^{\frac{16}{15}}} +\e \mathcal{I}_1(t,\e)+\e^{\frac{1}{2}}+\e^{k-1}\mathcal{I}_2 (t,\e) +C\e^{\frac{1}{2}}\Big[\e^{\frac{3}{2}}\|h_0\|_\infty+\sup_{0\leq s\leq T}\|f(s)\|_0+\e^3\Big]\Big)\mathrm{d}s\Big\}\\
\leq &\, \Big\{\int_{\mathbb{R}^3}\e^{-2k}H(0,x,\e) [(\psi_Re^{\psi_R}-e^{\psi_R}+1)(0,\e)] \mathrm{d}x+\|\nabla_x \phi_R(0)\|_0^2 +\|\sqrt{\theta_0}f(0)\|_0^2+1\Big\}\\
& \times \exp\Big\{ C +C\e t\mathcal{I}_1(t,\e)+C\e^{\frac{1}{2}}t+C \e^{k-1}t\mathcal{I}_2(t,\e)+C\e^{\frac{1}{2}}t
\Big[\e^{\frac{3}{2}}\|h_0\|_\infty+\sup_{0\leq s\leq T}\|f(s)\|_0\Big]
\Big\}.
\end{align*}
Here, as defined in \eqref{Hee},  $H(\e)=H(t,x, \e)=\exp\{\phi_0+\cdots+\e^{2k-1}\phi_{2k-1}\}$. It is clear that for $t\leq \e^{-m}$, $0<m<\frac{1}{2}$, it hold 
\begin{align*}
\mathcal{I}_1(t,\e)=&\sum_{i=1}^{2k-1}[\e(1+t)]^{i-1}
+\Big(\sum_{i=1}^{2k-1}[\e(1+t)]^{i-1}\Big)^2
\leq  C\sum_{i=1}^{2k-1}\e^{(1-m)(i-1)}\leq C,\\
\mathcal{I}_2(t,\e)=&\sum_{2k\leq i+j\leq 4k-2}\e^{i+j-2k}(1+t)^{i+j-2}
\leq C_k \e^{2k(1-m)}\e^{2m-2k}\leq C\e^{2m-2mk}.  
\end{align*}
Furthermore,
for $0\leq t \leq \e^{-m}$, $0<m\leq \frac{1}{2}\frac{2k-3}{2k-2}\,\,(<\frac{1}{2})$, 
\begin{align}\label{time}
\mathcal{I}_1(t,\e) t \e+\mathcal{I}_2(t,\e) t \e^{k-1}\leq C \e^{\frac{1}{2}-m}\leq 1.
\end{align}
It follows from \eqref{C1C2} that
\begin{align*}
\|\phi_R\|_0^2+\|\nabla_x \phi_R\|_0^2 +\|f\|_0^2+1
\leq & \big\{\|\phi_R(0)\|_0^2+\|\nabla_x \phi_R(0)\|_0^2 +\|f(0)\|_0^2+1\big\}\\
& \times C\exp\Big\{C\e^{\frac{1}{2}-m}
\Big[\e^{\frac{3}{2}}\|h_0\|_\infty+\sup_{0\leq s\leq T}\|f(s)\|_0\Big]\Big\}.
\end{align*}
For $\varepsilon$ sufficiently small, it holds
\begin{align*}
\|\phi_R(t)\|_0^2+\|\nabla_x \phi_R(t)\|_0^2 +\|f(t)\|_0^2+1
\leq &\, C\Big\{\|\phi_R(0)\|_0^2+\|\nabla_x \phi_R(0)\|_0^2 +\|f(0)\|_0^2+1\Big\}\\
& \times \Big\{1+\e^{\frac{1}{2}-m}
\Big[\e^{\frac{3}{2}}\|h_0\|_\infty+\sup_{0\leq s\leq T}\|f(s)\|_0\Big] \Big\}.
\end{align*}
Hence, there exists a constant $C$, independent of $\e$, such that
\begin{align}\label{f-DphiR-L2}
&\sup_{0 \leq t \leq \e^{-m}}\{\|\phi_R(t)\|_0+\|\nabla \phi_R(t)\|_0+\|f(t)\|_0\} \nonumber\\
&\quad\leq   C\big(1+\|\phi_R(0)\|_0+\|\nabla \phi_R(0)\|_0+\|f(0)\|_0+\varepsilon^{\frac{3}{2}}\|h_0\|_\infty\big).
\end{align}
In conclusion, combining with \eqref{nabla-high}, we have
\begin{align*}
&\sup_{0\leq t\leq \e^{-m}}\big\{\|f(t)\|_0+\|\phi_R(t)\|_0+\|\nabla \phi_R (t)\|_0+\e^{\frac{3}{2}} \|(1+|v|)h(t)\|_\infty +\e^5 \|D_{x,v}h(t)\|_\infty\big\}\\
&\quad\leq C\big( 1+ \|f(0)\|_0+\|\phi_R(0)\|_0+\|\nabla \phi_R(0)\|_0+\e^{\frac{3}{2}} \|(1+|v|)h(0)\|_\infty +\e^5 \|D_{x,v}h(0)\|_\infty\big).
\end{align*}
This completes the proof of Theorem \ref{thm1}.
\hfill $\square$

\smallskip 
\appendix
\section{Global strong solution to the ionic Vlasov-Poisson-Boltzmann system \eqref{vpb}}\label{appendixB}
In this appendix, we recall the global existence of strong solutions to the ionic Vlasov-Poisson-Boltzmann system \eqref{vpb} in $\mathbb{R}^3$ obtained by Li-Yang-Zhong \cite{Zhong2016}.

We define two norms:
\begin{align*}
&\mathcal{X}_1^2:=\{g\in L^2: \|g\|_{\mathcal{X}_1^2}<\infty\}, \quad \|g\|_{\mathcal{X}_1^2}:=\Big(\sum_{|\alpha|+|\beta|\leq 2}\|\nu \partial_x^\alpha \partial_v^\beta g\|_{L^2}^2\Big)^{\frac{1}{2}},\\
&L^{2,1}:=L^2(\mathbb{R}_v^3; L^1(\mathbb{R}_x^3)), \quad \|g\|_{L^{2,1}}:=\Big(\int_{\mathbb{R}^3}\Big(\int_{\mathbb{R}^3} |g(x,v)|\mathrm{d}x\Big)^2 \mathrm{d}v\Big)^{\frac{1}{2}}.
\end{align*}
We also define 
$M(v)=(2\pi)^{-\frac{3}{2}}e^{-\frac{|v|^2}{2}}$ and $g=\frac{F-M}{\sqrt{M}}$ with the initial datum $g^{\rm{in}}=\frac{F^{\rm{in}}-M}{\sqrt{M}}$. Then we have
\begin{proposition}[\!\cite{Zhong2016}, Theorem 2.7]
\label{result-of-LiYangZhong2016}
Assume that $g^{\rm{in}}\in \mathcal{X}_1^2\cap L^{2,1}$ and $\|g^{\rm{in}}\|_{\mathcal{X}_1^2\cap L^{2,1}}\leq \delta_0$ for some constant $\delta_0$ small enough. Then, there exists a globally unique strong solution $F(t,x,v)$ to the system \eqref{vpb} with $g=\frac{F-M}{\sqrt{M}}$
satisfying
\begin{align*}
&\|\partial_x^\alpha \mathbf{P}g(t)\|_{L^2_{x,v}}+\|\partial_x^\alpha \phi (t)\|_{H^1_x}\leq C\delta_0 (1+t)^{-\frac{3}{4}},\\
&\|\partial_x^\alpha \{\mathbf{I}-\mathbf{P}\}g(t)\|_{L^2_{x,v}}\leq C\delta_0 (1+t)^{-\frac{5}{4}},\\
&\| \{\mathbf{I}-\mathbf{P}\}g(t)\|_{\mathcal{X}_1^2}
+\|\nabla_x \mathbf{P}g(t)\|_{L^2(\mathbb{R}_v^3; H^1(\mathbb{R}_x^3))}\leq C\delta_0 (1+t)^{-\frac{5}{4}},
\end{align*}
for $|\alpha|=0,1$ and constant $C>0$.
\end{proposition}

\section{Global smooth solution to the compressible ionic Euler-Poisson system \eqref{EPS}}\label{appendixA}
In this appendix, we recall the existence result on the compressible ionic Euler-Poisson system \eqref{EPS} in $\mathbb{R}^3$ obtained by Guo-Pausader \cite{Guo-ion-2011}.

Define two norms for $k\geq 5$ as
\begin{align*}
  \|g(x)\|_Y & :=\||\nabla|^{-1}g\|_{H^{2k+1}}+\|g\|_{W^{k+\frac{12}{5}, \frac{10}{9}}}, \\
  \|g(t,x)\|_X & :=\sup_t \big(\big\||\nabla|^{-1}(I-\Delta)^{k+\frac{1}{2}}g(t)\big\|_{L^2}
  +(1+t)^{\frac{16}{15}}
  \big\|(I-\Delta)^{\frac{k}{2}}g(t)\big\|_{L^{10}}\big).
\end{align*}
Then, we have
\begin{proposition}[\!\cite{Guo-ion-2011}, Theorem 1.1]\label{result-of-guoyan}
There exists an $\e_0>0$ such that any initial perturbation $[\rho_0^{\rm{in}}, u_0^{\rm{in}}]$ satisfying $\nabla \times u_0^{\rm{in}}=0$ and  $\|\rho_0^{\rm{in}}-1\|_Y+\|u_0^{\rm{in}}\|_Y\leq \e_0$ leads to a global solution $[\rho_0,u_0]$ of \eqref{EPS} with 
  $\|\rho_0-1\|_X+\|u_0\|_X\leq 2 \e_0$.
\end{proposition}
\begin{rem}\label{phi_0}
  In particular, both
$\rho_0-1=:\overline{\rho}$ and  $u_0$ are in $H^{2k-1}$, and decay in $W^{k, 10}$ as $t^{-\frac{16}{15}}$ for $k\geq 5$. As for $\phi_0$, note that
$$\overline{\rho}=(I-\Delta) \phi_0+\frac{\phi_0^2}{2}
+\Big[e^{\phi_0}-1-\phi_0-\frac{\phi_0^2}{2}\Big],$$
which defines an inverse operator $\overline{\rho}\mapsto \phi_0(\overline{\rho})$. Expanding it up to third order yields
$$\phi_0(\overline{\rho})=(I-\Delta)^{-1}\overline{\rho}
-\frac{1}{2}(I-\Delta)^{-1}[(I-\Delta)^{-1}\overline{\rho}]^2
+\mathcal{R}(\overline{\rho}),$$
where $\mathcal{R}$ satisfies good properties. Combining $\|\overline{\rho}\|_X\ll 1$ with $(I-\Delta)^{-1}\in \Psi^{-2}$,
we deduce that $\phi_0\in L^\infty(\mathbb{R}^1_+;H^{2k+1})$ and
\begin{align*}
\|\phi_0\|_{W^{k+2,10}}\leq C \|\bar{\rho}\|_{W^{k,10}}\leq C\e_0 (1+t)^{-\frac{16}{15}}, 
\end{align*}
which implies
that $\|\phi_0\|_{C^{k+1,\alpha}}\leq (1+t)^{-\frac{16}{15}}$ for $k\geq 5$.
\end{rem}

 \medskip
{\bf Acknowledgements}:
Li is  supported in part by  NSFC (Grant Nos. 12331007, 12071212)  and   the ``333 Project" of Jiangsu Province.
Wang is supported in part by NSFC (Grant No. 12401288)
and the Natural Science Foundation of Jiangsu Province (Grant No. BK20241259).

\end{document}